\documentclass[11pt,reqno]{amsart}

\usepackage{amssymb,amsmath,amsfonts,amsthm}
\usepackage{bm,cite,graphicx}

\newtheorem{theorem}{Theorem}[section]
\newtheorem{lemma}[theorem]{Lemma}

\newtheorem{definition}[theorem]{Definition}

\numberwithin{equation}{section}

\begin{document}

\title[inverse random source problems]{Inverse random source problems
for time-harmonic acoustic and elastic waves}

\author{Jianliang Li}
\address{School of Mathematics and Statistics, Changsha University of Science
and Technology, Changsha, 410114, P.R. China.}
\email{lijl@amss.ac.cn}

\author{Tapio Helin}
\address{Department of Mathematics and Statistics, University of Helsinki,
Helsinki, Finland.}
\email{tapio.helin@helsinki.fi}

\author{Peijun Li}
\address{Department of Mathematics, Purdue University, West Lafayette, Indiana
47907, USA.}
\email{lipeijun@math.purdue.edu}




\begin{abstract}
This paper concerns the random source problems for the time-harmonic acoustic
and elastic wave equations in two and three dimensions. The goal is to
determine the compactly supported external force from the radiated wave field
measured in a domain away from the source region. The source is assumed to be a
microlocally isotropic generalized Gaussian random function such that its
covariance operator is a classical pseudo-differential operator. Given such a
distributional source, the direct problem is shown to have a unique solution by
using an integral equation approach and the Sobolev embedding theorem. For the
inverse problem, we demonstrate that the amplitude of the scattering field
averaged over the frequency band, obtained from a single realization of the
random source, determines uniquely the principle symbol of the covariance
operator. The analysis employs asymptotic expansions of the Green functions and
microlocal analysis of the Fourier integral operators associated with the
Helmholtz and Navier equations. 
\end{abstract}

\maketitle

\section{Introduction}\label{sec1}

The inverse source scattering in waves, as an important and active research
subject in inverse scattering theory, are to determine the unknown sources that
generate prescribed radiated wave patterns \cite{I}. It has been considered as a
basic mathematical tool for the solution of many medical imaging modalities
\cite{SA}, such as magnetoencephalography (MEG), electroencephalography (EEG),
electroneurography (ENG). These imaging modalities are non-invasive
neurophysiological techniques that measure the electric or magnetic fields
generated by neuronal activity of the brain. The spatial distributions of the
measured fields are analyzed to localize the sources of the activity within the
brain to provide information about both the structure and function of the brain
\cite{ABF, FKM, NOH}. The inverse source scattering problem has also attracted
much research in the community of antenna design and synthesis \cite{MD}. A
variety of antenna-embedding materials or substrates, including non-magnetic
dielectrics, magneto-dielectrics, and double negative meta-materials are of
great interest. 

Driven by these significant applications, the inverse source scattering problems
have continuously received much attention and have been extensively studied by
many researchers. There are a lot of available mathematical and numerical
results, especially for the acoustic waves or the Helmholtz equation
\cite{ACTV, BLRX, BN, DML, EV, ZG}. In general, the inverse source problem does
not have a unique solution due to the existence of non-radiating sources
\cite{BC, BC-JMP, DS, HKP}. Some addition constraint or information is needed in
order to obtain a unique solution, such as to seek the minimum energy solution
which represents the pseudo-inverse solution for the inverse problem.
For electromagnetic waves, Ammari et al. showed uniqueness and presented an
inversion scheme in \cite{ABF} to reconstruct dipole sources based on a
low-frequency asymptotic analysis of the time-harmonic Maxwell equations. In
\cite{AM}, Albanese and Monk discussed uniqueness and non-uniqueness of the
inverse source problems for Maxwell's equations. Computationally, a more serious
issue is the lack of stability, i.e., a small variation in the measured data may
lead to a huge error in the reconstruction. Recently, it has been realized that
the use of multi-frequency data can overcome the difficulties of non-uniqueness
and instability which are encountered at a single frequency. In \cite{BLT}, Bao
et al. initialized the mathematical study on the stability of the inverse source
problem for the Helmholtz equation by using multi-frequency data. Since then,
the increasing stability has become an interesting research topic in the study
of inverse source problems \cite{BLZ, CIL, LY}. We refer to \cite{BLLT} for a
topic review on solving general inverse scattering problems with
multi-frequencies. 

Recently, the elastic wave scattering problems have received ever increasing
attention for their important applications in may scientific areas such as
geophysics and seismology \cite{ABGK, PC, LL, MP, JT}. However, the inverse
source problem is much less studied for the elastic waves. The elastic wave
equation is challenging due to the coexistence of compressional and shear waves
that have different wavenumbers. Consequently, the Green tensor of the Navier
equation has a more complicated expression than the Green function of the
Helmholtz equation does. A more sophisticated analysis is required.

In many applications the source and hence the radiating field may not be
deterministic but rather are modeled by random processes \cite{BAZC}. Therefore,
their governing equations are stochastic differential equations. Although the
deterministic counterparts have been well studied, little is known for the
stochastic inverse problems due to randomness and uncertainties. A uniqueness
result may be found in \cite{D} for an inverse random source problem. It
was shown that the auto-correlation function of the random source was uniquely
determined by the auto-correlation function of the radiated field. Recently,
effective mathematical models and efficient computational methods have been
developed in \cite{BCL16, BCL18, BCLZ, BX, PL, LCL} for inverse random source
scattering problems, where the stochastic wave equations are considered and the
random sources are assumed to be driven by additive white noise. The inverse
problems are formulated to determine the statistical properties, such as the
mean and variance, of the random source from the boundary measurement of the
wave field at multiple frequencies. The method requires to know the expectation
of the scattering data. By the strong law of large numbers, the expectation has
to be approximated by taking fairly large number of realizations of the
measurement. We refer to \cite{KS} for statistical inversion theory on general
random inverse problems. 

In this paper, we consider a new model for the random source. A unified
theory is developed on both of the direct and inverse scattering problems for
the time-harmonic acoustic and elastic wave equations. The source is assumed to
be a generalized Gaussian random function which is supported in a bounded domain
$D\subset\mathbb R^d$, $d=2$ or $3$. In addition, we assume that the covariance
of the random source is described by a pseudo-differential operator with the
principle symbol given by $\phi(x)|\xi|^{-m}, m\in [d,d+\frac{1}{2})$, where
$\phi$ is a smooth non-negative function supported on $D$ and is called the
micro-correlation strength of the source. This large class of random fields
includes stochastic processed like the fractional Brownian motion and Markov
field \cite{LPS}. In fact, when $m\in [d,d+\frac{1}{2})$, we can only ensure
that the source belongs to a Sobolev space with negative smoothness index almost
surely. Hence, the direct scattering problem requires a careful analysis since
the source is a so-called rough field. In this work, we establish the
well-posedness of the direct scattering problems for both wave equations with
such rough sources. The inverse scattering problem aims at reconstructing the
micro-correlation strength of the source $\phi$ from the scattered field
measured in a bounded domain $U$ where $\overline U\cap \overline D=\emptyset$.
For a single realization of the random source, we measure the amplitude of the
scattering field averaged over the frequency band in a bounded and simply
connected domain. Combining harmonic and microlocal analysis, we show that: for
acoustic waves, the micro-correlation strength function $\phi$ can be recovered
by these measurements; For elastic waves, note that the source is a vector, if
the components of the random source are independent and the principle symbol of
the pseudo-differential operator of each component coincides, thus, the
micro-correlation strength function $\phi$ can be determined uniquely by these
measurements. 

This work is motivated by \cite{LPS, CHL}, where an inverse problem
was considered for the two-dimensional random Schr\"{o}dinger equation. The
potential function in the Schr\"{o}dinger equation was assumed to be a Gaussian
random function with a pseudo-differential operator describing its covariance.
It was shown that the principle symbol of the covariance operator can be
determined uniquely by the backscattered field, generated by a single
realization of the random potential and a point source as the incident field. A
closely related problem can be found in \cite{HLP}. The authors considered the
uniqueness for an inverse acoustic scattering problem in a half-space with an
impedance boundary condition, where the impedance coefficient was assumed to a
Gaussian random function whose covariance operator is a pseudo-differential
operator. 

The paper is organized as follows. In Section 2, we introduce some commonly
used Sobolev spaces, give a precise mathematical description of the
generalized Gaussian random function, and present several lemmas on rough
fields and random variables. Section 3 is devoted to the study of the acoustic
wave equation in the two- and three-dimensional cases. The well-posedness of
the direct problems are examined. The uniqueness of the inverse problem are
achieved. Section 4 addresses the two- and three-dimensional elastic wave
equations. Analogous results are obtained. The direct problem is shown to
have a unique solution and the inverse problem is proved to have the uniqueness
to recover the principle symbol of the covariance operator for the random
source. This paper is concluded with some general remarks in Section 5.

\setcounter{equation}{0}

\section{Preliminaries}

In this section, we introduce some necessary notation such as 
Sobolev spaces and generalized Gaussian random functions which are used
throughout the paper. 

\subsection{Sobolev spaces}

Let $\mathbb R^d$ be the $d$-dimensional space, where $d=2 \text{ or }
3$. Denote by $C_0^{\infty}(\mathbb R^d)$ the set of smooth functions
with compact support and by  $\mathcal{D'}(\mathbb R^d)$ the set of
generalized (distributional) functions. For $1<p<\infty, s\in\mathbb R$, 
the Sobolev space $H^{s,p}(\mathbb R^d)$ is defined by
\[
H^{s,p}(\mathbb R^d)=\{h=(I-\Delta)^{-\frac{s}{2}}g : g\in L^p(\mathbb R^d)\},
\]
which has the norm
\[
\|h\|_{H^{s,p}(\mathbb R^d)}=\|(I-\Delta)^{\frac{s}{2}}h\|_{L^p(\mathbb R^d)}.
\]
With the definition of Sobolev spaces in the whole space, we can define the
Sobolev spaces $H^{s,p}(V)$ for any Lipschitz domain $V\subset \mathbb R^d$
as the restrictions to $V$ of the elements in $H^{s,p}(\mathbb R^d)$. The norm
is defined by
\[
\|h\|_{H^{s,p}(V)}=\inf\{\|g\|_{H^{s,p}(\mathbb R^d)} : g|_V=h\}.
\]
According to \cite{JK}, for $s\in\mathbb R$ and $1<p<\infty$, we can define
$H^{s,p}_0(V)$ as the space of all distributions $h\in H^{s,p}(\mathbb R^d)$
such that ${\rm supp} h\subset\overline{V}$ and the the norm is defined by
\[
\|h\|_{H^{s,p}_0(V)}=\|h\|_{H^{s,p}(\mathbb R^d)}.
\]
It is known that $C_0^{\infty}(V)$ is dense in $H_0^{s,p}(V)$ for any
$1<p<\infty, s\in\mathbb R$; $C_0^{\infty}(V)$ is dense in $H^{s,p}(V)$ for any
$1<p<\infty, s\leq 0$; $C^{\infty}(\overline{V})$ is dense in $H^{s,p}(V)$ for
any  $1<p<\infty, s\in\mathbb R$. Additionally, by \cite[Propositions 2.4 and
2.9]{JK}, for any $s\in\mathbb R$, $p,q\in (1,\infty)$ satisfying
$\frac{1}{p}+\frac{1}{q}=1$, we have
\[
H^{-s,q}_0(V)=(H^{s,p}(V))'\quad {\rm and}\quad H^{-s,q}(V)=(H_0^{s,p}(V))',
\]
where the prime denotes the dual space.

\subsection{Generalized Gaussian random functions}

In this subsection, we provide a precise mathematical description of the
generalized Gaussian random function.  Let $(\Omega, \mathcal{F}, \mathcal{P})$
be a complete probability space. The function $q$ is said to be a generalized
Gaussian random function if $q:\Omega \rightarrow \mathcal{D'}(\mathbb R^d)$ is
a measurable map such that for every $\hat{\omega}\in \Omega$, the mapping
$\hat{\omega}\in \Omega\longmapsto \langle q(\hat{\omega}), \psi\rangle$ is a
Gaussian random variable for all $\psi\in C_0^{\infty}(\mathbb R^d)$. The
expectation and the covariance of the generalized Gaussian random function $q$
can be defined by
\begin{eqnarray*}
\mathbb Eq: \psi\in C_0^{\infty}(\mathbb R^d)&\longmapsto& \mathbb E\langle
q,\psi\rangle\in\mathbb R,\\
{\rm Cov} q:(\psi_1,\psi_2)\in C_0^{\infty}(\mathbb R^d)^2&\longmapsto& {\rm
Cov}(\langle q,\psi_1\rangle,\langle q,\psi_2\rangle)\in\mathbb R,
\end{eqnarray*}
where $\mathbb E\langle q,\psi\rangle$ denotes the expectation of $\langle
q,\psi\rangle$ and 
\begin{eqnarray*}
{\rm Cov}(\langle q,\psi_1\rangle,\langle q,\psi_2\rangle)=\mathbb E((\langle
q,\psi_1\rangle-\mathbb E\langle q,\psi_1\rangle)(\langle
q,\psi_2\rangle-\mathbb E\langle q,\psi_2\rangle))
\end{eqnarray*}
denotes the covariance of $\langle q, \psi_1\rangle$ and $\langle
q,\psi_2\rangle$. The covariance operator $C_{q}: C_0^{\infty}(\mathbb
R^d)\rightarrow \mathcal{D'}(\mathbb R^d)$ is defined by 
\begin{eqnarray}\label{a1}
\langle C_{q}\psi_1,\psi_2\rangle={\rm Cov}(\langle q,\psi_1\rangle,\langle
q,\psi_2\rangle)=\mathbb E(\langle q-\mathbb Eq,\psi_1\rangle \langle q-\mathbb
Eq,\psi_2\rangle).
\end{eqnarray}
Let $k_{q}(x,y)$ be the Schwartz kernel of the covariance operator $C_{q}$. We
also call $k_{q}(x,y)$ the covariance function of $q$. Thus, (\ref{a1}) means
that 
\begin{eqnarray}\label{a2}
k_{q}(x,y)=\mathbb E((q(x)-\mathbb E q(x))(q(y)-\mathbb E q(y)))
\end{eqnarray}
in the sense of generalized functions. 

In this paper, we assume that each component of the external source is  a
generalized, microlocally isotropic Gaussian random function. For this end, let
$D\subset\mathbb R^d$ be a bounded and simply connected domain. We
introduce the following definition. 

\begin{definition}
A generalized Gaussian random function $q$ on $\mathbb R^d$ is called
microlocally isotropic of order $m$ in D, if the realizations of $q$ are almost
surely supported in the domain $D$ and its covariance operator $C_q$ is a
classical pseudo-differential operator having the principal symbol
$\phi(x)|\xi|^{-m}$, where $\phi\in C_0^{\infty}(\mathbb R^d)$, ${\rm supp}
\phi\subset D$ and $\phi(x)\geq 0$ for all $x\in \mathbb R^d$.
\end{definition}

In particular, we are interested in the case $m\in[d,d+\frac{1}{2})$, which
corresponds to rough fields. Now we introduce three lemmas and give an
assumption which will be used in subsequent analysis.  

\begin{lemma}\label{lemma1a}
Let $f$ be a generalized and microlocally isotropic Gaussian random function of
order $m$ in $D$. If $m=d$, then $f \in H^{-\varepsilon, p}(D)$ almost
surely for all $\varepsilon>0, 1<p<\infty$.  If $m\in(d,d+\frac{1}{2})$,
then $f\in C^{\alpha}(D)$ almost surely for all $\alpha\in(0,
\frac{m-d}{2})$.
\end{lemma}
\begin{lemma}\label{lemma2a}
Let $X$ and $Y$ be two zero-mean random variables such that the pair $(X,Y)$ is
a Gaussian random vector. Then we have
\begin{eqnarray*}
\mathbb E((X^2-\mathbb E X^2)(Y^2-\mathbb E Y^2))=2(\mathbb EXY)^2.
\end{eqnarray*}
\end{lemma}
\begin{lemma}\label{lemma3a}
Let $X_t$, $t\geq 0$ be a real valued stochastic process with a continuous path
of zero mean, i.e., $\mathbb E X_t = 0$. Assume that for some constants $c>0$
and
$\beta>0$ such that the condition
\begin{eqnarray*}
|\mathbb E(X_tX_{t+r})|\leq c(1+r)^{-\beta}
\end{eqnarray*}
holds for all $t,r\geq 0$. Then 
\begin{eqnarray*}
\lim_{Q\rightarrow\infty}\frac{1}{Q}\int_1^QX_t{\rm d}t=0
\end{eqnarray*}
almost surely.
\end{lemma}

Lemma \ref{lemma1a} is a direct consequence of Theorem 2 in \cite{LPS}. 
Lemma \ref{lemma2a} is shown in \cite{CHL} as Lemma 4.2. The ergodic
result of Lemma \ref{lemma3a} is an immediate corollary of \cite[p. 94]{CL}.
To establish the main results, we need the following assumption. 

{\bf Assumption A}: The external source $f$ is assumed to have a compact
support $D\subset \mathbb R^d$. Let $U\subset\mathbb R^d\setminus \overline{D}$
be the measurement domain of the wave field. We assume that $D$ and $U$
are two bounded and simply connected domains and there is a positive distance
between $D$ and $U$.

\section{Acoustic waves}

This section addresses the direct and inverse source scattering problems for
the Helmholtz equation in two- and three-dimensional space. The external source
is assumed to be a generalized Gaussian random function whose covariance
operator is a classical pseudo-differential operator. The direct problem is
shown to have a unique solution. For the inverse problem, we show that the
principle symbol of the covariance operator can be determined uniquely by the
scattered field obtained from a single realization of the random source. 

\subsection{The direct scattering problem}

Consider the Helmholtz equation in a homogeneous medium
\begin{eqnarray}\label{r1}
\Delta u+\kappa^2u=f\quad{\rm in} ~ \mathbb R^d,
\end{eqnarray}
where $\kappa>0$ is the wavenumber, $u$ is the wave field, and $f$ is a
generalized Gaussian random function. Note that $u$ is a random field since $f$
is a random function. To ensure the uniqueness of the solution for \eqref{r1},
the usual Sommerfeld radiation condition is imposed
\begin{eqnarray}\label{r2}
\lim_{r\rightarrow\infty}r^{\frac{d-1}{2}}\left(\partial_r u-{\rm
i}\kappa u\right)=0,\quad r=|x|,
\end{eqnarray} 
uniformly for all directions $\hat{x}=x/|x|$. In addition, the external source
function $f$ satisfies the following assumption. 

{\bf Assumption B:} The generalized Gaussian random field $f$ is microlocally
isotropic of order $m$ in $D$, where $m\in[d,d+\frac{1}{2})$. The principle
symbol of its covariance operator $C_f$ is $\phi(x)|\xi|^{-m}$ with $\phi\in
C_0^{\infty}(D)$ and $\phi\geq 0$. Moreover, the mean value of $f$ is zero,
i.e., $\mathbb E(f)=0$.

By Lemma \ref{lemma1a}, the random source $f(\hat{\omega})$ belongs with
probability one to the Sobolev space $H^{-\varepsilon,p}(D)$ for
all $\varepsilon>0, 1<p<\infty$. Hence it suffices to show that the direct
scattering problem is well-posed when $f$ is a deterministic non-smooth
function in $H^{-\varepsilon, p}(D)$.

First, we show some regularity results of the fundamental solution. These
results play an important role in the proof of the well-posedness. Let
$\Phi_d(x,y,\kappa)$ be the fundamental solution for the
two- and three-dimensional Helmholtz equation. Explicitly, we have 
\begin{eqnarray}\label{r3}
\Phi_2(x,y,\kappa)=\frac{{\rm i}}{4}H_0^{(1)}(\kappa|x-y|), \quad 
\Phi_3(x,y,\kappa)=\frac{1}{4\pi}\frac{e^{{\rm i}\kappa |x-y|}}{|x-y|},
\end{eqnarray}
where $H_0^{(1)}$ is the Hankel function of the first kind with order zero.
We shall study the asymptotic properties of the fundamental
solutions and their derivatives when $x$ is close to $y$. For the
two-dimensional case, we recall that 
\begin{eqnarray}\label{r4}
H_n^{(1)}(t)=J_n(t)+{\rm i}Y_n(t),
\end{eqnarray}
where $J_n$ and $Y_n$ are the Bessel functions of the first and second kind
with order $n$, respectively. They admit the following expansions
\begin{align}\label{r5}
J_n(t)&=\sum_{p=0}^{\infty}\frac{(-1)^p}{p!(n+p)!}\left(\frac{t}{2}\right)^{
n+2p},\\\notag
Y_n(t)&=\frac{2}{\pi}\left\{\ln\frac{t}{2}+\gamma\right\}J_n(t)-\frac{1}{\pi}
\sum_{p=0}^{n-1}\frac{(n-1-p)!}{p!}\left(\frac{2}{t}\right)^{n-2p}\\\label{r6}
&\qquad-\frac{1}{\pi}\sum_{p=0}^{\infty}\frac{(-1)^p}{p!(n+p)!}\left(\frac{t}{2}
\right)^{n+2p}\{\psi(p+n)+\psi(p)\},
\end{align}
where
$\gamma:=\lim_{p\rightarrow\infty}\left\{\sum_{j=1}^p j^{-1}
-\ln p\right\}$ denotes the Euler constant,  $\psi(0)=0$,
$\psi(p)=\sum_{j=1}^p j^{-1}$, and the finite sum in
(\ref{r6}) is set to be zero for $n=0$. Using (\ref{r4})--(\ref{r6}), we may
verify that 
\begin{align}\label{r7}
H_0^{(1)}(t)&=\frac{2{\rm i}}{\pi}\ln\frac{t}{2}+(1+\frac{2{\rm
i}}{\pi}\gamma)+O(t^2\ln\frac{t}{2}),\\\label{r8}
H_1^{(1)}(t)&=-\frac{2{\rm i}}{\pi}\frac{1}{t}+\frac{{\rm
i}}{\pi}t\ln\frac{t}{2}+(1+\frac{2{\rm i}}{\pi}\gamma-\frac{\rm
i}{\pi})\frac{t}{2}+O(t^3\ln\frac{t}{2}).
\end{align}
Using the recurrence relations for the Hankel function of the first
kind (see \cite[Eq. (5.6.3)]{NN})
\begin{eqnarray}\label{r9}
\frac{\rm d}{{\rm d}t}[t^{-n}H_n^{(1)}(t)]=-t^{-n}H_{n+1}^{(1)}(t), 
\end{eqnarray}
we may show from (\ref{r7})--(\ref{r8}) that 
\begin{align}
\Phi_2(x,y,\kappa)&=\frac{\rm i}{4}H_0^{(1)}(\kappa|x-y|)\notag\\\label{r10}
&=-\frac{1}{2\pi}\ln\frac{|x-y|}{2}+(\frac{\rm
i}{4}-\frac{\gamma}{2\pi})+O(|x-y|^2\ln\frac{|x-y|}{2}),\\\notag
\partial_{y_i}\Phi_2(x,y,\kappa) &=-\frac{\kappa {\rm
i}}{4}(y_i-x_i)\frac{H_1^{(1)}(\kappa|x-y|)}{|x-y|}\\\label{r11}
&=-\frac{1}{2\pi}\frac{y_i-x_i}{|x-y|^2}+O((y_i-x_i)\ln\frac{|x-y|}{2}).
\end{align}

For the three-dimensional case, a simple calculation yields that 
\begin{align}\label{r12}
\Phi_3(x,y,\kappa)&=\frac{e^{{\rm i}\kappa |x-y|}}{4\pi |x-y|},\\\label{r13}
\partial_{y_i} \Phi_3(x,y,\kappa)&=\frac{(y_i-x_i)}{4\pi|x-y|^3}e^{{\rm
i}\kappa |x-y|}({\rm i}\kappa  |x-y|-1).
\end{align}

\begin{lemma}\label{lemma_z}
Given any $x\in \mathbb R^d$, we have $\Phi_2(x,\cdot,\kappa)\in L^2_{\rm
loc}(\mathbb R^2)\cap H^{1,p}_{\rm loc}(\mathbb R^2)$ for any $p\in (1, 2)$ and
$\Phi_3(x,\cdot,\kappa)\in L^2_{\rm loc}(\mathbb R^3)\cap H^{1,p}_{\rm
loc}(\mathbb R^3)$ for any $p\in (1,\frac{3}{2})$.
\end{lemma}

\begin{proof}
For any fixed $x\in\mathbb R^d$, let $V\subset\mathbb R^d$ be a bounded domain
containing $x$. Denote $\rho:=\sup\limits_{y\in V}|x-y|$, then we have
$V\subset B_{\rho}(x)$. 

For $d=2$,  by (\ref{r10}) and (\ref{r11}), it suffices to show that 
\begin{eqnarray*}
\ln\frac{|x-y|}{2}\in L^2(V),\quad \frac{y_i-x_i}{|x-y|^2}\in
L^p(V), \quad \forall\, p\in(1,2).
\end{eqnarray*}
A direct calculation yields 
\begin{eqnarray*}
\int_V\left|\ln\frac{|x-y|}{2}\right|^2{\rm d}y\leq\int_{B_{\rho}(x)}
\left|\ln\frac {|x-y|}{2}\right|^2{\rm
d}y\lesssim\int_0^{\rho}r\left|\ln\frac{r}{2} \right|^2{\rm d}r<\infty
\end{eqnarray*}
and
\begin{eqnarray*}
\int_V\left|\frac{y_i-x_i}{|x-y|^2}\right|^p{\rm
d}y\leq\int_{B_{\rho}(x)}\frac{1} {
|x-y|^p}{\rm d}y\lesssim\int_0^{\rho}r^{1-p}{\rm d}r<\infty,\quad\forall\, p\in
(1,2).
\end{eqnarray*}
Hereafter, the notation $a\lesssim b$ means $a\leq Cb$, where $C>0$ is a generic
constant which may change step by step in the proofs.
Thus, we conclude that $\Phi_2(x,\cdot,\kappa)\in L^2_{\rm loc}(\mathbb R^2)\cap
H^{1,p}_{\rm loc}(\mathbb R^2)$ for any $p\in (1,2)$. 

For $d=3$, from (\ref{r12}) and (\ref{r13}), it suffices
to prove that 
\begin{eqnarray*}
\frac{e^{{\rm i}\kappa |x-y|}}{|x-y|}\in L^2(V),\quad e^{{\rm
i}\kappa |x-y|}\frac{y_i-x_i}{|x-y|^3}\in L^p(V)\quad\forall\, p\in
(1,\frac{3}{2}).
\end{eqnarray*}
Similarly, we may have from a simple calculation that 
\begin{eqnarray*}
\int_V\left|\frac{e^{{\rm i}\kappa
|x-y|}}{|x-y|}\right|^2{\rm d}y\leq\int_{B_{\rho}(x)}\frac{1}{|x-y|^2}
{\rm d}y\lesssim\int_0^{\rho}1 {\rm d}r<\infty
\end{eqnarray*}
and
\begin{eqnarray*}
\int_V\left|e^{{\rm i}\kappa
|x-y|}\frac{y_i-x_i}{|x-y|^3}\right|^p{\rm d}y\leq\int_{B_{\rho}(x)}\frac{1}{
|x-y|^{2p}}{\rm d}y\lesssim\int_0^{\rho}r^{2-2p}{\rm
d}r<\infty\quad\forall\,p\in (1,\frac{3}{2}),
\end{eqnarray*}
which show that $\Phi_3(x,\cdot,\kappa)\in L^2_{\rm loc}(\mathbb
R^3)\cap H^{1,p}_{\rm loc}(\mathbb R^3)$ for any $p\in(1,\frac{3}{2})$.
\end{proof}

 Let $V$ and $G$ be any two bounded domains in $\mathbb R^d$. By Lemma
\ref{lemma_z} and the Sobolev embedding theorem, we obtain that
$\Phi_d(x,\cdot,\kappa)\in H^s(V)$ where $s\in (0,1)$ for $d=2$ and $s\in
(0,\frac{1}{2})$ for $d=3$. Hence, given $g\in H_0^{-s}(V)$, we can define the
operator $H_{\kappa}$ in the dual sense by
\begin{eqnarray*}
H_{\kappa}g(x)=\int_{V}\Phi_d(x,y,\kappa)g(y){\rm d}y,\quad x\in G.
\end{eqnarray*}

Following the similar arguments in \cite[Theorem 8.2]{CK}, we
may show the following regularity of the operator $H_\kappa$. The proof is
omitted here for brevity. 

\begin{lemma}\label{lemma_v}
The operator $H_{\kappa}:H_0^{-s}(V)\rightarrow H^{s}(G)$ is bounded for $s\in
(0,1)$ in two dimensions or for $s\in (0,\frac{1}{2})$ in three dimensions. 
\end{lemma}

\begin{theorem}
For some fixed $s\in (0, 1-\frac{d}{6})$, assume $1<p<\frac{2d}{d+2(1-s)}$ and
$\frac{1}{p}+\frac{1}{p'}=1$, then the scattering problem (\ref{r1})--(\ref{r2})
with the source $f\in H_0^{-1,p'}(D)$ attains a unique solution $u\in
H^{1,p}_{\rm loc}(\mathbb R^d)$, which can be represented by
\begin{eqnarray}\label{r14}
u(x,\kappa)=-\int_{D}\Phi_d(x,y,\kappa)f(y){\rm d}y.
\end{eqnarray}
\end{theorem}

\begin{proof}
It is clear that  the scattering problem (\ref{r1})--(\ref{r2}) with $f=0$ only
has the zero solution. Hence the uniqueness follows. Now we focus on the
existence. Since $s\in
(0, 1-\frac{d}{6})$, a simple calculation shows that
$1<\frac{2d}{d+2(1-s)}<\frac{3}{2}$. By Lemma \ref{lemma_z},
we obtain that $\Phi_d(x,\cdot,\kappa)\in H_{\rm loc}^{1, p}(\mathbb R^d)$.
Since $\Delta u+\kappa^2 u=f\in H_0^{-1,p'}(D)$, we have in the sense
of distribution that 
 \begin{eqnarray}\label{r15}
 \int_{B_r}(\Delta u(y)+\kappa^2
u(y))\Phi_d(x,y,\kappa){\rm d}y=\int_{B_r}\Phi_d(x,y,\kappa)f(y){\rm d}y.
 \end{eqnarray}
Here $B_r=\{y\in \mathbb R^d: |y|\leq r\}$, where $r>0$ is sufficiently large
such that $\overline D\subset B_r$. Denote by $S_{\rm A}$ the operator which
acts on $u$ on the left-hand side of (\ref{r15}).  For $\varphi\in
C^{\infty}(\mathbb R^d)$, we have
 \begin{eqnarray*}
&& (S_{\rm A}\varphi)(x):= \int_{B_r}(\Delta \varphi(y)+\kappa^2
\varphi(y))\Phi_d(x,y,\kappa){\rm d}y\\
 &&= \int_{B_r\setminus B_{\delta}(x)}(\Delta \varphi(y)+\kappa^2
\varphi(y))\Phi_d(x,y,\kappa){\rm d}y+ \int_{B_{\delta}(x)}(\Delta
\varphi(y)+\kappa^2 \varphi(y))\Phi_d(x,y,\kappa){\rm d}y\\
 &&= \int_{B_r\setminus B_{\delta}(x)}(\Delta
\varphi(y)\Phi_d(x,y,\kappa)-\varphi(y)\Delta\Phi_d(x,y,\kappa)){\rm d}y+
\int_{B_{\delta}(x)}(\Delta \varphi(y)+\kappa^2
\varphi(y))\Phi_d(x,y,\kappa){\rm d}y\\
 &&=\int_{B_{\delta}(x)}(\Delta \varphi(y)+\kappa^2
\varphi(y))\Phi_d(x,y,\kappa){\rm d}y+\int_{\partial B_r}\left(\frac{\partial
\varphi(y)}{\partial \nu(y)}\Phi_d(x,y,\kappa)-\varphi(y)\frac{\partial
\Phi_d(x,y,\kappa)}{\partial \nu(y)}\right){\rm d}s(y)\\
 &&\quad+\int_{\partial B_{\delta}(x)}\left(\frac{\partial \varphi(y)}{\partial
\nu(y)}\Phi_d(x,y,\kappa)-\varphi(y)\frac{\partial \Phi_d(x,y,\kappa)}{\partial
\nu(y)}\right){\rm d}s(y),
\end{eqnarray*}
where $\delta>0$ is a sufficiently small number, and $\nu(y)$ denotes the unit
normal which directs to the exterior of $B_r$ for $y\in \partial B_r$ and
directs to the interior of $B_{\delta}(x)$ for $y\in\partial B_{\delta}(x)$.
Using the mean value theorem, we get
 \begin{eqnarray*}
 \lim_{\delta\rightarrow 0}\int_{\partial B_{\delta}(x)}\left(\frac{\partial
\varphi(y)}{\partial \nu(y)}\Phi_d(x,y,\kappa)-\varphi(y)\frac{\partial
\Phi_d(x,y,\kappa)}{\partial \nu(y)}\right){\rm d}s(y)=-\varphi(x)
 \end{eqnarray*}
 and 
 \begin{eqnarray*}
 \lim_{\delta\rightarrow 0}\int_{B_{\delta}(x)}(\Delta \varphi(y)+\kappa^2
\varphi(y))\Phi_d(x,y,\kappa){\rm d}y=0.
 \end{eqnarray*}
Combining the above equations gives that 
 \begin{eqnarray*}
 (S_{\rm A}\varphi)(x)=-\varphi(x)+\int_{\partial B_r}\left(\frac{\partial
\varphi(y)}{\partial \nu(y)}\Phi_d(x,y,\kappa)-\varphi(y)\frac{\partial
\Phi_d(x,y,\kappa)}{\partial \nu(y)}\right){\rm d}s(y),
 \end{eqnarray*}
which implies
  \begin{eqnarray*}
 (S_{\rm A} u)(x)=-u(x)+\int_{\partial B_r}\left(\frac{\partial u(y)}{\partial
\nu(y)}\Phi_d(x,y,\kappa)-u(y)\frac{\partial \Phi_d(x,y,\kappa)}{\partial
\nu(y)}\right){\rm d}s(y).
 \end{eqnarray*}
Since both $u$ and $\Phi_d$ satisfy the Sommerfeld radiation
condition, we have 
\begin{eqnarray*}
\lim_{r\rightarrow\infty}\int_{\partial B_r}\left(\frac{\partial u(y)}{\partial
\nu(y)}\Phi_d(x,y,\kappa)-u(y)\frac{\partial \Phi_d(x,y,\kappa)}{\partial
\nu(y)}\right){\rm d}s(y)=0.
\end{eqnarray*}
Therefore
\begin{eqnarray*}
u(x,\kappa)=-\int_{D}\Phi_d(x,y,\kappa)f(y){\rm d}y=-H_{\kappa}f(x).
\end{eqnarray*}

Next is to show that $u\in H^{1,p}_{\rm loc}(\mathbb
R^d)$. From Lemma \ref{lemma_v}, we have that the operator
$H_{\kappa}:H_0^{-s}(D)\rightarrow H_{\rm loc}^{s}(\mathbb R^d)$ for
$s\in(0,1-\frac{d}{6})$ is bounded. The assumption $1<p<\frac{2d}{d+2(1-s)}$
implies that $\frac{1}{2}+\frac{1-s}{d}<\frac{1}{p}<1$ which yields
$\frac{1}{2}-\frac{s}{d}<\frac{1}{p}-\frac{1}{d}$. Thus, the Sobolev embedding
theorem implies that $H^{s}(D)$ is embedded into $H^{1, p}(D)$ and $H_0^{-1,
p'}(D)$ is embedded into $H_0^{-s}(D)$. Thus, the operator $H_{\kappa}:H_0^{-1,
p'}(D)\rightarrow H^{1,p}_{\rm loc}(\mathbb R^d)$ is bounded, which completes
the proof.
\end{proof}

\subsection{The two-dimensional case}

First we discuss the two-dimensional case and show that the function $\phi$ in
the principle symbol can be uniquely determined by the scattered field obtained
from a single realization of the random source $f$. Let us begin with the
asymptotic of the Hankel function $H_n^{(1)}$ with a large argument. By
\cite[Eqs. (9.2.7)--(9.2.10)]{AS} and \cite[Eqs.(5.11.4)] {NN}, we have:
\begin{eqnarray}\label{b5}
H_n^{(1)}(z)=\sqrt{\frac{1}{ z}}e^{{\rm i}(
z-(\frac{n}{2}+\frac{1}{4})\pi)}\left(\sum_{j=0}^{N}a_j^{(n)}z^{-j}+O(|z|^{-N-1})\right),\quad 
{\rm for}\; |{\rm arg}\; z|\leq \pi-\delta,
\end{eqnarray}
for large $|z|$, where $\delta$ is an arbitrarily small positive number and the
coefficients $a_j^{(n)}=(-2{\rm i})^j\sqrt{\frac{2}{\pi}}(n,j)$ with
\[
\quad (n,j)=\frac{(4n^2-1)(4n^2-3^2)\cdots(4n^2-(2j-1)^2)}{2^{2j}j!}\quad{\rm and}\quad(n,0)=1.
\]
Using the first $N$ terms in the asymptotic of $H_n^{(1)}(\kappa |z|)$, we
define 
\begin{eqnarray}\label{b7}
H_{n,N}^{(1)}(\kappa |z|):=\sqrt{\frac{1}{\kappa |z|}}e^{{\rm i}(\kappa
|z|-(\frac{n}{2}+\frac{1}{4})\pi)}\sum_{j=0}^N a_j^{(n)}\left(\frac{1}{\kappa
|z|}\right)^j.
\end{eqnarray}
It is easy to show from (\ref{b5}) that 
\begin{eqnarray}\label{b8}
\big|H_n^{(1)}(\kappa |z|)-H_{n,N}^{(1)}(\kappa
|z|)\big|\leq c\bigg(\frac{1}{\kappa |z|}\bigg)^{N+\frac{3}{2}}.
\end{eqnarray}
Using (\ref{b7}), we define $\tilde{u}(x,\kappa)$ as 
\begin{eqnarray}\label{b9}
\tilde{u}(x,\kappa):=-\frac{{\rm i}}{4}\int_{\mathbb R^2}H^{(1)}_{0,2}(\kappa
|x-y|)f(y){\rm d}y.
\end{eqnarray}

\begin{lemma}\label{lemma1b}
The random variable $u(x,\kappa)-\tilde{u}(x,\kappa)$ satisfies almost surely
the condition 
\begin{eqnarray*}
|u(x,\kappa)-\tilde{u}(x,\kappa)|\leq c\kappa^{-\frac{7}{2}}, \quad x\in U,
\end{eqnarray*}
where the constant $c$ depends only on $L^2(D)$-norm of $f$.
\end{lemma}

\begin{proof}
Noting Assumption A, we know that there exists a positive constant $M$
such that $|x-y|\geq M$ holds for all $x\in U$ and $y\in D$. By (\ref{r14}),
(\ref{b8}), and (\ref{b9}), we have for $x\in U$ that
\begin{align*}
|u(x,\kappa)-\tilde{u}(x,\kappa)|&=\left|\frac{\rm
i}{4}\int_{D}\left[H_0^{(1)}(\kappa |x-y|)-H_{0,2}^{(1)}(\kappa
|x-y|)\right]f(y)dy\right|\\
&\lesssim \|H_0^{(1)}(\kappa
|x-\cdot|)-H_{0,2}^{(1)}(\kappa|x-\cdot|)\|_{H^{1,p}(D)}\|f\|_{H_0^{-1,p'}(D)}\\
&\leq c\kappa^{-\frac{7}{2}},
\end{align*}
where the constant $c$ depends only on $H_0^{-1,p}(D)$-norm of $f$. 
\end{proof}

Now we are in the position to compute the covariance of $\tilde{u}(x,\kappa)$. 
Using (\ref{b7}) and (\ref{b9}), we have from a direct calculation that 
\begin{eqnarray}\label{b10}
\mathbb
E(\tilde{u}(x,\kappa_1)\overline{\tilde{u}(x,\kappa_2)})=\frac{1}{16}\sum_{j_1,
j_2=0}^{2}\frac{a_{j_1}^{(0)}\overline{a_{j_2}^{(0)}}}{\kappa_1^{j_1+\frac{1}{2}
}\kappa_2^{j_2+\frac{1}{2}}}\int_{\mathbb R^4}\frac{e^{{\rm
i}(\kappa_1|x-y|-\kappa_2|x-z|)}}{|x-y|^{j_1+\frac{1}{2}}|x-z|^{j_2+\frac{1}{2}}
}\mathbb E(f(y)f(z)){\rm d}y{\rm d}z.
\end{eqnarray}
From (\ref{b10}), it is easy to see that $\mathbb
E(\tilde{u}(x,\kappa_1)\overline{\tilde{u}(x,\kappa_2)})$ is a linear
combination of the following integral 
\begin{eqnarray}\label{b11}
I(x,\kappa_1,\kappa_2):=\frac{1}{\kappa_1^{l_1}\kappa_2^{l_2}}\int_{\mathbb
R^{2d}}e^{{\rm i}(c_1\kappa_1|x-y|-c_2\kappa_2|x-z|)}K(x,y,z)\mathbb
E(q(y)q(z)){\rm d}y{\rm d}z,
\end{eqnarray}
where
\begin{eqnarray*}\label{b12}
K(x,y,z):=\frac{(x_1-y_1)^{m_1}\cdots(x_d-y_d)^{m_d}(x_1-z_1)^{n_1}
\cdots(x_d-z_d)^{n_d}}{|x-y|^{p_1}|x-z|^{p_2}}.
\end{eqnarray*}
Here $q$ stands for a generalized Gaussian random function satisfying
Assumption B, and $l_1$, $l_2$, $c_1$, $c_2$, $m_1$, ..., $n_d$, $p_1$,
$p_2$ are nonnegative constants.

\begin{lemma}\label{lemma2b}
For $\kappa_1,\kappa_2\geq 1$, the estimates
\begin{align}\label{b13}
|I(x,\kappa_1,\kappa_2)|&\leq
c_n(\kappa_1+\kappa_2)^{-(m+2\min(l_1,l_2))}(1+|\kappa_1-\kappa_2|)^{-n},
\\\label{b14} |\mathbb E(\tilde{u}(x,\kappa_1)\tilde{u}(x,\kappa_2))|&\leq
c_n(\kappa_1+\kappa_2)^{-n}(1+|\kappa_1-\kappa_2|)^{-m}
\end{align}
holds uniformly for $x\in U$, where $n\in\mathbb N$ is arbitrary.
\end{lemma}

\begin{proof}
 To estimate the integral $I(x,\kappa_1,\kappa_2)$, we introduce the multiple
coordinate transformation that allows to use the microlocal methods in our
analysis. Noting $\mathbb Eq = 0$ and (\ref{a2}), we conclude that the
correlation function $\mathbb E(q(y)q(z))$ is the Schwartz kernel of a
pseudo-differential operator $C_q$ with a classical symbol $\sigma(y,\xi)\in
S_{1,0}^{-m}(\mathbb R^d\times\mathbb R^d)$ which is defined by  
 \begin{eqnarray*}
 S_{1,0}^{-m}(\mathbb R^d\times\mathbb R^d):=\{a(x,\xi)\in C^{\infty}(\mathbb
R^d\times\mathbb R^d):|\partial_{\xi}^{\alpha}\partial_{x}^{\beta}a(x,\xi)|\leq
C_{\alpha, \beta}(1+|\xi|)^{-m-|\alpha|}\}.
 \end{eqnarray*}
Here $\alpha$, $\beta$ are multiple indices, $|\alpha|$ denotes the sum of its
component. The principle symbol of $C_q$ is $\sigma^p(y,\xi)=\phi(y)|\xi|^{-m}$.
The support of $\mathbb E(q(y)q(z))$ is contained in $D\times D$. We can write
$\mathbb E(q(y)q(z))$ in terms of its symbol by
\begin{eqnarray}\label{b15}
\mathbb E(q(y)q(z))=(2\pi)^{-d}\int_{\mathbb R^d}e^{{\rm i}(y-z)\cdot
\xi}\sigma(y,\xi){\rm d}\xi.
\end{eqnarray}

In order to establish a uniform estimate with respect to the variable $x$, we
extend the covariance function into the space $\mathbb R^{2d}\times \mathbb
R^d$, and define $B_1(y,z,x)=\mathbb E(q(y)q(z))\theta(x)$ where $\theta(x)\in
C_0^{\infty}(\mathbb R^d)$ equals to one in the domain $U$ and has its support
outside of the domain $\overline{D}$. Thus, we have 
\begin{eqnarray*}\label{b16}
B_1(y,z,x)=(2\pi)^{-d}\int_{\mathbb R^d}e^{{\rm i}(y-z)\cdot
\xi}c_1(y,x,\xi){\rm d}\xi,
\end{eqnarray*}
where $c_1(y,x,\xi)=\sigma(y,\xi)\theta(x)\in S_{1,0}^{-m}(\mathbb
R^{2d}\times\mathbb R^d)$ with a principle symbol 
\begin{eqnarray*}\label{b17}
 c_1^p(y,x,\xi)=\phi(y)|\xi|^{-m}\theta(x).
 \end{eqnarray*}

 To proceed the analysis, let us briefly revisit the conormal distributions of
H\"{o}remainder type \cite{L3}. If $X\subset\mathbb R^d$ is an open set and
$S\subset X$ is a smooth submanifold of $X$, we denote by $I(X; S)$ the
distributions in $\mathcal{D'}(X)$ that are smooth in $X\setminus S$ and have a
conormal singularity at $S$. In consequence, by (\ref{b15}), the correlation
function $\mathbb E(q(y)q(z))$ is a conormal distribution in $\mathbb R^{2d}$ of
H\"{o}remainder type having conormal singularity on the surface
$S_1=\{(y,z)\in\mathbb R^{2d}: y-z=0\}$. Moreover,
let $I_{\rm comp}(X;S)$ be the set of distributions supported in a compact
subset of $X$. Let $\bm{D}\subset \mathbb R^{3d}$ be an open set containing
$D\times D\times {\rm supp}(\theta)$ so that $B_1\in I_{\rm comp}({\bf D};
S_1\cap {\bf D})$.

Define the first coordinate transformation $\eta:\mathbb
R^{3d}\rightarrow\mathbb R^{3d}$ by 
\begin{eqnarray}\label{b18}
(v,w,x)=\eta(y,z,x)=(y-z,y+z,x).
\end{eqnarray}
Substituting the coordinate transformation (\ref{b18})  into $B_1(y,z,x)$ gives
\begin{eqnarray*}
B_2(v,w,x)=B_1(\eta^{-1}(v,w,x))=(2\pi)^{-d}\int_{\mathbb R^d} e^{{\rm i}v\cdot
\xi}c_1(\frac{v+w}{2},x,\xi){\rm d}\xi,
\end{eqnarray*}
which means that $B_2\in I(\mathbb R^{3d}, S_2)$ where $S_2:=\{(v,w,x):v=0\}$.
Actually, $B_2\in I_{\rm comp}(X_2, X_2\cap S_2)$ where $X_2:=\eta({\bf D})$. To
find out how the symbol transforms in the change of coordinates, we need to
represent $c_1(\frac{v+w}{2},x,\xi)$ with a symbol that does not depend on $v$.
Using the representation theorem of conormal distribution \cite[Lemma
18.2.1]{L3}), we obtain 
\begin{eqnarray*}
B_2(v,w,x)=(2\pi)^{-d}\int_{\mathbb R^d} e^{{\rm i}v\cdot
\xi}c_2(w,x,\xi){\rm d}\xi,
\end{eqnarray*}
where $c_2(w,x,\xi)$ has the asymptotic expansion
\begin{eqnarray*}
c_2(w,x,\xi)\sim\sum_{l=0}^{\infty}\frac{\langle-{\rm i} D_v,
D_{\xi}\rangle^l}{l!}c_1(\frac{v+w}{2},x,\xi)\bigg|_{v=0}\in
S_{1,0}^{-m}(\mathbb R^{2d}\times\mathbb R^d).
\end{eqnarray*}
 In particular, the principle symbol of $c_2(w,x,\xi)$ is 
 \begin{eqnarray*}
 c_2^p(w,x,\xi)=\phi(\frac{v+w}{2})|\xi|^{-m}\theta(x)\bigg|_{v=0}.
 \end{eqnarray*}

We consider the phase of $I(x,\kappa_1,\kappa_2)$. A simple calculation
shows that 
\begin{align}
c_1\kappa_1|x-y|-c_2\kappa_2|x-z|&=(c_1\kappa_1+c_2\kappa_2)\frac{|x-y|-|x-z|}{
2}\notag\\\label{b19}
&\quad +(c_1\kappa_1-c_2\kappa_2)\frac{|x-y|+|x-z|}{2}.
\end{align}
In the second set of coordinates, let $\frac{|x-y|\pm|x-z|}{2}$ play the role of
two coordinates. We will do this change in two steps. First, for the
two-dimensional case where $d=2$, we define $\tau_1:\mathbb R^{6}\rightarrow
\mathbb R^{6}$ by
\[
\tau_1(y,z,x)=(E_1,E_2,x),
\]
where $E_1=(t_1,s_1)$ and $E_2=(t_2,s_2)$
with
\begin{eqnarray*}
t_1=\frac{1}{2}|x-y|,\quad
s_1=\frac{1}{2}\arcsin\left(\frac{y_1-x_1}{|x-y|}\right),\\
t_2=\frac{1}{2}|x-z|,\quad
s_2=\frac{1}{2}\arcsin\left(\frac{z_1-x_1}{|x-z|}\right).
\end{eqnarray*}
For the three-dimensional case where $d=3$, we define $\tau_1:\mathbb
R^9\rightarrow \mathbb R^9$ by
\[
\tau_1(y,z,x)=(E_1,E_2,x),
\]
where $E_1=(t_1,s_1,r_1)$ and $E_2=(t_2,s_2,r_2)$ with
\begin{eqnarray*}
t_1=\frac{1}{2}|x-y|,\quad
s_1=\frac{1}{2}\arccos\left(\frac{y_3-x_3}{|x-y|}\right),\quad r_1 =
\frac{1}{2}|x-y|\arctan\left(\frac{y_2-x_2}{y_1-x_1}\right),\\
t_2=\frac{1}{2}|x-z|,\quad
s_2=\frac{1}{2}\arccos\left(\frac{z_3-x_3}{|x-z|}\right),\quad r_2 =
\frac{1}{2}|x-z|\arctan\left(\frac{z_2-x_2}{z_1-x_1}\right).
\end{eqnarray*}
Second, we define $\tau_2:\mathbb R^{3d}\rightarrow \mathbb R^{3d}$ by
\[
(g,h,x)=\tau_2(E_1,E_2,x)=(E_1-E_2,E_1+E_2,x). 
\]
Thus, combining the definitions of $\tau_1$, $\tau_2$ and (\ref{b19}), we have
\begin{eqnarray}\label{b20}
c_1\kappa_1|x-y|-c_2\kappa_2|x-z|=(c_1\kappa_1+c_2\kappa_2)g\cdot
e_1+(c_1\kappa_1-c_2\kappa_2)h\cdot e_1,
\end{eqnarray}
where $e_1=(1,0)$ for $d=2$, and $e_1=(1,0,0)$ for $d=3$. Now we denote
$\tau=\tau_2\circ\tau_1:\mathbb R^{3d}\rightarrow\mathbb R^{3d}$ with
$\tau(y,z,x)=(g,h,x)$. We consider the transformation
$\rho=\eta\circ\tau^{-1}:\mathbb R^{3d}\rightarrow\mathbb R^{3d}$ with
$\rho(g,h,x)=(v,w,x)$. Let us decompose the coordinate transform $\rho$ into two
parts $\rho=(\rho_1,\rho_2)$, the $\mathbb R^d$-valued function
$\rho_1(g,h,x)=v$ and the $\mathbb R^{2d}$-valued function
$\rho_2(g,h,x)=(w,x)$. The Jacobian $J_{\rho}$ corresponding to the decomposing
of the variables is given by 
\begin{eqnarray*}
J_{\rho}=\begin{bmatrix} \rho'_{11} & \rho'_{12}\\ \rho'_{21} & \rho'_{22}
\end{bmatrix}
=\begin{bmatrix} J_{g}\rho_1 & J_{(h,x)}\rho_1\\ J_{g}\rho_2 &J_{(h,x)}\rho_2
\end{bmatrix}.
\end{eqnarray*}
By the definition of $\rho$, it is easy to see that $v=0$ if $g=0$. Hence we
have $\rho_1(0,h,x)=0$ which implies $\rho'_{12}(0,h,x)=0$.

Next we consider the pull-back distribution $B_3=B_2\circ\rho$. It follow from 
\cite[Theorem 18.2.9]{L3} that we get a representation for $B_3$:
\begin{eqnarray*}\label{b21}
B_3(g,h,x)=(2\pi)^{-d}\int_{\mathbb R^d}e^{{\rm i}g\cdot
\xi}c_3(h,x,\xi){\rm d}\xi,
\end{eqnarray*}
where $c_3(h,x,\xi)\in S_{1,0}^{-m}(\mathbb R^{2d}\times\mathbb R^d)$ is a
symbol satisfying 
\begin{eqnarray*}
c_3(h,x,\xi)=c_2(\rho_2(g,h,x),((\rho'_{11}(g,h,x))^{-1})^T\xi)\times|{\rm
det}\rho'_{11}(g,h,x)|^{-1}\bigg|_{g=0}+r(h,x,\xi),
\end{eqnarray*}
where $r(h,x,\xi)\in S_{1,0}^{-m-1}(\mathbb R^{2d}\times\mathbb R^d)$. The
principle symbol of $c_3(h,x,\xi)$ is given by
\begin{eqnarray*}
c_3^p(h,x,\xi)=\phi(y(g,h,x))|((\rho'_{11}(g,h,x))^{-1})^T\xi|^{-m}
\theta(x)\times|{\rm det}\rho'_{11}(g,h,x)|^{-1}\big|_{g=0}.
\end{eqnarray*}
Let $X_3:=\tau({\bf D})$ and $S_3:=\{(g,h,x):g=0\}$, we have $B_3\in I_{\rm
comp}(X_3, X_3\cap S_3)$. So we can write $I(x,\kappa_1,\kappa_2)$ in the
following form
\begin{eqnarray*}
I(x,\kappa_1,\kappa_2)= \frac{1}{\kappa_1^{l_1}\kappa_2^{l_2}}\int_{\mathbb
R^{2d}}e^{{\rm i}[(c_1\kappa_1+c_2\kappa_2)g\cdot
e_1+(c_1\kappa_1-c_2\kappa_2)h\cdot e_1]}B_3(g,h,x)H(g,h,x){\rm
d}g{\rm d}h,
\end{eqnarray*}
where
\begin{eqnarray}
\label{b23}
H(g,h,x)=K(x,y,z){\rm det}((\tau^{-1})'(g,h,x)).
\end{eqnarray}
Here $y=y(g,h,x)$ and $z=z(g,h,x)$. Since $H$ is smooth in $X_3$ in all
variables and $I(\mathbb R^{3d}, S_3)$ is closed under multiplication with a
smooth function, we conclude that $B_3(g,h,x)H(g,h,x)\in I(\mathbb R^{3d},
S_3)$. Multiplying (\ref{b20}) by $H$, we arrive at
\begin{eqnarray}\label{b24}
B_3(g,h,x)H(g,h,x)=(2\pi)^{-d}\int_{\mathbb R^d}e^{{\rm i}g\cdot
\xi}c_4(h,x,\xi){\rm d}\xi,
\end{eqnarray}
where $c_4(h,x,\xi)$ has the asymptotic expansion 
\begin{eqnarray*}
c_4(h,x,\xi)\sim\sum_{l=0}^{\infty}\frac{\langle-{\rm
i}D_g,D_{\xi}\rangle^l}{l!}(c_3(h,x,\xi)H(g,h,x))\bigg|_{g=0}.
\end{eqnarray*}
In particular, the principle symbol of $c_4(h,x,\xi)$ is given by
\begin{eqnarray}\label{b25}
c_4^p(h,x,\xi)=\phi(y(g,h,x))|((\rho'_{11}(g,h,x))^{-1})^T\xi|^{-m}\theta(x)
|{\rm det}\rho'_{11}(g,h,x)|^{-1}H(g,h,x)\big|_{g=0}.
\end{eqnarray}

Combining (\ref{b24}) and the Fourier inversion rule, we obtain
\begin{eqnarray}\label{b26}
B_3(g,h,x)H(g,h,x)=(\mathcal{F}^{-1}c_4)(h,x,g).
\end{eqnarray}
Substituting (\ref{b26}) into $I(x,\kappa_1,\kappa_2)$ gives
\begin{align}\label{b27}
I(x,\kappa_1,\kappa_2)&=\frac{1}{\kappa_1^{l_1}\kappa_2^{l_2}}\int_{\mathbb
R^{2d}}e^{{\rm i}[(c_1\kappa_1+c_2\kappa_2)g\cdot
e_1+(c_1\kappa_1-c_2\kappa_2)h\cdot
e_1]}(\mathcal{F}^{-1}c_4)(h,x,g){\rm d}g{\rm d}h\\\notag
&=\frac{1}{\kappa_1^{l_1}\kappa_2^{l_2}}\int_{\mathbb R^d}e^{{\rm
i}(c_1\kappa_1-c_2\kappa_2)h\cdot
e_1}c_4(h,x,-(c_1\kappa_1+c_2\kappa_2)e_1){\rm d}h\\\notag
&=\frac{1}{\kappa_1^{l_1}\kappa_2^{l_2}}\frac{1}{{\rm
i}(c_1\kappa_1-c_2\kappa_2)}\int_{\mathbb
R^d}c_4(h,x,-(c_1\kappa_1+c_2\kappa_2)e_1){\rm d}e^{{\rm
i}(c_1\kappa_1-c_2\kappa_2)h_1}\cdots {\rm d}h_{d}\\\notag
&=-\frac{1}{\kappa_1^{l_1}\kappa_2^{l_2}}\frac{1}{{\rm
i}(c_1\kappa_1-c_2\kappa_2)}\int_{\mathbb R^d}e^{{\rm
i}(c_1\kappa_1-c_2\kappa_2)h_1}\partial_{h_1}c_4(h,x,
-(c_1\kappa_1+c_2\kappa_2)e_1){\rm d}h\\\notag
&=(-1)^n\frac{1}{\kappa_1^{l_1}\kappa_2^{l_2}}\frac{1}{({\rm
i}(c_1\kappa_1-c_2\kappa_2))^n}\int_{\mathbb R^d}e^{{\rm
i}(c_1\kappa_1-c_2\kappa_2)h_1}\partial^n_{h_1}c_4(h,x,
-(c_1\kappa_1+c_2\kappa_2)e_1){\rm d}h,
\end{align}
where we use the integrations by parts $n$ times and the fact that
$c_4(h,x,\xi)$ is $C^{\infty}$ smooth and compactly supported in the $(g,h,x)$
variables. Since $c_4(h,x,\xi)\in S_{1,0}^{-m}(\mathbb R^{2d}\times\mathbb
R^d)$, we have $|\partial^n_{h_1}c_4(h,x,\xi)|\leq c_n(1+|\xi|)^{-m}$ for all
positive integer $n$, where $c_n$ is independent of $(h,x)\in \mathbb R^{2d}$.
Therefore
\begin{align}
|I(x,\kappa_1,\kappa_2)|&\lesssim
\frac{1}{\kappa_1^{l_1}\kappa_2^{l_2}}\frac{1}{(1+|c_1\kappa_1-c_2\kappa_2|)^n}
\frac{1}{(1+|c_1\kappa_1+c_2\kappa_2|)^{m}}\notag\\\label{b28}
&\lesssim
\frac{1}{(\kappa_1\kappa_2)^{\min(l_1,l_2)}}\frac{1}{
(1+|c_1\kappa_1-c_2\kappa_2|)^n}\frac{1}{(c_1\kappa_1+c_2\kappa_2)^{m}},
\end{align}
where we use the fact that $\kappa_1\geq1, \kappa_2\geq 1$. We need to
consider the cases where $|c_1\kappa_1-c_2\kappa_2|\geq
(c_1\kappa_1+c_2\kappa_2)/2$ and
$|c_1\kappa_1-c_2\kappa_2|\leq(c_1\kappa_1+c_2\kappa_2)/2$.
If  $|c_1\kappa_1-c_2\kappa_2|\leq (c_1\kappa_1+c_2\kappa_2)/2$, a simple
calculation shows that $\kappa_1\kappa_2\geq
3(c_1\kappa_1+c_2\kappa_2)^2/(16c_1c_2)$ which implies 
\[
|I(x,\kappa_1,\kappa_2)|\lesssim\frac{1}{(1+|c_1\kappa_1-c_2\kappa_2|)^n}\frac
{1}{(c_1\kappa_1+c_2\kappa_2)^{m+2\min(l_1,l_2)}}.
\]
If $|c_1\kappa_1-c_2\kappa_2|\geq(c_1\kappa_1+c_2\kappa_2)/2$, we have
\[
|I(x,\kappa_1,\kappa_2)|\lesssim\frac{1}{(1+|c_1\kappa_1-c_2\kappa_2|)^{
n-2\min(l_1,l_2)}}\frac{1}{(c_1\kappa_1+c_2\kappa_2)^{m+2\min(l_1,l_2)}}.
\]
Noting that the positive integer $n$ is arbitrary, we conclude  
 \begin{align*}
|I(x,\kappa_1,\kappa_2)|&\lesssim\frac{1}{(1+|c_1\kappa_1-c_2\kappa_2|)^{n}}
\frac{1}{(c_1\kappa_1+c_2\kappa_2)^{m+2\min(l_1,l_2)}}\\
&\lesssim(1+|\kappa_1-\kappa_2|)^{-n}(\kappa_1+\kappa_2)^{-(m+2\min(l_1,l_2))},
\end{align*}
where we use the facts $c_1\kappa_1+c_2\kappa_2\leq
\min(c_1,c_2)(\kappa_1+\kappa_2)$ and $|c_1\kappa_1-c_2\kappa_2|\leq
c|\kappa_1-\kappa_2|$ for some constant $c$. So the estimate (\ref{b13}) holds.

Similarly, it is easy to see that $\mathbb
E(\tilde{u}(x,\kappa_1)\tilde{u}(x,\kappa_2))$ is a linear combination of the
integral
\begin{eqnarray*}
\tilde{I}(x,\kappa_1,\kappa_2):=\frac{1}{\kappa_1^{l_1}\kappa_2^{l_2}}\int_{
\mathbb R^{2d}}e^{{\rm i}(c_1\kappa_1|x-y|+c_2\kappa_2|x-z|)}K(x,y,z)\mathbb
E(q(y)q(z)){\rm d}y{\rm d}z.
\end{eqnarray*}
Observe that $\tilde{I}$ is analogous to $I$ where we replace $\kappa_2$ with
$-\kappa_2$. Since the proof of (\ref{b28}) allows $\kappa_2$ to be negative,
we may show that 
\begin{align*}
|\tilde{I}(x,\kappa_1,\kappa_2)|&\lesssim
\frac{1}{\kappa_1^{l_1}\kappa_2^{l_2}}\frac{1}{|c_1\kappa_1+c_2\kappa_2|^n}\frac
{1}{(1+|c_1\kappa_1-c_2\kappa_2|)^{m}}\\
&\lesssim (\kappa_1+\kappa_2)^{-n}(1+|\kappa_1-\kappa_2|)^{-m},
\end{align*}
which shows the estimate (\ref{b14}) and completes the proof. 
\end{proof}

Since $\mathbb E(\tilde{u}(x,\kappa_1)\overline{\tilde{u}(x,\kappa_2)})$ is a
linear combination of $I$ which satisfies the estimate (\ref{b13}), thus the
following result is a direct consequence of Lemma \ref{lemma2b}. 

\begin{lemma}\label{lemma3b}
For $\kappa_1,\kappa_2\geq 1$, the estimates
\begin{align*}
|\mathbb E(\tilde{u}(x,\kappa_1)\overline{\tilde{u}(x,\kappa_2)})|& \leq
c_n(\kappa_1+\kappa_2)^{-(m+1)}(1+|\kappa_1-\kappa_2|)^{-n},\\
|\mathbb E(\tilde{u}(x,\kappa_1)\tilde{u}(x,\kappa_2))|&\leq
c_n(\kappa_1+\kappa_2)^{-n}(1+|\kappa_1-\kappa_2|)^{-m}.
\end{align*}
holds uniformly for $x\in U$, where $n\in \mathbb N$ is arbitrary and $c_n>0$ is
a constant depending only on $n$.
\end{lemma}

To derive the linear relationship between the scattering data and the function
in the principle symbol, it is required to compute the order of $\mathbb
E(|\tilde{u}(x,\kappa)|^2)$ in terms of $\kappa$. To this end, we study the
asymptotic of $I(x,\kappa,\kappa)$ for large $\kappa$.
\begin{lemma}\label{lemma4b}
For $\kappa_1=\kappa_2=\kappa$, the following asymptotic holds
\[
 I(x,\kappa,\kappa) =
R_d(x,\kappa)\kappa^{-(l_1+l_2+m)}+O(\kappa^{-(l_1+l_2+m+1)}),
 \]
 where $R_d(x,\kappa)$ is given by
 \[
R_d(x,\kappa)=C_d\int_{\mathbb R^d}e^{{\rm
i}(c_1-c_2)|x-y|\kappa}\frac{(x_1-y_1)^{m_1+n_1}\cdots(x_d-y_d)^{m_d+n_d}}{
|x-y|^{p_1+p_2}}\phi(y){\rm d}y
\]
with 
\begin{eqnarray*}
C_2=-\frac{1}{64}c_{m},\quad C_3 =\frac{1}{8}c_{m},\quad
c_{m}=\left(\frac{2}{c_1+c_2}\right)^{m}.
\end{eqnarray*}
\end{lemma}

\begin{proof}
Setting $\kappa_1=\kappa_2=\kappa$ in (\ref{b27}) gives 
\begin{eqnarray*}\label{b33}
I(x,\kappa,\kappa) = \frac{1}{\kappa^{l_1+l_2}}\int_{\mathbb R^{d}}e^{{\rm
i}(c_1-c_2)\kappa h\cdot e_1}c_4(h,x,-(c_1+c_2)\kappa e_1){\rm d}g{\rm d}h.
\end{eqnarray*}
 The symbol $c_4(h,x,\xi)\in S_{1,0}^{-m}(\mathbb R^{2d}\times\mathbb R^d)$ can
be decomposed into 
\begin{eqnarray*}
c_4(h,x,\xi)=c_4^p(h,x,\xi)+r(h,x,\xi),
\end{eqnarray*}
where $c_4^p(h,x,\xi)\in S_{1,0}^{-m}(\mathbb R^{2d}\times\mathbb R^d)$ is the
principal symbol which is given by (\ref{b25}) and $r(h,x,\xi)\in
S_{1,0}^{-m-1}(\mathbb R^{2d}\times\mathbb R^d)$ is the lower order remainder
terms which is smooth and compactly supported in $(g,h,x)$-variables. Thus, we
have 
\begin{align}
I(x,\kappa,\kappa)&=\frac{1}{\kappa^{l_1+l_2}}\int_{\mathbb R^d}e^{{\rm
i}(c_1-c_2)\kappa h\cdot e_1}(c_4^p(h,x,-(c_1+c_2)\kappa
e_1)+r(h,x,-(c_1+c_2)\kappa e_1)){\rm d}h\notag\\\label{b34}
&=\frac{1}{\kappa^{l_1+l_2}}\int_{\mathbb R^d}e^{{\rm i}(c_1-c_2)\kappa h\cdot
e_1}c_4^p(h,x,-(c_1+c_2)\kappa e_1){\rm d}h+O(\kappa^{-(l_1+l_2+m+1)}).
\end{align}
By (\ref{b25}), 
\begin{align}
c_4^p(h,x,-(c_1+c_2)\kappa
e_1)&=\phi(y(g,h,x))((c_1+c_2)|((\rho'_{11}(g,h,x))^{-1})^\top
e_1|\kappa)^{-m}\theta(x)\notag\\\label{b35}
&\quad\times|{\rm det}\rho'_{11}(g,h,x)|^{-1}H(g,h,x)\big|_{g=0}.
\end{align}
Letting $a=(c_1+c_2)^2|((\rho'_{11}(g,h,x))^{-1})^\top e_1|^2\neq 0$, we
substitute (\ref{b35}) into formula (\ref{b34}) and obtain
\[
I(x,\kappa,\kappa) =
R_d(x,\kappa)\kappa^{-(l_1+l_2+m)}+O(\kappa^{-(l_1+l_2+m+1)}),
\]
where
\begin{eqnarray}\label{b38}
R_d(x,\kappa)=\theta(x)\int_{\mathbb R^d} e^{{\rm i}(c_1-c_2)\kappa h\cdot
e_1}\frac{\phi(y(0,h,x))H(0,h,x)}{a^{\frac{m}{2}}|{\rm
det}\rho'_{11}(0,h,x)|}{\rm d}h.
\end{eqnarray}

Next we need to compute $a$. Noting that
$a=(c_1+c_2)^2|((\rho'_{11}(g,h,x))^{-1})^\top e_1|^2$, we compute
$\rho'_{11}(g,h,x))^{-1}$ first. 

In two dimensions, we have from the definition of $\rho_1$ that
\begin{eqnarray*}
\rho'_{11}(g,h,x)=\partial_{g}v=\begin{bmatrix} \partial_{g_1}v_1 &
\partial_{g_2}v_1\\ \partial_{g_1}v_2 & \partial_{g_2}v_2\end{bmatrix}.
\end{eqnarray*}
It is convenient to compute
\begin{eqnarray*}
\rho'_{11}(g,h,x)^{-1}=\partial_{v}g=\begin{bmatrix} \partial_{v_1}g_1 &
\partial_{v_2}g_1\\ \partial_{v_1}g_2 & \partial_{v_2}g_2\end{bmatrix}.
\end{eqnarray*}
By the definition of $\eta$, we have $v=y-z$, $w=y+z$ which implies that
$y=(v+w)/2$, $z=(w-v)/2$. Thus we obtain $y_1=(v_1+w_1)/2$,  $y_2=(v_2+w_2)/2$,
$z_1=(w_1-v_1)/2$, and $z_2=(w_2-v_2)/2$.  When $g=0$ which means $y=z$, we have
\begin{align*}
\frac{\partial g_1}{\partial v_1}&=\frac{\partial t_1}{\partial
v_1}-\frac{\partial t_2}{\partial v_1}=\frac{\partial t_1}{\partial
y_1}\frac{\partial y_1}{\partial v_1}-\frac{\partial t_2}{\partial
z_1}\frac{\partial z_1}{\partial v_1}
=\frac{1}{4}\left(\frac{y_1-x_1}{|x-y|}+\frac{z_1-x_1}{|x-z|}\right)=\frac{1}{2}
\frac{y_1-x_1}{|x-y|}=\frac{1}{2}\sin\alpha,\\
\frac{\partial g_1}{\partial v_2}&=\frac{\partial t_1}{\partial
v_2}-\frac{\partial t_2}{\partial v_2}=\frac{\partial t_1}{\partial
y_2}\frac{\partial y_2}{\partial v_2}-\frac{\partial t_2}{\partial
z_2}\frac{\partial z_2}{\partial v_2}
=\frac{1}{4}\left(\frac{y_2-x_2}{|x-y|}+\frac{z_2-x_2}{|x-z|}\right)=\frac{1}{2}
\frac{y_2-x_2}{|x-y|}=\frac{1}{2}\cos\alpha,\\
\frac{\partial g_2}{\partial v_1}&=\frac{\partial s_1}{\partial
v_1}-\frac{\partial s_2}{\partial v_1}=\frac{\partial s_1}{\partial
y_1}\frac{\partial y_1}{\partial v_1}-\frac{\partial s_2}{\partial
z_1}\frac{\partial z_1}{\partial v_1}
=\frac{1}{2}\left(\frac{\partial s_1}{\partial y_1}+\frac{\partial s_2}{\partial
z_1}\right)\\
&=\frac{1}{2}\bigg[\frac{y_1-x_1}{|x-y|}\arcsin\left(\frac{y_1-x_1}{|x-y|}
\right)+\bigg(1-\left(\frac{y_1-x_1}{|x-y|}\right)^2\bigg)^{\frac{1}{2}}\bigg]
=\frac{1}{2}(\alpha\sin\alpha+\cos\alpha),\\
\frac{\partial g_2}{\partial v_2}&=\frac{\partial s_1}{\partial
v_2}-\frac{\partial s_2}{\partial v_2}=\frac{\partial s_1}{\partial
y_2}\frac{\partial y_2}{\partial v_2}-\frac{\partial s_2}{\partial
z_2}\frac{\partial z_2}{\partial v_2}
=\frac{1}{2}\left(\frac{\partial s_1}{\partial y_2}+\frac{\partial s_2}{\partial
z_2}\right)\\
&=\frac{1}{2}\bigg[\frac{y_2-x_2}{|x-y|}\arcsin\left(\frac{y_1-x_1}{|x-y|}
\right)-\left(1-\left(\frac{y_1-x_1}{|x-y|}\right)^2\right)^{-\frac{1}{2}}\frac{
(y_1-x_1)(y_2-x_2)}{|x-y|^2}\bigg]\\
&=\frac{1}{2}(\alpha\cos\alpha-\sin\alpha),
\end{align*}
where $\alpha=\arcsin\big(\frac{y_1-x_1}{|x-y|}\big)$. Hence,
\begin{align*}
\left(\rho'_{11}(0,h,x)\right)^{-1}&=\frac{1}{2}\begin{bmatrix} \sin\alpha &
\cos\alpha\\ \alpha\sin\alpha+\cos\alpha &
\alpha\cos\alpha-\sin\alpha\end{bmatrix},\\
\left(\left(\rho'_{11}(0,h,x)\right)^{-1}\right)^{T}e_1&=\frac{1}{2}
(\sin\alpha, \cos\alpha)^\top.
\end{align*}
Thus we obtain $a=(c_1+c_2)^2/4$ and $|{\rm det}\rho'_{11}(0,h,x)|^{-1}=1/4$.
Next we focus on the computation of $H(0,h,x)$. From (\ref{b23}) we have
$H(g,h,x)=K(x,y,z){\rm det}((\tau^{-1})'(g,h,x))$, thus we compute $|{\rm
det}\tau'(0,h,x)|$ first. Recalling that $\tau:\mathbb R^6\rightarrow\mathbb
R^6$ is given by $\tau(g,h,x)=(y,z,x)$, we have 
\begin{eqnarray*}
\tau'(g,h,x)=\begin{bmatrix}\partial_g y & \partial_h
y & \partial_x y\\ \partial_g z & \partial_h z & \partial_x z\\
\partial_g x & \partial_h x & \partial_x x
\end{bmatrix}, \quad
(\tau'(g,h,x))^{-1}=\begin{bmatrix}\partial_y g & \partial_z
g & \partial_x g \\ \partial_y h & \partial_z h & \partial_x
h \\ 
\partial_y x & \partial_z x & \partial_x
x \end{bmatrix}.
\end{eqnarray*}
Now we calculate $\partial_y g$. Noting $g=(t_1-t_2,
s_1-s_2)$, we obtain 
\begin{align*}
\frac{\partial g_1}{\partial y_1}=\frac{\partial t_1}{\partial
y_1}&=\frac{1}{2}\frac{y_1-x_1}{|x-y|}=\frac{1}{2}\sin\alpha,\qquad
\frac{\partial g_1}{\partial y_2}=\frac{\partial t_1}{\partial
y_2}=\frac{1}{2}\frac{y_2-x_2}{|x-y|}=\frac{1}{2}\cos\alpha,\\
\frac{\partial g_2}{\partial y_1}=\frac{\partial s_1}{\partial
y_1}&=\frac{1}{2}\bigg[\frac{y_1-x_1}{|x-y|}\arcsin\left(\frac{y_1-x_1}{|x-y|}
\right)+\bigg(1-\left(\frac{y_1-x_1}{|x-y|}\right)^2\bigg)^{\frac{1}{2}}\bigg]
=\frac{1}{2}(\alpha\sin\alpha+\cos\alpha),\\
\frac{\partial g_2}{\partial y_2}=\frac{\partial s_1}{\partial
y_2}&=\frac{1}{2}\bigg[\frac{y_2-x_2}{|x-y|}\arcsin\left(\frac{y_1-x_1}{|x-y|}
\right)-\left(1-\left(\frac{y_1-x_1}{|x-y|}\right)^2\right)^{-\frac{1}{2}}\frac{
(y_1-x_1)(y_2-x_2)}{|x-y|^2}\bigg]\\
&=\frac{1}{2}(\alpha\cos\alpha-\sin\alpha).
\end{align*}
Thus
\[
\partial_y g = \frac{1}{2}\begin{bmatrix} \sin\alpha &
\cos\alpha\\ \alpha\sin\alpha+\cos\alpha &
\alpha\cos\alpha-\sin\alpha\end{bmatrix}.
\]
When $g=0$, we have $y=z$. A similar calculation yields that $\partial_z
g =-\partial_y g$, $\partial_y h=\partial_y g$, $\partial_z h=\partial_y g$.
Obviously, $\partial_y x=0$, $\partial_z x=0$, $\partial_x x=I$.
Hence
\begin{align*}
&{\rm det}((\tau'(0,h,x))^{-1})\\
&=\frac{1}{2^4}\left|\begin{array}{cccc}\sin\alpha
& \cos\alpha & -\sin\alpha &-\cos\alpha\\
\alpha\sin\alpha+\cos\alpha&\alpha\cos\alpha-\sin\alpha&-\alpha\sin\alpha-\cos
\alpha&-\alpha\cos\alpha+\sin\alpha\\ \sin\alpha & \cos\alpha & \sin\alpha
&\cos\alpha\\
\alpha\sin\alpha+\cos\alpha&\alpha\cos\alpha-\sin\alpha&\alpha\sin\alpha+\cos
\alpha&\alpha\cos\alpha-\sin\alpha\end{array}\right|.
\end{align*}
A direct calculation shows that ${\rm det}((\tau'(0,h,x))^{-1})=1/4$ which
implies ${\rm det}((\tau'(0,h,x)))=4$. Substituting these results into
(\ref{b38}) gives
\begin{eqnarray*}
R_2(x,\kappa)=\frac{1}{16}c_m\int_{\mathbb R^2}e^{{\rm
i}(c_1-c_2)|x-y|\kappa}\frac{(x_1-y_1)^{m_1+n_1}(x_2-y_2)^{m_2+n_2}}{|x-y|^{
p_1+p_2}}\phi(y){\rm d}h,
\end{eqnarray*}
where $(y,z,x)=\tau(g,h,x)$. Noting ${\rm det}\left(\frac{\partial
h}{\partial y}\right)=-\frac{1}{4}$, we have 
\begin{eqnarray*}
R_2(x,\kappa)=C_2\int_{\mathbb R^2}e^{{\rm
i}(c_1-c_2)|x-y|\kappa}\frac{(x_1-y_1)^{m_1+n_1}(x_2-y_2)^{m_2+n_2}}{|x-y|^{
p_1+p_2}}\phi(y){\rm d}y.
\end{eqnarray*}

In three dimensions, by the definition of $\rho_1$, we have
\begin{eqnarray*}
\rho'_{11}(g,h,x)=\partial_{g}v=\begin{bmatrix} \partial_{g_1}v_1 &
\partial_{g_2}v_1 & \partial_{g_3}v_1\\ \partial_{g_1}v_2 & \partial_{g_2}v_2 &
\partial_{g_3}v_2\\
\partial_{g_1}v_3 & \partial_{g_2}v_3 & \partial_{g_3}v_3\end{bmatrix},
\end{eqnarray*}
which is available after we compute its inverse matrix
\begin{eqnarray*}
\rho'_{11}(g,h,x)^{-1}=\partial_{v}g=\begin{bmatrix} \partial_{v_1}g_1 &
\partial_{v_2}g_1 & \partial_{v_3}g_1\\ \partial_{v_1}g_2 & \partial_{v_2}g_2 &
\partial_{v_3}g_2\\
\partial_{v_1}g_3 & \partial_{v_2}g_3 & \partial_{v_3}g_3\end{bmatrix}.
\end{eqnarray*}
By the definition of $\eta$, we have $v=y-z$, $w=y+z$ which implies that
$y=(v+w)/2$, $z=(w-v)/2$. Thus $y_1=(v_1+w_1)/2$,  $y_2=(v_2+w_2)/2$, 
$y_3=(v_3+w_3)/2$, $z_1=(w_1-v_1)/2$, $z_2=(w_2-v_2)/2$, and $z_3=(w_3-v_3)/2$.
Define
\begin{eqnarray*}
\alpha = \arccos\left(\frac{y_3-x_3}{|x-y|}\right),\quad \beta =
\arctan\left(\frac{y_2-x_2}{y_1-x_1}\right).
\end{eqnarray*}
A direct computation shows that 
\begin{eqnarray*}
\frac{y_1-x_1}{|x-y|}=\sin\alpha\cos\beta,\quad
\frac{y_2-x_2}{|x-y|}=\sin\alpha\sin\beta,\quad 
\frac{y_3-x_3}{|x-y|}=\cos\alpha. 
\end{eqnarray*}
 When $g=0$ which implies $y=z$, we have
\begin{align*}
\frac{\partial g_1}{\partial v_1}&=\frac{\partial t_1}{\partial
v_1}-\frac{\partial t_2}{\partial v_1}=\frac{\partial t_1}{\partial
y_1}\frac{\partial y_1}{\partial v_1}-\frac{\partial t_2}{\partial
z_1}\frac{\partial z_1}{\partial v_1}
=\frac{1}{4}\left(\frac{y_1-x_1}{|x-y|}+\frac{z_1-x_1}{|x-z|}\right)=\frac{1}{2}
\frac{y_1-x_1}{|x-y|}=\frac{1}{2}\sin\alpha\cos\beta,\\
\frac{\partial g_1}{\partial v_2}&=\frac{\partial t_1}{\partial
v_2}-\frac{\partial t_2}{\partial v_2}=\frac{\partial t_1}{\partial
y_2}\frac{\partial y_2}{\partial v_2}-\frac{\partial t_2}{\partial
z_2}\frac{\partial z_2}{\partial v_2}
=\frac{1}{4}\left(\frac{y_2-x_2}{|x-y|}+\frac{z_2-x_2}{|x-z|}\right)=\frac{1}{2}
\frac{y_2-x_2}{|x-y|}=\frac{1}{2}\sin\alpha\sin\beta,\\
\frac{\partial g_1}{\partial v_3}&=\frac{\partial r_1}{\partial
v_3}-\frac{\partial r_2}{\partial v_3}=\frac{\partial r_1}{\partial
y_3}\frac{\partial y_3}{\partial v_3}-\frac{\partial r_2}{\partial
z_3}\frac{\partial z_3}{\partial v_3}
=\frac{1}{4}\left(\frac{y_3-x_3}{|x-y|}+\frac{z_3-x_3}{|x-z|}\right)=\frac{1}{2}
\frac{y_3-x_3}{|x-y|}=\frac{1}{2}\cos\alpha,\\
\frac{\partial g_2}{\partial v_1}&=\frac{\partial s_1}{\partial
v_1}-\frac{\partial s_2}{\partial v_1}=\frac{\partial s_1}{\partial
y_1}\frac{\partial y_1}{\partial v_1}-\frac{\partial s_2}{\partial
z_1}\frac{\partial z_1}{\partial v_1}
=\frac{1}{2}\left(\frac{\partial s_1}{\partial y_1}+\frac{\partial s_2}{\partial
z_1}\right)\\
&=\frac{1}{2}\bigg[\frac{y_1-x_1}{|x-y|}\arccos\left(\frac{y_3-x_3}{|x-y|}
\right)+\bigg(1-\left(\frac{y_3-x_3}{|x-y|}\right)^2\bigg)^{-\frac{1}{2}}\frac{
(y_3-x_3)(y_1-x_1)}{|x-y|^2}\bigg]\\
&=\frac{1}{2}\cos\beta(\alpha\sin\alpha+\cos\alpha),\\
\frac{\partial g_2}{\partial v_2}&=\frac{\partial s_1}{\partial
v_2}-\frac{\partial s_2}{\partial v_2}=\frac{\partial s_1}{\partial
y_2}\frac{\partial y_2}{\partial v_2}-\frac{\partial s_2}{\partial
z_2}\frac{\partial z_2}{\partial v_2}
=\frac{1}{2}\left(\frac{\partial s_1}{\partial y_2}+\frac{\partial s_2}{\partial
z_2}\right)\\
&=\frac{1}{2}\bigg[\frac{y_2-x_2}{|x-y|}\arccos\left(\frac{y_3-x_3}{|x-y|}
\right)+\bigg(1-\left(\frac{y_3-x_3}{|x-y|}\right)^2\bigg)^{-\frac{1}{2}}\frac{
(y_3-x_3)(y_2-x_2)}{|x-y|^2}\bigg]\\
&=\frac{1}{2}\sin\beta(\alpha\sin\alpha+\cos\alpha),\\
\frac{\partial g_2}{\partial v_3}&=\frac{\partial s_1}{\partial
v_3}-\frac{\partial s_2}{\partial v_3}=\frac{\partial s_1}{\partial
y_3}\frac{\partial y_3}{\partial v_3}-\frac{\partial s_2}{\partial
z_3}\frac{\partial z_3}{\partial v_3}
=\frac{1}{2}\left(\frac{\partial s_1}{\partial y_3}+\frac{\partial s_2}{\partial
z_3}\right)\\
&=\frac{1}{2}\bigg[\frac{y_3-x_3}{|x-y|}\arccos\left(\frac{y_3-x_3}{|x-y|}
\right)+\bigg(1-\left(\frac{y_3-x_3}{|x-y|}\right)^2\bigg)^{\frac{1}{2}}\bigg]
=\frac{1}{2}(\alpha\cos\alpha-\sin\alpha),
\end{align*}

\begin{align*}
\frac{\partial g_3}{\partial v_1}&=\frac{\partial r_1}{\partial
v_1}-\frac{\partial r_2}{\partial v_1}=\frac{\partial r_1}{\partial
y_1}\frac{\partial y_1}{\partial v_1}-\frac{\partial r_2}{\partial
z_1}\frac{\partial z_1}{\partial v_1}
=\frac{1}{2}\left(\frac{\partial r_1}{\partial y_1}+\frac{\partial r_2}{\partial
z_1}\right)\\
&=\frac{1}{2}\bigg[\frac{y_1-x_1}{|x-y|}\arctan\left(\frac{y_2-x_2}{y_1-x_1}
\right)-
\frac{|x-y|}{y_1-x_1}\frac{y_2-x_2}{\sqrt{(y_1-x_1)^2+(y_2-x_2)^2}}\bigg]\\
&=\frac{1}{2}\left(\beta\sin\alpha\cos\beta-\frac{\sin\beta}{
\sin\alpha\cos\beta}\right),\\
\frac{\partial g_3}{\partial v_2}&=\frac{\partial r_1}{\partial
v_2}-\frac{\partial r_2}{\partial v_2}=\frac{\partial r_1}{\partial
y_2}\frac{\partial y_2}{\partial v_2}-\frac{\partial r_2}{\partial
z_2}\frac{\partial z_2}{\partial v_2}
=\frac{1}{2}\left(\frac{\partial r_1}{\partial y_2}+\frac{\partial r_2}{\partial
z_2}\right)\\
&=\frac{1}{2}\bigg[\frac{y_2-x_2}{|x-y|}\arctan\left(\frac{y_2-x_2}{y_1-x_1}
\right)+
\frac{|x-y|}{\sqrt{(y_1-x_1)^2+(y_2-x_2)^2}}\bigg]\\
&=\frac{1}{2}\left(\beta\sin\alpha\sin\beta+\frac{1}{\sin\alpha}\right),\\
\frac{\partial g_3}{\partial v_3}&=\frac{\partial r_1}{\partial
v_3}-\frac{\partial r_2}{\partial v_3}=\frac{\partial r_1}{\partial
y_3}\frac{\partial y_3}{\partial v_3}-\frac{\partial r_2}{\partial
z_3}\frac{\partial z_3}{\partial v_3}
=\frac{1}{2}\left(\frac{\partial r_1}{\partial y_3}+\frac{\partial r_2}{\partial
z_3}\right)\\
&=\frac{1}{2}\bigg[\frac{y_3-x_3}{|x-y|}\arctan\left(\frac{y_2-x_2}{y_1-x_1}
\right)\bigg]=\frac{1}{2}\beta\cos\alpha.
\end{align*}

Hence,
\begin{eqnarray*}
\left(\rho'_{11}(0,h,x)\right)^{-1}= \frac{1}{2}\begin{bmatrix}
\sin\alpha\cos\beta & \sin\alpha\sin\beta & \cos\alpha\\
\cos\beta(\alpha\sin\alpha+\cos\alpha)&\sin\beta(\alpha\sin\alpha+\cos\alpha)
&\alpha\cos\alpha-\sin\alpha\\
\beta\sin\alpha\cos\beta-\frac{\sin\beta}{\sin\alpha\cos\beta}&
\beta\sin\alpha\sin\beta+\frac{1}{\sin\alpha}&\beta\cos\alpha
\end{bmatrix},
\end{eqnarray*}
which gives 
\begin{eqnarray*}
\left(\left(\rho'_{11}(0,h,x)\right)^{-1}\right)^\top e_1=\frac{1}{2}
(\sin\alpha\cos\beta, \sin\alpha\sin\beta, \cos\alpha)^\top.
\end{eqnarray*}
Thus we obtain $a=(c_1+c_2)^2/4$ and $|{\rm
det}\rho'_{11}(0,h,x)|^{-1}=\frac{1}{8\sin\alpha\cos\beta}=\frac{1}{8}\frac{
|y-x|}{y_1-x_1}$. Next we focus on the computation of $H(0,h,x)$, which
requires to compute $|{\rm det}\tau'(0,h,x)|$ first. Recalling that
$\tau^{-1}:\mathbb R^9\rightarrow\mathbb R^9$ is given by
$\tau^{-1}(g,h,x)=(y,z,x)$, we have 
\begin{eqnarray*}
(\tau^{-1})'(g,h,x)=\begin{bmatrix}\partial_g y & \partial_h
y & \partial_x y\\ \partial_g z & \partial_h z & \partial_x z\\
\partial_g x & \partial_h x & \partial_x x
\end{bmatrix},\quad
((\tau^{-1})'(g,h,x))^{-1}=\begin{bmatrix}\partial_y g
& \partial_z g & \partial_x g \\
\partial_y h & \partial_z h & \partial_x h \\
\partial_y x & \partial_z x & \partial_x x
\end{bmatrix}.
\end{eqnarray*}
Now we calculate $\frac{\partial g}{\partial y}$. Noting $g=(t_1-t_2, s_1-s_2,
r_1-r_2)$, we obtain that 
\begin{eqnarray*}
\frac{\partial g}{\partial y}= \frac{1}{2}\begin{bmatrix} \sin\alpha\cos\beta &
\sin\alpha\sin\beta & \cos\alpha\\
\cos\beta(\alpha\sin\alpha+\cos\alpha)&\sin\beta(\alpha\sin\alpha+\cos\alpha)
&\alpha\cos\alpha-\sin\alpha\\
\beta\sin\alpha\cos\beta-\frac{\sin\beta}{\sin\alpha\cos\beta}&
\beta\sin\alpha\sin\beta+\frac{1}{\sin\alpha}&\beta\cos\alpha
\end{bmatrix}.
\end{eqnarray*}
When $g=0$, we have $y=z$. A simple calculation yields that $\frac{\partial
g}{\partial z}=-\frac{\partial g}{\partial y}$, $\frac{\partial h}{\partial
y}=\frac{\partial g}{\partial y}$, $\frac{\partial h}{\partial
z}=\frac{\partial g}{\partial y}$. It is clear to note that $\frac{\partial
x}{\partial y}=0$, $\frac{\partial x}{\partial z}=0$, $\frac{\partial
x}{\partial x}=I$. Hence
\[
{\rm det}(((\tau^{-1})'(0,h,x))^{-1})=\frac{1}{8}\frac{|y-x|^2}{(y_1-x_1)^2},
\]
which implies 
\[
{\rm det}((\tau^{-1})'(0,h,x))=8\frac{(y_1-x_1)^2}{|y-x|^2}.
\]
Substituting these results into (\ref{b38}) gives
\begin{align*}
R_3(x,\kappa)=&c_m\int_{\mathbb R^3}e^{{\rm i}(c_1-c_2)\kappa h\cdot
e_1}\frac{(x_1-y_1)^{m_1+n_1}(x_2-y_2)^{m_2+n_2}(x_3-y_3)^{m_3+n_3}}{|x-y|^{
p_1+p_2}}\\
&\quad\times\frac{(y_1-x_1)}{|y-x|}\phi(y(0,h, x)){\rm d}h,
\end{align*}
where $(y,z,x)=\tau^{-1}(g,h,x)$. Noting  ${\rm det}\left(\frac{\partial
h}{\partial y}\right)=\frac{1}{8}\frac{|x-y|}{y_1-x_1}$, we arrive at 
\begin{eqnarray*}
R_3(x,\kappa)=C_3\int_{\mathbb R^3}e^{{\rm
i}(c_1-c_2)|x-y|\kappa}\frac{(x_1-y_1)^{m_1+n_1}(x_2-y_2)^{m_2+n_2}(x_3-y_3)^{
m_3+n_3}}{|x-y|^{p_1+p_2}}\phi(y){\rm d}y,
\end{eqnarray*}
which completes the proof.
\end{proof}

Now we are ready to estimate the order of $\mathbb E(|\tilde{u}(x,\kappa)|^2)$.
Setting $\kappa_1=\kappa_2=\kappa$ in (\ref{b10}) and applying Lemma
\ref{lemma4b}, we obtain 
\begin{eqnarray}\label{b39}
\mathbb
E(|\tilde{u}(x,\kappa)|^2)=T^{(2)}_{\rm A}(x)\kappa^{-(m+1)}+O(\kappa^{-(m+2)}),
\end{eqnarray}
where
\begin{eqnarray}\label{b40}
T^{(2)}_{\rm A}(x)=\frac{C_2}{8\pi}\int_{\mathbb R^2}\frac{1}{|x-y|}\phi(y){\rm
d}y.
\end{eqnarray}

Before presenting the main result, we need the following lemma.

\begin{lemma}\label{lemma_o}
Let $V_1, V_2\subset\mathbb R^d$ be two open, bounded, and simply connected
domains with positive distance. For some positive integer $l$ and $\phi\in
C_0^{\infty}(V_1)$, define the integral
\begin{eqnarray*}
T(x)=\int_{\mathbb R^d}\frac{1}{|x-y|^l}\phi(y){\rm d}y,\quad x\in V_2.
\end{eqnarray*}
Then $T(x), x\in V_2$ uniquely determines the function $\phi$.
\end{lemma}

\begin{proof}
A simple calculation yields   
\begin{eqnarray*}
\Delta_x |x-y|^{-n}=n^2|x-y|^{-n-2}, \quad n\in \mathbb N,
\end{eqnarray*}
which implies 
\begin{eqnarray*}\label{b47}
\Delta^n T(x)=c_n\int_{V_1}\frac{1}{|x-y|^{l+2n}}\phi(y){\rm d}y,
\end{eqnarray*}
where $c_n$ is a constant depending on $n$. Since $T(x)$ is known in an open set
$V_2$ which has a positive distance to the support of $\phi\in
C_0^{\infty}(\mathbb R^2)$, so as $\Delta^n T(x), n\in \mathbb N$ is known in
the set $V_2$. A linear combination of $\Delta^n T(x)$ shows that the integral 
\begin{eqnarray}\label{b48}
\int_{V_1}\frac{1}{|x-y|^l}P(\frac{1}{|x-y|^2})\phi(y){\rm d}y
\end{eqnarray}
is known in the set $V_2$, where  $P(t)=\sum_{j=0}^{J}a_jt^j$ is a
polynomial of order $J\in \mathbb N$. In (\ref{b48}), by changing the integral
variables, we deduce
\begin{align*}
\int_{V_1}\frac{1}{|x-y|^l}P(\frac{1}{|x-y|^2})\phi(y){\rm d}y&=\int_{r_1}^{r_2}
\frac {1}{r^l}P(\frac{1}{r^2})\int_{|y-x|=r}\phi(y){\rm d}s(y){\rm d}r\\
&=\int_{r_1}^{r_2}\frac{1}{r^l}P(\frac{1}{r^2})S(x,r){\rm d}r\\
&=\frac{1}{2}\int_{\frac{1}{r_2^2}}^{\frac{1}{r_1^2}}P(t)S(x,\frac{1}{\sqrt{t}}
)t^{\frac{l-3}{2}}{ \rm d}t,
\end{align*}
where $S(x,r)=\int_{|y-x|=r}\phi(y){\rm d}s(y)$ denotes the integral of
$\phi(y)$ along the circle $|y-x|=r$, $r_1=\min_{y\in V_1}|x-y|$ and 
$r_2=\max_{y\in V_1}|x-y|$ denote the minimum and the maximum distance
between the fixed point $x\in V_2$ and the domain $V_1$, respectively. Due to
$\phi\in C_0^{\infty}(\mathbb R^2)$, the function $S(x,r)$ is continuous with
respect to $r$ and is compact supported in the interval $[r_1, r_2]$. We obtain
$S(x,\frac{1}{\sqrt{t}})t^{\frac{l-3}{2}}$ is continuous in
$[r_2^{-2}, r_1^{-2}]$. Note that the polynomial function
$P(t)$ is dense in $C([r_2^{-2}, r_1^{-2}])$, thus the function
$S(x,\frac{1}{\sqrt{t}})t^{\frac{l-3}{2}}$ is uniquely determined which implies
$S(x,r)$ is uniquely determined for all $r>0$.

Let $g(x)=e^{-\frac{|x|^2}{2}}$ for $x\in\mathbb R^2$, then we have
\begin{align}
(g\ast\phi)(x)&=\int_{\mathbb
R^2}e^{-\frac{|x-y|^2}{2}}\phi(y){\rm d}y=\int_{V_1}e^{
-\frac { |x-y|^2}{2}}\phi(y){\rm d}y\notag\\\label{b49}
&=\int_{r_1}^{r_2}e^{-\frac{r^2}{2}}\int_{|y-x|=r}\phi(y){\rm
d}y{\rm d}r=\int_{r_1}^{r_2} e^{-\frac{r^2}{2}}S(x,r){\rm d}r.
\end{align}
Since $S(x,r)$ is uniquely determined for all $r>0$, we can compute the
convolution $g\ast \phi$ by (\ref{b49}) for $x\in V_2$. Because $V_2$ is open
and $g\ast\phi$ is real analytic, hence $g\ast\phi$ is known everywhere, and
the Fourier transform $\mathcal{F}(g\ast\phi)$ is known everywhere. Since
\begin{align*}
\mathcal{F}g(\xi)&=\int_{\mathbb R^2}e^{-{\rm
i}x\cdot\xi}g(x){\rm d}x=\int_{\mathbb R^2}e^{-(\frac{|x|^2}{2}+{\rm
i}x\cdot\xi)}{\rm d}x\\
&=\int_{\mathbb R^2}e^{-\frac{1}{2}(x_1^2+2{\rm
i}x_1\xi_1)}e^{-\frac{1}{2}(x_2^2+2{\rm i}x_2\xi_2)}{\rm d}x\\
&=e^{-\frac{1}{2}|\xi|^2}\int_{\mathbb R}e^{-\frac{1}{2}(x_1+{\rm
i}\xi_1)^2}{\rm d}x_1\int_{\mathbb R}e^{-\frac{1}{2}(x_2+{\rm
i}\xi_2)^2}{\rm d}x_2\\
&=2\pi e^{-\frac{1}{2}|\xi|^2},
\end{align*}
we conclude $\mathcal{F}g$ is smooth and non-zero all over $\mathbb R^2$.
Therefore $\mathcal{F}{\phi}=\mathcal{F}(g\ast\phi)/\mathcal{F}g$ is
uniquely determined which shows that $\phi$ is uniquely determined. 
\end{proof}

We are in the position to present the main result for the time-harmonic
acoustic waves.  

\begin{theorem}\label{theorem5b}
Let the external source $f$ be a microlocally isotropic Gaussian random field
which satisfies Assumption B. Then for all $x\in U$, it holds almost
surely that
\begin{eqnarray}\label{b41}
\lim_{Q\rightarrow\infty}\frac{1}{Q-1}\int_1^Q\kappa^{m+1}|u(x,
\kappa)|^2{\rm d}\kappa=T^{(2)}_{\rm A}(x).
\end{eqnarray}
Moreover, the scattering data $T^{(2)}_{\rm A}(x)$, $x\in U$ uniquely determines
the micro-correlation strength $\phi$ through the linear relation (\ref{b40}).
\end{theorem}

\begin{proof}
A simple calculation shows that 
\begin{align*}
&\frac{1}{Q-1}\int_1^Q\kappa^{m+1}|u(x,\kappa)|^2{\rm d}\kappa\\
&=\frac{1}{Q-1}\int_1^Q\kappa^{m+1}|\tilde{u}(x,\kappa)+u(x,\kappa)-\tilde{u}(x
,\kappa)|^2{\rm d}\kappa\\
&=\frac{1}{Q-1}\int_1^Q\kappa^{m+1}|\tilde{u}(x,\kappa)|^2{\rm d}\kappa+
\frac{1}{Q-1}\int_1^Q\kappa^{m+1}|u(x,\kappa)-\tilde{u}(x,\kappa)|^2{
\rm d}\kappa\\
&\quad+\frac{2}{Q-1}\int_1^Q\kappa^{m+1}
\Re\left[\overline{\tilde{u}(x,\kappa)}
(u(x,\kappa)-\tilde{u}(x,\kappa))\right]{\rm d}\kappa.
\end{align*}
It is clear that (\ref{b41}) follows as long as we show that
\begin{align}\label{b42}
&\lim_{Q\rightarrow\infty}\frac{1}{Q-1}\int_1^Q\kappa^{m+1}|\tilde{u}(x,
\kappa)|^2{\rm d}\kappa=T^{(2)}_{\rm A}(x),\\\label{b43}
&\lim_{Q\rightarrow\infty}\frac{1}{Q-1}\int_1^Q\kappa^{m+1}|u(x,\kappa)-\tilde{
u}(x,\kappa)|^2{\rm d}\kappa=0,\\\label{b44}
&\lim_{Q\rightarrow\infty}\frac{2}{Q-1}\int_1^Q\kappa^{m+1}
\Re\left[\overline{\tilde{u}(x,\kappa)}(u(x,\kappa)-\tilde{u}(x,\kappa))\right]
{\rm d}\kappa=0.
\end{align}
To prove (\ref{b42}), we define
$Y(x,\kappa)=\kappa^{m+1}(|\tilde{u}(x,\kappa)|^2-\mathbb E
|\tilde{u}(x,\kappa)|^2 )$. Since
\begin{eqnarray*}
\int_1^Q\kappa^{m+1}|\tilde{u}(x,\kappa)|^2{\rm d}\kappa
=\int_1^Q\kappa^{m+1}\mathbb E|\tilde{u}(x,\kappa)|^2{\rm d}\kappa+
\int_1^Q Y(x,\kappa){\rm d}\kappa,
\end{eqnarray*}
(\ref{b42}) holds as long as we prove 
\begin{eqnarray*}\label{b45}
\lim_{Q\rightarrow\infty}\frac{1}{Q-1}\int_1^Q\kappa^{m+1}\mathbb
E|\tilde{u}(x,\kappa)|^2{\rm d}\kappa= T^{(2)}_{\rm A}(x), \quad
\lim_{Q\rightarrow\infty}\frac{1}{Q-1}\int_1^QY(x,\kappa){\rm d}\kappa=0.
\end{eqnarray*}
By (\ref{b39}), it is easy to see that 
\begin{eqnarray*}\label{b46}
\frac{1}{Q-1}\int_1^Q\kappa^{m+1}\mathbb E|\tilde{u}(x,\kappa)|^2{\rm d}\kappa
=\frac{1}{Q-1}\int_1^Q\left(T^{(2)}_{A}(x)+O(\kappa^{-1})\right){\rm d}\kappa.
\end{eqnarray*}
Clearly, we have 
\begin{eqnarray*}
\frac{1}{Q-1}\int_{1}^QT^{(2)}_{\rm A}(x){\rm d}\kappa=T^{(2)}_{\rm A}(x)
\end{eqnarray*}
and
\begin{eqnarray*}
\bigg|\frac{1}{Q-1}\int_1^Q O(\kappa^{-1}){\rm d}\kappa\bigg|
\lesssim \frac{1}{Q-1}\int_1^Q \kappa^{-1}{\rm d}\kappa
=\frac{\ln Q}{Q-1}\rightarrow 0\quad {\rm as}\;\;Q\rightarrow \infty.
\end{eqnarray*}
Hence,
\begin{eqnarray*}
\lim_{Q\rightarrow\infty}\frac{1}{Q-1}\int_1^Q\kappa^{m+1}\mathbb
E|\tilde{u}(x,\kappa)|^2{\rm d}\kappa=T^{(2)}_{A}(x).
\end{eqnarray*}

To prove (\ref{b42}), it suffices to show
\begin{eqnarray*}
\lim_{Q\rightarrow\infty}\frac{1}{Q-1}\int_1^QY(x,\kappa){\rm d}\kappa=0.
\end{eqnarray*}
By the definition of $Y(x,\kappa)$, we obtain
\begin{align*}
Y(x,\kappa)&= \kappa^{m+1}(|\tilde{u}(x,\kappa)|^2-\mathbb E
|\tilde{u}(x,\kappa)|^2 )\\
&=\kappa^{m+1}\bigg((\Re\tilde{u}(x,\kappa))^2-\mathbb
E(\Re\tilde{u}(x,\kappa))^2+(\Im\tilde{u}(x,\kappa))^2-\mathbb
E(\Im\tilde{u}(x,\kappa))^2\bigg).
\end{align*}
Therefore 
\begin{eqnarray*}
\mathbb E(Y(x,\kappa_1)Y(x,\kappa_2))= I_{A,1}+I_{A,2}+I_{A,3}+I_{A,4},
\end{eqnarray*}
where
\begin{align*}
I_{A,1}&= \kappa_1^{m+1} \kappa_2^{m+1}\mathbb
E\bigg[((\Re\tilde{u}(x,\kappa_1))^2-\mathbb
E(\Re\tilde{u}(x,\kappa_1))^2)((\Re\tilde{u}(x,\kappa_2))^2-\mathbb
E(\Re\tilde{u}(x,\kappa_2))^2)\bigg],\\
I_{A,2}&= \kappa_1^{m+1} \kappa_2^{m+1}\mathbb
E\bigg[((\Re\tilde{u}(x,\kappa_1))^2-\mathbb
E(\Re\tilde{u}(x,\kappa_1))^2)((\Im\tilde{u}(x,\kappa_2))^2-\mathbb
E(\Im\tilde{u}(x,\kappa_2))^2)\bigg],\\
 I_{A,3}&=\kappa_1^{m+1} \kappa_2^{m+1}\mathbb
E\bigg[((\Im\tilde{u}(x,\kappa_1))^2-\mathbb
E(\Im\tilde{u}(x,\kappa_1))^2)((\Re\tilde{u}(x,\kappa_2))^2-\mathbb
E(\Re\tilde{u}(x,\kappa_2))^2)\bigg],\\
 I_{A,4}&= \kappa_1^{m+1} \kappa_2^{m+1}\mathbb
E\bigg[((\Im\tilde{u}(x,\kappa_1))^2-\mathbb
E(\Im\tilde{u}(x,\kappa_1))^2)((\Im\tilde{u}(x,\kappa_2))^2-\mathbb
E(\Im\tilde{u}(x,\kappa_2))^2)\bigg].
\end{align*}

Combing the expression of $\tilde{u}(x,\kappa)$ and the assumption $\mathbb
E(f)=0$ gives that both $\Re \tilde{u}(x,\kappa)$ and $\Im\tilde{u}(x,\kappa)$
are zero-mean Gaussian random variables. Applying Lemmas \ref{lemma2a} and  
\ref{lemma3b} leads to
\begin{align*}
 I_{A,1}&=2\kappa_1^{m+1} \kappa_2^{m+1}[\mathbb E
(\Re\tilde{u}(x,\kappa_1)\Re\tilde{u}(x,\kappa_2))]^2\\
&=\frac{1}{2}\kappa_1^{m+1} \kappa_2^{m+1}\bigg[\mathbb
E\Big(\Re(\tilde{u}(x,\kappa_1)\tilde{u}(x,\kappa_2))+
\Re(\tilde{u}(x, \kappa_1)\overline{\tilde{u}(x,\kappa_2)})\Big)\bigg]^2\\
&\lesssim\bigg[\frac{\kappa_1^{\frac{m+1}{2}}\kappa_2^{\frac{m+1}{2}}}{
(\kappa_1+\kappa_2)^n(1+|\kappa_1-\kappa_2|)^{m}}+\frac{\kappa_1^{\frac{m+1}{2}}
\kappa_2^{\frac{m+1}{2}}}{(\kappa_1+\kappa_2)^{m+1}(1+|\kappa_1-\kappa_2|)^n}
\bigg]^2\\
&\lesssim \bigg[\frac{1}{(1+|\kappa_1-\kappa_2|)^{m}}+\frac{1}{
(1+|\kappa_1-\kappa_2|)^n}\bigg]^2.
\end{align*}
 We can obtain the same estimates for  $I_{A,2}$,  $I_{A,3}$, and $I_{A,4}$ by
the similar arguments. Thus, an application of Lemma \ref{lemma3a} gives
that 
\[
\lim_{Q\rightarrow\infty}\frac{1}{Q-1}\int_1^QY(x,\kappa){\rm d}\kappa=0.
\]
 
To prove (\ref{b43}), we obtain from Lemma \ref{lemma1b} that
\begin{align*}
&\bigg|\frac{1}{Q-1}\int_1^Q\kappa^{m+1}|u(x,\kappa)-\tilde{u}(x,
\kappa)|^2{\rm d}\kappa\bigg|\lesssim
\frac{1}{Q-1}\int_1^Q\kappa^{m+1}\kappa^{-7}{\rm d}\kappa\\
&=\frac{1}{Q-1}\int_1^Q\kappa^{m-6}{\rm d}\kappa
=\frac{1}{m-5}\frac{Q^{m-5}-1}{Q-1}\rightarrow 0 \quad{\rm
as}\;\;Q\rightarrow\infty.
\end{align*}
To prove (\ref{b44}), by the H\"{o}lder inequality, we have 
\begin{eqnarray*}
&&\bigg|\frac{2}{Q-1}\int_1^Q\kappa^{m+1}\Re\left[\overline{\tilde{u}(x,\kappa)}
(u(x,\kappa)-\tilde{u}(x,\kappa))\right]{\rm d}\kappa\bigg|\\
&&\leq\frac{2}{Q-1}\int_1^Q\kappa^{m+1}|\tilde{u}(x,\kappa)||u(x,\kappa)-\tilde{
u}(x, \kappa)|{\rm d}\kappa\\
&&\leq 2\bigg[\frac{1}{Q-1}\int_1^Q\kappa^{m+1}|\tilde{u}(x,
\kappa)|^2{\rm
d}\kappa\bigg]^{\frac{1}{2}}\bigg[\frac{1}{Q-1}\int_1^Q\kappa^{m+1} |u(x
,\kappa)-\tilde{u}(x, \kappa)|^2{\rm d}\kappa\bigg]^{\frac{1}{2}}\\
&&\rightarrow 2T^{(2)}_{\rm A}(x)^{\frac{1}{2}}\cdot 0 = 0\quad{\rm
as}\;\;Q\rightarrow\infty.
\end{eqnarray*}
Hence, (\ref{b42})--(\ref{b44}) hold which means that (\ref{b41}) holds. The
unique determination of $\phi$ by the scattering data $T_{\rm A}^{(2)}(x)$ for
$x\in U$ is a direct consequence of Lemma \ref{lemma_o}.
\end{proof}

\subsection{The three-dimensional case}

In this subsection, we show that the scattering data obtained from a single
realization of the random source can determine uniquely the function $\phi$ in
the principle symbol in the three dimensions. By
(\ref{r3}) and (\ref{r14}), we have 
\begin{eqnarray*}\label{c1}
u(x,\kappa)=-\frac{1}{4\pi}\int_{\mathbb R^3}\frac{e^{{\rm
i}\kappa|x-y|}}{|x-y|}f(y){\rm d}y,
\end{eqnarray*} 
which yields 
\begin{eqnarray}\label{c2}
\mathbb E(u(x,\kappa_1)\overline{u(x,\kappa_2)})=\frac{1}{16\pi^2}\int_{\mathbb
R^6}\frac{e^{{\rm i}(\kappa_1|x-y|-\kappa_2|x-z|)}}{|x-y||x-z|}\mathbb
E(f(y)f(z)){\rm d}y{\rm d}z.
\end{eqnarray}

We apply directly Lemma \ref{lemma2b} and obtain the estimates of $\mathbb
E(u(x,\kappa_1)\overline{u(x,\kappa_2)})$ and $\mathbb
E(u(x,\kappa_1)u(x,\kappa_2))$.

\begin{lemma}\label{lemma1c}
For $\kappa_1\geq 1, \kappa_2\geq 1$, the following estimates
\begin{align*}
|\mathbb E(u(x,\kappa_1)\overline{u(x,\kappa_2)})|&\leq
c_n(\kappa_1+\kappa_2)^{-m}(1+|\kappa_1-\kappa_2|)^{-n},\\
|\mathbb E(u(x,\kappa_1)u(x,\kappa_2))|&\leq
c_n(\kappa_1+\kappa_2)^{-n}(1+|\kappa_1-\kappa_2|)^{-m}
\end{align*}
holds uniformly for $x\in U$, where $n\in \mathbb N$ is arbitrary and $c_n>0$ is
a constant depending only on $n$.
\end{lemma}

To derive the relationship between the scattering data and the function $\phi$
in the principle symbol,  by setting $\kappa_1=\kappa_2=\kappa$ in (\ref{c2})
and using Lemma \ref{lemma4b}, we get 
\begin{eqnarray}\label{c5}
\mathbb E(|u(x,\kappa)|^2)=T^{(3)}_{\rm A}(x)\kappa^{-m}+O(\kappa^{-(m+1)}),
\end{eqnarray}
where
\begin{eqnarray}\label{c6}
T^{(3)}_{\rm A}(x)=\frac{C_3}{16\pi^2}\int_{\mathbb
R^3}\frac{1}{|x-y|^2}\phi(y){\rm d}y.
\end{eqnarray}

Now we are ready to present the main result for the three-dimensional case.

\begin{theorem}
Let the external source $f$ be a microlocally isotropic Gaussian random field
which satisfies Assumption B. Then for all $x\in U$, it holds almost
surely that 
\begin{eqnarray}\label{c7}
\lim_{Q\rightarrow\infty}\frac{1}{Q-1}\int_1^Q\kappa^{m}|u(x,
\kappa)|^2{\rm d}\kappa=T^{(3)}_{\rm A}(x).
\end{eqnarray}
Moreover, the scattering data $T^{(3)}_{\rm A}(x)$, $x\in U$ uniquely determines
the micro-correlation strength $\phi$ through the linear integral equation 
(\ref{c6}).
\end{theorem}

\begin{proof}
We decompose $\kappa^{m}|u(x,\kappa)|^2$ into two parts:
\[
\kappa^{m}|u(x,\kappa)|^2=\kappa^{m}\mathbb E |u(x,\kappa)|^2+Y(x,\kappa),
\]
where 
\[
Y(x,\kappa):=\kappa^{m}(|u(x,\kappa)|^2-\mathbb E |u(x,\kappa)|^2).
\]
Clearly,
\[
\frac{1}{Q-1}\int_1^Q\kappa^{m}|u(x,\kappa)|^2{\rm d}\kappa=\frac{1}{Q-1}
\int_1^Q\kappa^{m}\mathbb E|u(x,\kappa)|^2{\rm
d}\kappa+\frac{1}{Q-1}\int_1^QY(x, \kappa){\rm d}\kappa.
\]
Hence, (\ref{c7}) holds as long as we show that 
\begin{eqnarray}\label{c8}
\lim_{Q\rightarrow\infty}\frac{1}{Q-1}\int_1^Q\kappa^{m}\mathbb
E|u(x,\kappa)|^2{\rm d}\kappa=T^{(3)}_{\rm A}(x), \quad
\lim_{Q\rightarrow\infty}\frac{1}{Q-1}\int_1^QY(x,\kappa){\rm d}\kappa=0.
\end{eqnarray}
The second equation in (\ref{c8}) can be obtained by a similar argument to the
two-dimensional case. Using (\ref{c5}) gives 
\begin{eqnarray*}
\lim_{Q\rightarrow\infty}\frac{1}{Q-1}\int_1^Q\kappa^{m}\mathbb
E|u(x,\kappa)|^2{\rm d}\kappa=\lim_{Q\rightarrow\infty}\frac{1}{Q-1}\int_1^Q(T^{
(3)}_{\rm A}(x)+O(\kappa^{-1})){\rm d}\kappa=T^{(3)}_{\rm A}(x).
\end{eqnarray*}
Hence, the first equation in (\ref{c8}) holds which implies that
(\ref{c7}) holds. A direct application of Lemma (\ref{lemma_o}) implies that
$\phi$ is uniquely determined by the scattering data $T_{\rm A}^{(3)}(x)$ for
$x\in U$.
\end{proof}

\section{Elastic waves}

This section concerns the direct and inverse source problems for the elastic
wave equation in the two- and three-dimensional cases. Following the general
theme for the acoustic case presented in Section 3, we discuss the
well-posedness of the direct problem and show the uniqueness of the inverse
problem. We prove that the direct scattering problem with a distributional
source indeed has a unique solution. For the inverse problem, we assume that
each component of the external source is a microlocally isotropic Gaussian
random field whose covariance operator is a classical pseudo-differential
operator. Moreover, the principle symbol of the covariance operator of each
component is assumed to be coincided. Our main results are as follows: in either
the two- or three-dimensional case, given the scattering data which is obtained
from a single realization of the random source, the principle symbol of the
covariance operator can be uniquely determined. The technical details
differ from acoustic waves due to the different model equation and Green
tensors.

\subsection{The direct scattering problem}

In this subsection, we introduce the model problem of the random source
scattering for elastic waves, and show that the direct problem with a
distributional source is well-posed.

Consider the time-harmonic Navier equation in a homogeneous medium 
\begin{eqnarray}\label{s1}
\mu\Delta \bm{u}+(\lambda+\mu)\nabla\nabla\cdot
\bm{u}+\omega^2\bm{u}=\bm{f} \quad {\rm in} ~ \mathbb R^{d},
\end{eqnarray}
where $\omega>0$ is the angular frequency, $\lambda$ and $\mu$ are the
Lam${\rm\acute{e}}$ constants satisfying $\mu>0$ and $\lambda+\mu>0$, the
external source $\bm{f}\in\mathbb C^d$ is a generalized random function
supported in a bounded and simply connected domain $D$ in $\mathbb R^d$, and
$\bm{u}\in\mathbb C^{d}$ is the displacement of the random wave field.

Since the problem is imposed in the open domain $\mathbb R^d$, an appropriate
radiation condition is needed to complete the formulation of the scattering
problem. We adopt the Kupradze--Sommerfeld radiation condition to
describe the asymptotic behavior of the displacement field away from the source.
According to the Helmholtz decomposition, the displacement $\bm{u}$ can be
decomposed into the compressional part $\bm{u}_{\rm p}$ and the shear part
$\bm{u}_{\rm s}$:
\[
\bm{u}=\bm{u}_{\rm p}+\bm{u}_{\rm s}\quad{\rm in} ~ \mathbb
R^d\setminus \bar{D}.
\]
The Kupradze--Sommerfeld radiation condition requires that $\bm{u}_{\rm p}$
and $\bm{u}_{\rm s}$ satisfy the Sommerfeld radiation condition:
\begin{eqnarray}\label{s2}
\lim_{r\rightarrow\infty}r^{\frac{d-1}{2}}\left(\partial_r
\bm{u}_{\rm p} -{\rm i}\kappa_{\rm p}\bm{u}_{\rm p}\right)=0,\quad
\lim_{r\rightarrow\infty}r^{\frac{d-1}{2}}\left(\partial_r
\bm{u}_{\rm s} -{\rm i}\kappa_{\rm s}\bm{u}_s\right)=0,\quad r=|x|,
\end{eqnarray}
where 
\begin{eqnarray*}\label{s3}
\kappa_{\rm p}=\frac{\omega}{(\lambda+2\mu)^{1/2}}=c_{\rm p}\omega,\qquad
\kappa_{\rm s}=\frac{\omega}{\mu^{1/2}}=c_{\rm s}\omega
\end{eqnarray*}
are known as the compressional and shear wavenumbers with 
\begin{eqnarray*}\label{s4}
c_{\rm p} = (\lambda+2\mu)^{-1/2},\quad c_{\rm s} = \mu^{-1/2}.
\end{eqnarray*}
Note that $c_{\rm p}$ and $c_{\rm s}$ are independent of $\omega$ and
$c_{\rm p}<c_{\rm s}$.

In (\ref{s1}), the external source $\bm{f}$ is a vector with components $f_i
, i = 1, ..., d$. To achieve the main results,  throughout this section, we
assume that each component $f_i$ satisfies the following condition.

\textbf{Assumption C:} Recall that $D\subset \mathbb R^d$ denotes a bounded and
simply connected domain, $f_i$ is assumed to be a microlocally
isotropic Gaussian random field of the same order $m\in [d,d+\frac{1}{2})$ in
$D$. Each covariance operator $C_{f_i}$ is assumed to have the same principle
symbol $\phi(x)|\xi|^{-m}$ with $\phi\in C_0^{\infty}(D)$ and $\phi\geq 0$.
Moreover, we assume that $\mathbb E(f_i)=0$ and $\mathbb E(f_if_j)=0$ if
$i\neq j$ for $i,j=1,...,d.$

According to Lemma \ref{lemma1a}, if $m=d$, we have $\bm{f}(\hat{\omega})\in
H^{-\varepsilon,p}(D)^3$. Thus it suffices to show that the scattering problem
for such a deterministic, distributional source $\bm{f}\in
H^{-\varepsilon,p}(D)^3$ has a unique solution. 

Introduce the Green tensor $\bm{G}(x,y,\omega)\in\mathbb C^{d\times d}$ for
the Navier equation (\ref{s1}) which is given by
\begin{eqnarray}\label{s5}
\bm{G}(x,y,\omega)=\frac{1}{\mu}\Phi_d(x,y,\kappa_{\rm
s})\bm{I}_d+\frac{1}{\omega^2}\nabla_x\nabla_x^\top(\Phi_d(x,y,
\kappa_{\rm s})-\Phi_d(x,y,\kappa_{\rm p})),
\end{eqnarray}
where $\bm{I}_d$ is the $d\times d$ identity matrix and $\Phi_d(x,y,\kappa)$
is the fundamental solution for the $d$-dimensional Helmholtz
equation given in (\ref{r3}). Here the notation $\nabla_x\nabla_x^\top$ is
given by 
\begin{eqnarray*}
\nabla_x\nabla_x^\top\varphi=\begin{bmatrix} \partial^2_{x_1x_1}\varphi &
\partial^2_{x_1x_2}\varphi\\[3pt] \partial^2_{x_2x_1}\varphi &
\partial^2_{x_2x_2}\varphi
\end{bmatrix}\quad\text{ if } d=2
\end{eqnarray*}
and 
\begin{eqnarray*}
\nabla_x\nabla_x^\top\varphi =\begin{bmatrix}\partial^2_{x_1x_1}\varphi
&\partial^2_{x_1x_2}\varphi &\partial^2_{x_1x_3}\varphi\\[3pt]
\partial^2_{x_2x_1}\varphi&
\partial^2_{x_2x_2}\varphi & \partial^2_{x_2x_3}\varphi\\[3pt]
\partial^2_{x_3x_1}\varphi &
\partial^2_{x_3x_2}\varphi & \partial^2_{x_3x_3}\varphi \end{bmatrix}\quad
\text{if } d=3
\end{eqnarray*} 
for some scalar function $\varphi$ defined in $\mathbb R^d$. It is easily
verified that the Green tensor $\bm{G}(x,y,\omega)$ is symmetric with
respect to the variables $x$ and $y$.

We study the asymptotic expansion of the Green's tensor
$\bm{G}(x,y,\omega)$ when $|x-y|$ is close to zero. For the two-dimensional
case, using (\ref{r4})--(\ref{r6}) gives 
\begin{eqnarray}\label{s8}
H_2^{(1)}(t)=-\frac{4{\rm i}}{\pi}\frac{1}{t^2}-\frac{\rm i}{\pi}+\frac{\rm
i}{4\pi}t^2\ln\frac{t}{2}+(\frac{\gamma {\rm i}}{4\pi}-\frac{3{\rm
i}}{16\pi}+\frac{1}{8})t^2+O(t^4\ln\frac{t}{2})\quad\text{as } t\to 0.
\end{eqnarray}
Recall the recurrence relations (\ref{r9}), a direct
calculation shows that
\begin{align*}
\Phi_2(x,y,\kappa)&=\frac{\rm i}{4}H_0^{(1)}(\kappa|x-y|),\\
\partial_{x_i}\Phi_2(x,y,\kappa) &= -\frac{\kappa {\rm
i}}{4}(x_i-y_i)\frac{H_1^{(1)}(\kappa|x-y|)}{|x-y|},\\ 
\partial^2_{x_ix_j}\Phi_2(x,y,\kappa) &= -\frac{\kappa {\rm
i}}{4}\frac{H_1^{(1)}(\kappa|x-y|)}{|x-y|}\delta_{ij}+\frac{\kappa^2{\rm
i}}{4}(x_i-y_i)(x_j-y_j)\frac{H_2^{(1)}(\kappa |x-y|)}{|x-y|^2},
\end{align*}
where $\delta_{ij}$ is the Kronecker delta function. Hence, by (\ref{r8}) and
(\ref{s8}), we have 
\begin{align*}
&\kappa_{\rm s}H_1^{(1)}(\kappa_{\rm s}|x-y|)-\kappa_{\rm
p}H_1^{(1)}(\kappa_{\rm p}|x-y|)=\frac{\rm
i}{\pi}|x-y|\left(\kappa_{\rm s}^2\ln\frac{\kappa_{\rm s}|x-y|}{2}-\kappa_{\rm
p}^2\ln\frac{\kappa_{\rm p}|x-y|}{2}\right)\\
&\hspace{2cm}+(\frac{1}{2}+\frac{{\rm i}}{\pi}\gamma-\frac{\rm
i}{2\pi})(\kappa_{\rm s}^2-\kappa_{\rm
p}^2)|x-y|+O(|x-y|^3\ln\frac{|x-y|}{2}),\\
&\kappa_{\rm s}^2 H_2^{(1)}(\kappa_{\rm s}|x-y|)-\kappa_{\rm p}^2
H_2^{(1)} (\kappa_{\rm p}|x-y|)=\frac{\rm
i}{4\pi}\left(\kappa_{\rm s}^4\ln\frac{\kappa_{\rm s}|x-y|}{2} -\kappa_{\rm
p}^4\ln\frac{\kappa_{\rm p}|x-y|}{2}\right)|x-y|^2\\
&\hspace{2cm}-\frac{\rm i}{\pi}(\kappa_{\rm s}^2-\kappa_{\rm
p}^2)+(\frac{\gamma{\rm i}}{4\pi}-\frac{3{\rm i}}{16\pi}+\frac{1}{8}
)(\kappa_{\rm s}^4-\kappa_{\rm p}^4)|x-y|^2+O(|x-y|^4\ln\frac{|x-y|}{2}),
\end{align*}
which gives
\begin{align}\label{s12}
&\partial^2_{x_i x_j}[\Phi_2(x,y,\kappa_{\rm
s})-\Phi_2(x,y,\kappa_{\rm p})]\notag\\
&=-\frac{\rm
i}{4}\frac{1}{|x-y|}[\kappa_{\rm s}H_1^{(1)}(\kappa_{\rm
s}|x-y|)-\kappa_{\rm p}H_1^{(1)}(\kappa_{\rm p}|x-y|)]\delta_{ij}\notag\\
&\quad+\frac{\rm i}{4}\frac{(x_i-y_i)(x_j-y_j)}{|x-y|^2}[\kappa_{\rm
s}^2 H_2^{(1)}(\kappa_{\rm
s}|x-y|)-\kappa_{\rm p}^2 H_2^{(1)}(\kappa_{\rm p}|x-y|)]\notag\\
&=\frac{1}{4\pi}\left(\kappa_{\rm s}^2\ln\frac{\kappa_{\rm s}|x-y|}{2}
-\kappa_{\rm p}^2\ln\frac{\kappa_{\rm
p}|x-y|}{2}\right)\delta_{ij}+\frac{1}{4\pi} (\kappa_{\rm
s}^2-\kappa_{\rm p}^2)\frac{(x_i-y_i)(x_j-y_j)}{|x-y|^2}\notag\\
&\quad-\frac{\rm i}{8}(1+\frac{2{\rm i}}{\pi}\gamma-\frac{\rm
i}{\pi})(\kappa_{\rm s}^2-\kappa_{\rm p}^2)\delta_{ij}
-\frac{1}{16\pi}(x_i-y_i)(x_j-y_j)\left(\kappa_{\rm
s}^4\ln\frac{\kappa_{\rm s}|x-y|}{2}
-\kappa_{\rm p}^4\ln\frac{\kappa_{\rm p}|x-y|}{2}\right)\notag\\
&\quad-\frac{\rm i}{4}\left(\frac{\gamma{\rm
i}}{4\pi}-\frac{3{\rm i}}{16\pi}+\frac{1}{8}
\right)(\kappa_{\rm
s}^4-\kappa_{\rm p}^4)(x_i-y_i)(x_j-y_j)+O(|x-y|^2\ln\frac{|x-y|}{2}).
\end{align}

For the three-dimensional case, it follows from direct calculations
that 
\begin{align*}
\partial_{x_i} \Phi_3(x,y,\kappa)=&\frac{(x_i-y_i)}{4\pi|x-y|^3}e^{{\rm
i}\kappa |x-y|}({\rm i}\kappa  |x-y|-1),\\
\partial^2_{x_ix_j} \Phi_3(x,y,\kappa)=&\frac{|x-y|^2\delta_{ij}
-3(x_i-y_i)(x_j-y_j)}{4\pi|x-y|^5}
e^{{\rm i}\kappa|x-y|}({\rm i}\kappa |x-y|-1)\\
&\quad -\kappa^2\frac{(x_i-y_i)(x_j-y_j)}{4\pi|x-y|^3}e^{{\rm i}\kappa |x-y|},
\end{align*}
which lead to 
\begin{align}\label{s15}
&\partial^2_{x_ix_j}(\Phi_3(x,y,\kappa_{\rm
s})-\Phi_3(x,y,\kappa_{\rm p}))\notag\\
&=\frac{|x-y|^2\delta_{ij}-3(x_i-y_i)(x_j-y_j)}{4\pi|x-y|^5}(e^{{\rm
i}\kappa_{\rm s}|x-y|}({\rm i}\kappa_{\rm s}|x-y|-1)-e^{{\rm
i}\kappa_{\rm p}|x-y|}({\rm i}\kappa_{\rm p} |x-y|-1))\notag\\
&\quad-\frac{(x_i-y_i)(x_j-y_j)}{4\pi|x-y|^3}(\kappa_{\rm s}^2 e^{{\rm
i}\kappa_{\rm s}|x-y|}-\kappa_{\rm p}^2 e^{{\rm i}\kappa_{\rm p}|x-y|}).
\end{align}

\begin{lemma}\label{lemma_y}
For some fixed $x\in\mathbb R^d$, $\bm{G}(x,\cdot,\omega)\in (L^2_{\rm
loc}(\mathbb R^d)\cap H^{1,p}_{\rm loc}(\mathbb R^d))^{d\times d}$, where $p\in
(1,2)$ for $d=2$ and $p\in (1,\frac{3}{2})$ for $d=3$.
\end{lemma}

\begin{proof}
For any fixed $x\in \mathbb R^d$, we choose a bounded domain $V\subset\mathbb
R^d$ which contains $x$. Define $\rho:=\sup_{y\in V}|x-y|$, then we have
$V\subset B_{\rho}(x)$. For the two-dimensional case, from (\ref{r7}) and
(\ref{s12}), it is sufficient to show that 
\begin{eqnarray*}
\ln\frac{|x-y|}{2}\in L^2(V), \quad \frac{y_i-x_i}{|x-y|^2}\in
L^p(V),\quad\forall\, p\in(1,2),
\end{eqnarray*}
which are proved in Lemma \ref{lemma_z}. For the three-dimensional case, it
follows from the expansion of the exponential function $e^t$ that 
\begin{align*}
&\kappa_{\rm s}^2 e^{{\rm i}\kappa_{\rm s}|x-y|}-\kappa_{\rm p}^2 e^{{\rm
i}\kappa_{\rm p}|x-y|}=(\kappa_{\rm s}^2-\kappa_{\rm p}^2)+O(|x-y|),\\
&e^{{\rm i}\kappa_s|x-y|}({\rm i}\kappa_s |x-y|-1)-e^{{\rm
i}\kappa_p|x-y|}({\rm i}\kappa_p
|x-y|-1)=\frac{1}{2}(\kappa_p^2-\kappa_s^2)|x-y|^2+O(|x-y|^3).
\end{align*} 
 Thus, by (\ref{r12}) and (\ref{s15}), it is sufficient to prove 
 \begin{eqnarray*}
\frac{1}{|x-y|}\in L^2(V), \quad \frac{y_i-x_i}{|x-y|^3}\in
L^p(V),\quad\forall\,p\in(1,\frac{3}{2}),
\end{eqnarray*}
which can been similarly proved to the three-dimensional case in
Lemma \ref{lemma_z}.
\end{proof}

Let $V$ and $G$ be any two bounded domains in $\mathbb R^d$. By Lemma
\ref{lemma_y} and the Sobolev embedding theorem, we have
$\bm{G}(x,\cdot,\omega)\in (H^{s}(V))^{d\times d}$, where $s\in (0,1)$ for
$d=2$ and $s\in (0,\frac{1}{2})$ for $d=3$. Hence, for any given $\bm{g}\in
H_0^{-s}(V)^{d}$, in the dual sense, we define the operator
$\bm{H}_{\omega}$ by
\begin{eqnarray*}
(\bm{H}_{\omega}\bm{g})(x)=\int_{V}\bm{G}(x,y,\omega)\cdot
\bm{g}(y){\rm d}y,\quad x\in G,
\end{eqnarray*}
where the dot is the matrix-vector multiplication. By the similar arguments
to \cite[Theorem 8.2]{CK}, we have the following property.

\begin{lemma}\label{lemma_x}
The operator $\bm{H}_{\omega}: H_0^{-s}(V)^d\rightarrow H^s(G)^d$ is
bounded for $s\in (0,1), d=2$ or $s\in (0,\frac{1}{2}), d=3$.
\end{lemma}

\begin{theorem}
For some fixed $s\in(0,1-\frac{d}{6})$, assume $1<p<\frac{2d}{d+2(1-s)}$ and
$\frac{1}{p}+\frac{1}{p'}=1$, then the scattering problem (\ref{s1})--(\ref{s2})
with the source $\bm{f}\in H^{-1,p'}_0(D)^d$ attains a unique solution
$\bm{u}\in H^{1,p}_{\rm loc}(\mathbb R^d)^d$ given by
\begin{eqnarray}\label{s20}
\bm{u}(x,\omega)=-\int_{D}\bm{G}(x,y,\omega)\cdot \bm{f}(y){\rm d}y.
\end{eqnarray}
\end{theorem}

\begin{proof}
The uniqueness of the scattering problem (\ref{s1})--(\ref{s2}) is obvious. We
focus on the existence.  For convenience, we denote the differential
operator in the Navier equation by
\begin{eqnarray*}
\Delta^{\ast}\bm{u}:=\mu\Delta \bm{u}+(\lambda+\mu)\nabla\nabla\cdot
\bm{u}.
\end{eqnarray*}
Let $B_r:=\{y\in\mathbb R^d:|y|<r\}$ be a ball which is large enough such that
it contains the support of $\bm{f}$. Denote by  $\nu$ the unit normal
vector on the boundary $\partial B_r$. The generalized
stress vector on $\partial B_r$ is defined by
\begin{eqnarray*}
P\bm{u}:=\mu\partial_\nu \bm{u}+(\lambda+\mu)(\nabla\cdot 
\bm{u})\nu.
\end{eqnarray*}

The assumption $s\in(0,1-\frac{d}{6})$ leads to
$1<\frac{2d}{d+2(1-s)}<\frac{3}{2}$. By Lemma \ref{lemma_y}, we have $
\bm{G}(x,\cdot,\omega)\in (H_{\rm loc}^{1,p}(\mathbb R^d))^{d\times d}$.
Since $\Delta^{\ast}  \bm{u}+\omega^2 \bm{u}= \bm{f}\in
H^{-1,p'}_0(D)^d$, in the sense of distributions, we have
\begin{eqnarray*}
\int_{B_r} \bm{G}(x,y,\omega)\cdot(\Delta^{\ast}  \bm{u}(y)+\omega^2
\bm{u}(y)){\rm d}y=\int_{B_r} \bm{G}(x,y,\omega)\cdot 
\bm{f}(y){\rm d}y.
\end{eqnarray*}
Define the operator acting on $ \bm{u}$ in the left-hand side of the above
equation by $S_{\rm E}$. For $\bm{\varphi}\in C^{\infty}(\mathbb R^d)^d$, from
the divergence theorem, we obtain
\begin{align*}
&(S_{\rm E} \bm{\varphi})(x)=\int_{B_r}\bm{G}(x,y,\omega)\cdot(\Delta^{\ast}
\bm{\varphi}(y)+\omega^2\bm{\varphi}(y)){\rm d}y\\
 &= \int_{B_r\setminus B_{\delta}(x)}\bm{G}(x,y,\omega)\cdot (\Delta^{\ast} 
\bm{\varphi}(y)+\omega^2 \bm{\varphi}(y)){\rm d}y+
\int_{B_{\delta}(x)}\bm{G}(x,y,\omega)\cdot(\Delta^{\ast} 
\bm{\varphi}(y)+\omega^2 \bm{\varphi}(y)){\rm d}y\\
 &= \int_{B_r\setminus B_{\delta}(x)}(\bm{G}(x,y,\omega)\cdot\Delta^{\ast} 
\bm{\varphi}(y)-\Delta^{\ast}\bm{G}(x,y,\omega)\cdot\bm{\varphi}(y)){\rm d}y\\
 &\quad+ \int_{B_{\delta}(x)}\bm{G}(x,y,\omega)\cdot(\Delta^{\ast} 
\bm{\varphi}(y)+\omega^2 \bm{\varphi}(y)){\rm d}y\\
 &=\int_{B_{\delta}(x)}\bm{G}(x,y,\omega)\cdot(\Delta^{\ast} 
\bm{\varphi}(y)+\omega^2 \bm{\varphi}(y)){\rm d}y\\
 &\quad+\int_{\partial B_r\cup\partial
B_{\delta}(x)}\left(\bm{G}(x,y,\omega)\cdot P\bm{\varphi}(y)- P
\bm{G}(x,y,\omega)\cdot\bm{\varphi}(y)\right){\rm d}s(y),
\end{align*}
where $\delta>0$ is a sufficiently small constant. By the mean value theorem,
it is easy to verify that 
 \begin{eqnarray*}
 \lim_{\delta\rightarrow 0}\int_{\partial
B_{\delta}(x)}\left(\bm{G}(x,y,\omega)\cdot P\bm{\varphi}(y)-P
\bm{G}(x,y,\omega)\cdot\bm{\varphi}(y)\right){\rm d}s(y)=-\bm{\varphi}(x)
 \end{eqnarray*}
 and 
 \begin{eqnarray*}
 \lim_{\delta\rightarrow
0}\int_{B_{\delta}(x)}\bm{G}(x,y,\omega)\cdot(\Delta^{\ast}
\bm{\varphi}(y)+\omega^2 \bm{\varphi}(y)){\rm d}y=0.
 \end{eqnarray*}
 Hence, we obtain
 \begin{eqnarray*}
 (\bm{S}_{\rm E}\bm{\varphi})(x)=-\bm{\varphi}(x)+\int_{\partial
B_r}\left(\bm{G}(x,y,\omega)\cdot P\bm{\varphi}(y)-
P\bm{G}(x,y,\omega)\cdot\bm{\varphi}(y)\right){\rm d}s(y),
 \end{eqnarray*}
 which implies 
  \begin{eqnarray*}
 (\bm{S}_{\rm E}\bm{u})(x)=-\bm{u}(x)+\int_{\partial
B_r}\left(\bm{G}(x,y,\omega)\cdot
P\bm{u}(y)-P \bm{G}(x,y,\omega)\cdot \bm{u}(y)\right){\rm d}s(y).
 \end{eqnarray*}
Since $\bm{u}(y)$ and $\bm{G}(x,y,\omega)$ satisfy the Sommerfeld radiation
condition, we have 
\begin{eqnarray*}
\lim_{r\rightarrow\infty}\int_{\partial B_r}\left(\bm{G}(x,y,\omega)\cdot
\bm{P}\bm{u}(y)- \bm{P}\bm{G}(x,y,\omega)\cdot \bm{u}(y)\right){\rm d}s(y)=0.
\end{eqnarray*}
Therefore
\begin{eqnarray*}
\bm{u}(x,\omega)=-\int_{D}\bm{G}(x,y,\omega)\cdot
\bm{f}(y){\rm d}y=-\bm{H}_{\omega}\bm{f}(x).
\end{eqnarray*}

Next is to show that $\bm{u}\in H^{1,p}_{\rm
loc}(\mathbb R^d)^d$. By Lemma \ref{lemma_x}, we have that for $s\in
(0,1-\frac{d}{6})$, the operator $\bm{H}_{\omega}: H_0^{-s}(D)^d\rightarrow
H^{s}_{\rm loc}(\mathbb R^d)^d$ is bounded. The assumption
$1<p<\frac{2d}{d+2(1-s)}$ implies that $\frac{1}{2}+\frac{1-s}{d}<\frac{1}{p}<1$
which yields $\frac{1}{2}-\frac{s}{d}<\frac{1}{p}-\frac{1}{d}$. Thus, the
Sobolev embedding theorem implies that $H^{s}(D)$ is embedded into $H^{1, p}(D)$
and $H_0^{-1, p'}(D)$ is embedded into $H_0^{-s}(D)$. Thus, the operator
$\bm{H}_{\omega}: H_0^{-1, p'}(D)^d\rightarrow H^{1,p}_{\rm loc}(\mathbb
R^d)^d$ is bounded, which completes the proof. 
\end{proof}

\subsection{The two-dimensional case}

This subsection is devoted to study the two-dimensional case. It is required to
derive a relationship between the scattering data and the principle symbol of
the covariance operator of the component of $\bm{f}$. To this end, we need to
express the displacement $\bm{u}(x,\omega)$ explicitly. Substituting (\ref{s12})
into (\ref{s5}) gives that $\bm{u}(x,\omega)=(u_1(x,\omega),u_2(x,\omega))^\top$
where 
\begin{align*}
u_1(x,\omega)&=u_{11}(x,\omega)+u_{12}(x,\omega)+u_{13}(x,\omega)+u_{14}(x,
\omega),\\
u_2(x,\omega)&=u_{21}(x,\omega)+u_{22}(x,\omega)+u_{23}(x,\omega)+u_{24}(x,
\omega),
\end{align*}
where
\begin{align*}
&u_{11}(x,\omega) = \frac{{\rm i}}{4\mu}\int_{D}H_0^{(1)}(\kappa_{\rm
s}|x-y|)f_1(y){\rm d}y,\\
&u_{12}(x,\omega) = \frac{{\rm i}}{4\omega^2}\int_{D}\big[-\kappa_{\rm
s}H_1^{(1)}(\kappa_{\rm s}|x-y|)+\kappa_{\rm
p}H_1^{(1)}(\kappa_{\rm p}|x-y|)\big]\frac{1}{|x-y|}f_1(y){\rm d}y,\\
&u_{13}(x,\omega) =
\frac{{\rm i}}{4\omega^2}\int_{D}\big[\kappa_{\rm s}^2 H_2^{(1)}
(\kappa_{\rm s}|x-y|)-\kappa_{\rm p}^2 H_2^{(1)
}(\kappa_{\rm p}|x-y|)\big]\frac{(x_1-y_1)^2}{|x-y|^2}f_1(y){\rm d}y,\\
&u_{14}(x,\omega) = \frac{{\rm
i}}{4\omega^2}\int_{D}\big[\kappa_{\rm
s}^2 H_2^{(1)}(\kappa_{\rm s}|x-y|)-\kappa_{\rm p}^2
H_2^{(1)}(\kappa_{\rm p}|x-y|)\big]\frac{(x_1-y_1)(x_2-y_2)}{|x-y|^2}f_2(y){\rm
d}y,
\end{align*}
and
\begin{align*}
&u_{21}(x,\omega) = \frac{{\rm
i}}{4\mu}\int_{D}H_0^{(1)}(\kappa_{\rm s}|x-y|)f_2(y){\rm d}y,\\
&u_{22}(x,\omega) = \frac{{\rm
i}}{4\omega^2}\int_{D}\big[-\kappa_{\rm
s}H_1^{(1)}(\kappa_{\rm s}|x-y|)+\kappa_{\rm p}H_1^{(1)}
(\kappa_{\rm p}|x-y|)\big]\frac{1}{|x-y|}f_2(y){\rm d}y,\\
&u_{23}(x,\omega) = \frac{{\rm i}}{4\omega^2}\int_{D}\big[\kappa_{\rm
s}^2 H_2^{(1)}(\kappa_{\rm s}|x-y|)-\kappa_{\rm p}^2
H_2^{(1)}(\kappa_{\rm p}|x-y|)\big]\frac{(x_2-y_2)^2}{|x-y|^2}f_2(y){\rm d}y,\\
&u_{24}(x,\omega) = \frac{{\rm i}}{4\omega^2}\int_{D}\big[\kappa_{\rm
s}^2 H_2^{(1)}(\kappa_{\rm s}|x-y|)-\kappa_{\rm p}^2
H_2^{(1)}(\kappa_{\rm p}|x-y|)\big]\frac{(x_1-y_1)(x_2-y_2)}{|x-y|^2}f_1(y){\rm
d}y.
\end{align*}

To prove the main result, we need to establish the asymptotic of
$\bm{u}(x,\omega)$ for $\omega\rightarrow\infty$. Recalling the definition of
$H_{n, N}^{(1)}$ given in (\ref{b7}), we define
$\tilde{\bm{u}}(x,\omega)=(\tilde{u}_1(x,\omega),\tilde{u}_2(x,\omega))^\top$,
where
\begin{align*}
\tilde{u}_1(x,\omega)&=\tilde{u}_{11}(x,\omega)+\tilde{u}_{12}(x,\omega)+\tilde{
u}_{13}(x,\omega)+\tilde{u}_{14}(x,\omega),\\
\tilde{u}_2(x,\omega)&=\tilde{u}_{21}(x,\omega)+\tilde{u}_{22}(x,\omega)+\tilde{
u}_{23}(x,\omega)+\tilde{u}_{24}(x,\omega).
\end{align*}
Here
\begin{align*}
&\tilde{u}_{11}(x,\omega) = \frac{{\rm
i}}{4\mu}\int_{D}H_{0,2}^{(1)}(\kappa_{\rm s}|x-y|)f_1(y){\rm d}y,\\
&\tilde{u}_{12}(x,\omega) = \frac{{\rm
i}}{4\omega^2}\int_{D}\big[-\kappa_{\rm
s}H_{1,3}^{(1)}(\kappa_{\rm s}|x-y|)+\kappa_{\rm p}H_{1, 3}
^{(1)}(\kappa_{\rm p}|x-y|)\big]\frac{1}{|x-y|}f_1(y){\rm d}y,\\
&\tilde{u}_{13}(x,\omega) = \frac{{\rm
i}}{4\omega^2}\int_{D}\big[\kappa_{\rm s}^2 H_{2,4}^{(1)}
(\kappa_{\rm s}|x-y|)-\kappa_{\rm p}^2
H_{2, 4}^{(1)}(\kappa_{\rm p}|x-y|)\big]\frac{(x_1-y_1)^2}{|x-y|^2}f_1(y){\rm
d}y,\\
&\tilde{u}_{14}(x,\omega) = \frac{{\rm
i}}{4\omega^2}\int_{D}\big[\kappa_{\rm s}^2 H_{2,4}^{(1)}
(\kappa_{\rm s}|x-y|)-\kappa_{\rm p}^2
H_{2, 4}^{(1)}(\kappa_{\rm
p}|x-y|)\big]\frac{(x_1-y_1)(x_2-y_2)}{|x-y|^2}f_2(y){\rm d}y,
\end{align*}
and
\begin{align*}
&\tilde{u}_{21}(x,\omega) = \frac{{\rm
i}}{4\mu}\int_{D}H_{0,2}^{(1)}(\kappa_{\rm s}|x-y|)f_2(y){\rm d}y,\\
&\tilde{u}_{22}(x,\omega) = \frac{{\rm
i}}{4\omega^2}\int_{D}\big[-\kappa_{\rm
s}H_{1,3}^{(1)}(\kappa_{\rm s}|x-y|)+\kappa_{\rm p}H_{1, 3}
^{(1)}(\kappa_{\rm p}|x-y|)\big]\frac{1}{|x-y|}f_2(y){\rm d}y,\\
&\tilde{u}_{23}(x,\omega) = \frac{{\rm
i}}{4\omega^2}\int_{D}\big[\kappa_{\rm s}^2 H_{2,4}^{(1)}
(\kappa_{\rm s}|x-y|)-\kappa_{\rm p}^2 H_{2
,4}^{(1)}(\kappa_{\rm p}|x-y|)\big]\frac{(x_2-y_2)^2}{|x-y|^2}f_2(y){\rm d}y,\\
&\tilde{u}_{24}(x,\omega) = \frac{{\rm
i}}{4\omega^2}\int_{D}\big[\kappa_{\rm s}^2 H_{2,4}^{(1)}
(\kappa_{\rm s}|x-y|)-\kappa_{\rm p}^2
H_{2, 4}^{(1)}(\kappa_{\rm
p}|x-y|)\big]\frac{(x_1-y_1)(x_2-y_2)}{|x-y|^2}f_1(y){\rm d}y.
\end{align*}

\begin{lemma}\label{lemma1d}
The random variable $ \bm{u}(x,\omega)-\tilde{\bm{u}}(x,\omega)$ satisfies
almost surely the condition
\begin{eqnarray*}
|\bm{u}(x,\omega)-\tilde{\bm{u}}(x,\omega)|\leq c \omega^{-\frac{7}{2}}, \quad
x\in U, \quad \omega>0,
\end{eqnarray*}
where the constant $c$ depends only on $H_0^{-1,p'}(D)^2$-norm of
$\bm{f}$.
\end{lemma}

\begin{proof}
By Assumption A, it is known that there exists a positive constant $M$
such that $|x-y|\geq M$ holds for any $x\in U$ and $y\in D$. By (\ref{b8}), for
$x\in U$, we have
\begin{align*}
|u_{11}(x,\omega)-\tilde{u}_{11}(x,\omega)|&=\left|\frac{\rm
i}{4\mu}\int_{D}\left[H_0^{(1)}(\kappa_{\rm s} |x-y|)-H_{0,2}^{(1)}(\kappa_{\rm
s}|x-y|)\right]f_1(y){\rm d}y\right|\\
&\lesssim
\|H_0^{(1)}(\kappa_s|x-\cdot|)-H_{0,2}^{(1)}(\kappa_s|x-\cdot|)\|_{H^{1,p}(D)}
\|f_1\|_{H_0^{-1,p'}(D)}\\
&\leq c\omega^{-\frac{7}{2}},
\end{align*}
where the constant $c$ depends only on $H_0^{-1,p'}(D)$-norm of $f_1$.

Similarly, it is easy to verify that 
\[
|u_{ij}(x,\omega)-\tilde{u}_{ij}(x,\omega)|\leq c\omega^{-\frac{7}{2}}\quad
\text{ for } i=1,2,\,j=1,2,3,4,
\]
where the constant $c$ depends only on $H_0^{-1,p'}(D)^2$-norm of
$\bm{f}$. Therefore
\begin{eqnarray*}
|\bm{u}(x,\omega)-\tilde{\bm{u}}(x,\omega)|\leq
\sum_{i=1}^{2}\sum_{j=1}^{4}|u_{ij}(x,\omega)-\tilde{u}_{ij}(x,\omega)|\leq
c\omega^{-\frac{7}{2}},
\end{eqnarray*}
which completes the proof. 
\end{proof}

To derive the relationship between the scattering data and the function in the
principle symbol, we need to estimate $\mathbb E(\bm{u}(x,\omega_1)\cdot
\overline{\bm{u}(x,\omega_2)})$ for $\omega_1\geq 1, \omega_2\geq1$.
By Lemma \ref{lemma1d}, it reduces to find the estimate of $\mathbb
E(\tilde{\bm{u}}(x,\omega_1)\cdot \overline{\tilde{\bm{u}}(x,\omega_2)})$ for
$\omega_1\geq 1, \omega_2\geq1$.
Recalling $\tilde{\bm{u}}(x,\omega)=(\tilde{u}_1(x,\omega),\tilde{u}_2(x,
\omega))^\top$ and (\ref{b7}), we have
\begin{align}\label{d16}
\mathbb E(\tilde{\bm{u}}(x,\omega_1)\cdot \overline{\tilde{\bm{u}}(x,\omega_2)})
&=\mathbb E(\tilde{u}_1(x,\omega_1)\overline{\tilde{u}_1(x,\omega_2)})+\mathbb
E(\tilde{u}_2(x,\omega_1)\overline{\tilde{u}_2(x,\omega_2)})\notag\\
&=\sum_{i,j=1}^4\bigg(\mathbb
E(\tilde{u}_{1i}(x,\omega_1)\overline{\tilde{u}_{1j}(x,\omega_2)})+\mathbb
E(\tilde{u}_{2i}(x,\omega_1)\overline{\tilde{u}_{2j}(x,\omega_2)})\bigg)
\end{align}
where
\begin{align*}
\tilde{u}_{11}(x,\omega)&=
\frac{{\rm i}}{4\mu}\int_{D}H_{0,2}^{(1)}(\kappa_{\rm s}|x-y|)f_1(y){\rm d}y\\
&= \frac{{\rm i}}{4\mu}\int_{D}\kappa_{\rm s}^{-\frac{1}{2}}
|x-y|^{-\frac{1}{2}}e^{{\rm i}(\kappa_{\rm s} |x-y|-\frac{1}{4}\pi)}\sum_{j=0}^2
a_j^{(0)}\bigg(\frac{1}{\kappa_{\rm s} |x-y|}\bigg)^jf_1(y){\rm d}y,\\
\tilde{u}_{12}(x,\omega) &=\frac{{\rm
i}}{4\omega^2}\int_{D}\big[-\kappa_{\rm
s}H_{1,3}^{(1)}(\kappa_{\rm s}|x-y|)+\kappa_{\rm p}H_{1, 3}
^{(1)}(\kappa_{\rm p}|x-y|)\big]\frac{1}{|x-y|}f_1(y){\rm d}y\\
&=\frac{{\rm
i}}{4\omega^2}\int_{D}-\kappa_{\rm s}^{\frac{1}{2}}|x-y|^{-\frac{3}{2}}e^{{\rm
i}(\kappa_{\rm s}|x-y|-\frac{3}{4}\pi)}\sum_{j=0}^3 a_j^{(1)}\bigg(\frac{1}{
\kappa_{\rm s}|x-y|}\bigg)^j f_1(y){\rm d}y\\
&\quad+\frac{{\rm
i}}{4\omega^2}\int_{D}\kappa_{\rm p}^{\frac{1}{2}}|x-y|^{-\frac{3}{2}}e^{{\rm
i}(\kappa_{\rm p}|x-y|-\frac{3}{4}\pi)}\sum_{j=0}^3 a_j^{(1)}\bigg(\frac{1}{
\kappa_{\rm p}|x-y|}\bigg)^j f_1(y){\rm d}y,\\
\tilde{u}_{13}(x,\omega) &= \frac{{\rm
i}}{4\omega^2}\int_{D}\big[\kappa_{\rm s}^2 H_{2,4}^{(1)}
(\kappa_{\rm s}|x-y|)-\kappa_{\rm p}^2
H_{2, 4}^{(1)}(\kappa_{\rm p}|x-y|)\big]\frac{(x_1-y_1)^2}{|x-y|^2}f_1(y){\rm
d}y\\
&=\frac{{\rm
i}}{4\omega^2}\int_{D}\kappa_{\rm s}^{\frac{3}{2}}|x-y|^{-\frac{5}{2}}
(x_1-y_1)^2e^{{\rm i}(\kappa_{\rm s}|x-y|-\frac{5}{4}\pi)}\sum_{j=0}^4 a_j^{(2)}
\bigg(\frac{1}{
\kappa_{\rm s}|x-y|}\bigg)^j f_1(y){\rm d}y\\
&\quad-\frac{{\rm
i}}{4\omega^2}\int_{D}\kappa_{\rm p}^{\frac{3}{2}}|x-y|^{-\frac{5}{2}}
(x_1-y_1)^2e^{{\rm i}(\kappa_{\rm p}|x-y|-\frac{5}{4}\pi)}\sum_{j=0}^4 a_j^{(2)}
\bigg(\frac{1}{\kappa_{\rm p}|x-y|}\bigg)^j f_1(y){\rm d}y,\\
\tilde{u}_{14}(x,\omega) &= \frac{{\rm
i}}{4\omega^2}\int_{D}\big[\kappa_{\rm s}^2 H_{2,4}^{(1)}
(\kappa_{\rm s}|x-y|)-\kappa_{\rm p}^2
H_{2, 4}^{(1)}(\kappa_{\rm
p}|x-y|)\big]\frac{(x_1-y_1)(x_2-y_2)}{|x-y|^2}f_2(y){\rm d}y\\
&=\frac{{\rm
i}}{4\omega^2}\int_{D}\kappa_{\rm s}^{\frac{3}{2}}|x-y|^{-\frac{5}{2}}
(x_1-y_1)(x_2-y_2)e^{{\rm
i}(\kappa_{\rm s}|x-y|-\frac{5}{4}\pi)}\sum_{j=0}^4 a_j^{(2)}\bigg(\frac{1}{
\kappa_{\rm s}|x-y|}\bigg)^j f_2(y){\rm d}y\\
&\quad-\frac{{\rm
i}}{4\omega^2}\int_{D}\kappa_{\rm p}^{\frac{3}{2}}|x-y|^{-\frac{5}{2}}
(x_1-y_1)(x_2-y_2)e^{{\rm
i}(\kappa_{\rm p}|x-y|-\frac{5}{4}\pi)}\sum_{j=0}^4 a_j^{(2)}\bigg(\frac{1}{
\kappa_{\rm p}|x-y|}\bigg)^j f_2(y){\rm d}y.
\end{align*}
Using the assumption $\mathbb E(f_1 f_2)=0$, we obtain
\begin{align*}
&\mathbb E(\tilde{u}_{11}(x,\omega_1)\overline{\tilde{u}_{14}(x,\omega_2)}
)=0, \quad
\mathbb E(\tilde{u}_{12}(x,\omega_1)\overline{\tilde{u}_{14}(x,\omega_2)}
)=0,\\
&\mathbb
E(\tilde{u}_{13}(x,\omega_1)\overline{\tilde{u}_{14}(x,\omega_2)})=0, \quad
\mathbb
E(\tilde{u}_{14}(x,\omega_1)\overline{\tilde{u}_{11}(x,\omega_2)})=0,\\
&\mathbb
E(\tilde{u}_{14}(x,\omega_1)\overline{\tilde{u}_{12}(x,\omega_2)})=0, \quad
\mathbb E(\tilde{u}_{14}(x,\omega_1)\overline{\tilde{u}_{13}(x,\omega_2)})=0,
\end{align*}
and
\begin{align*}
&\mathbb E(\tilde{u}_{11}(x,\omega_1)\overline{\tilde{u}_{11}(x,\omega_2)})\\
&=\frac{1}{16\mu^2}\sum_{j_1,j_2=0}^2\frac{a^{(0)}_{j_1}\overline{a^{(0)}_{j_2}
}}{c_{\rm
s}^{j_1+j_2+1}\omega_1^{j_1+\frac{1}{2}}\omega_2^{j_1+\frac{1}{2}}}\int_{
\mathbb R^4}\frac{e^{{\rm
i}(c_{\rm s}\omega_1|x-y|-c_{\rm
s}\omega_2|x-z|)}}{|x-y|^{j_1+\frac{1}{2}}|x-z|^{ j_2+\frac {
1}{2}}}\mathbb E(f_1(y)f_1(z)){\rm d}y{\rm d}z,
\end{align*}
\begin{align*}
&\mathbb E(\tilde{u}_{11}(x,\omega_1)\overline{\tilde{u}_{12}(x,\omega_2)})
=\frac{e^{\frac{\pi}{2}{\rm
i}}}{16\mu}\sum_{j_1=0}^2\sum_{j_2=0}^3\frac{a_{j_1}^{(0)}\overline{a_{j_2}^{(1)
}}}{\omega_1^{j_1+\frac{1}{2}}\omega_2^{j_2+\frac{3}{2}}}\\
&\quad\times\int_{\mathbb R^4}\bigg[-\frac{e^{{\rm
i}(c_{\rm s}\omega_1|x-y|-c_{\rm
s}\omega_2|x-z|)}}{c_{\rm s}^{j_1+j_2}}+\frac{e^{{\rm
i}(c_{\rm s}\omega_1|x-y|-c_{\rm p}\omega_2|x-z|)}}{c_{\rm
s}^{j_1+\frac{1}{2}}c_{\rm p}^{j_2-\frac{1}{2}}}\bigg]\frac{\mathbb
E(f_1(y)f_1(z))}{|x-y|^{j_1+\frac{1}{2}}|x-z|^{j_2+\frac{3}{2}}}{\rm d}y{\rm
d}z,
\end{align*}
\begin{align*}
&\mathbb E(\tilde{u}_{11}(x,\omega_1)\overline{\tilde{u}_{13}(x,\omega_2)})
=\frac{e^{\pi{\rm
i}}}{16\mu}\sum_{j_1=0}^2\sum_{j_2=0}^4\frac{a_{j_1}^{(0)}\overline{a_{j_2}^{(2)
}}}{\omega_1^{j_1+\frac{1}{2}}\omega_2^{j_2+\frac{1}{2}}}\\
&\quad\times\int_{\mathbb R^4}\bigg[\frac{e^{{\rm
i}(c_{\rm s}\omega_1|x-y|-c_{\rm
s}\omega_2|x-z|)}}{c_{\rm s}^{j_1+j_2-1}}-\frac{e^{{\rm
i}(c_{\rm s}\omega_1|x-y|-c_{\rm
p}\omega_2|x-z|)}}{c_{\rm s}^{j_1+\frac{1}{2}}c_{\rm p}^{j_2-\frac{3}{2
}}}\bigg]\frac{(x_1-z_1)^2\mathbb
E(f_1(y)f_1(z))}{|x-y|^{j_1+\frac{1}{2}}|x-z|^{j_2+\frac{5}{2}}}{\rm d}y{\rm
d}z,
\end{align*}
\begin{align*}
&\mathbb E(\tilde{u}_{12}(x,\omega_1)\overline{\tilde{u}_{11}(x,\omega_2)})
=\frac{e^{-\frac{\pi}{2}{\rm
i}}}{16\mu}\sum_{j_1=0}^3\sum_{j_2=0}^2\frac{a_{j_1}^{(1)}\overline{a_{j_2}^{(0)
}}}{\omega_1^{j_1+\frac{3}{2}}\omega_2^{j_2+\frac{1}{2}}}\\
&\quad\times\int_{\mathbb R^4}\bigg[-\frac{e^{{\rm
i}(c_{\rm s}\omega_1|x-y|-c_{\rm
s}\omega_2|x-z|)}}{c_{\rm s}^{j_1+j_2}}+\frac{e^{{\rm
i}(c_{\rm p}\omega_1|x-y|-c_{\rm
s}\omega_2|x-z|)}}{c_{\rm p}^{j_1-\frac{1}{2}}c_{\rm s}^{j_2+\frac{1}{2
}}}\bigg]\frac{\mathbb
E(f_1(y)f_1(z))}{|x-y|^{j_1+\frac{3}{2}}|x-z|^{j_2+\frac{1 }{2}}}{\rm d}y{\rm
d}z,
\end{align*}
\begin{align*}
&\mathbb E(\tilde{u}_{12}(x,\omega_1)\overline{\tilde{u}_{12}(x,\omega_2)})
=\frac{1}{16}\sum_{j_1,j_2=0}^3\frac{a_{j_1}^{(1)}\overline{a_{j_2}^{(1)}}}{
\omega_1^{j_1+\frac{3}{2}}\omega_2^{j_2+\frac{3}{2}}}\\
&\quad\times\int_{\mathbb R^4}\bigg[\frac{e^{{\rm
i}(c_{\rm s}\omega_1|x-y|-c_{\rm
s}\omega_2|x-z|)}}{c_{\rm s}^{j_1+j_2-1}}+\frac{e^{{\rm
i}(c_{\rm p}\omega_1|x-y|-c_{\rm p}\omega_2|x-z|)}}{c_{\rm p}^{j_1+j_2-1}}\\
&\qquad\qquad-\frac{e^{{\rm
i}(c_{\rm s}\omega_1|x-y|-c_{\rm
p}\omega_2|x-z|)}}{c_{\rm s}^{j_1-\frac{1}{2}}c_{\rm p}^{j_2-\frac{1}{2
}}}-\frac{e^{{\rm
i}(c_{\rm p}\omega_1|x-y|-c_{\rm s}\omega_2|x-z|)}}{c_{\rm p}^{j_1-\frac{1}{2}}
c_{\rm s}^{j_2-\frac{1}{2 }}}\bigg]\frac{\mathbb
E(f_1(y)f_1(z))}{|x-y|^{j_1+\frac{3}{2}}|x-z|^{j_2+\frac{3}{2}}}{\rm d}y{\rm
d}z,\\
\end{align*}
\begin{align*}
&\mathbb E(\tilde{u}_{12}(x,\omega_1)\overline{\tilde{u}_{13}(x,\omega_2)})
=\frac{e^{\frac{\pi}{2}{\rm
i}}}{16}\sum_{j_1=0}^3\sum_{j_2=0}^4\frac{a_{j_1}^{(1)}\overline{a_{j_2}^{(2)}}}
{\omega_1^{j_1+\frac{3}{2}}\omega_2^{j_2+\frac{1}{2}}}\\
&\quad\times\int_{\mathbb R^4}\bigg[-\frac{e^{{\rm
i}(c_{\rm s}\omega_1|x-y|-c_{\rm
s}\omega_2|x-z|)}}{c_{\rm s}^{j_1+j_2-2}}-\frac{e^{{\rm
i}(c_{\rm p}\omega_1|x-y|-c_{\rm p}\omega_2|x-z|)}}{c_{\rm p}^{j_1+j_2-2}}\\
&\qquad\qquad+\frac{e^{{\rm
i}(c_{\rm s}\omega_1|x-y|-c_{\rm
p}\omega_2|x-z|)}}{c_{\rm s}^{j_1-\frac{1}{2}}c_{\rm p}^{j_2-\frac{ 3 } { 2
}}}+\frac{e^{{\rm
i}(c_{\rm p}\omega_1|x-y|-c_{\rm s}\omega_2|x-z|)}}{c_{\rm p}^{j_1-\frac{1}{2}}
c_{\rm s}^{j_2-\frac{3}{2
}}}\bigg]\frac{(x_1-z_1)^2\mathbb
E(f_1(y)f_1(z))}{|x-y|^{j_1+\frac{3}{2}}|x-z|^{j_2+\frac{5}{2}}}{\rm d}y{\rm
d}z,\\
\end{align*}
\begin{align*}
&\mathbb E(\tilde{u}_{13}(x,\omega_1)\overline{\tilde{u}_{11}(x,\omega_2)})
=\frac{e^{-\pi{\rm
i}}}{16\mu}\sum_{j_1=0}^4\sum_{j_2=0}^2\frac{a_{j_1}^{(2)}\overline{a_{j_2}^{(0)
}}}{\omega_1^{j_1+\frac{1}{2}}\omega_2^{j_2+\frac{1}{2}}}\\
&\quad\times\int_{\mathbb R^4}\bigg[\frac{e^{{\rm
i}(c_{\rm s}\omega_1|x-y|-c_{\rm
s}\omega_2|x-z|)}}{c_{\rm s}^{j_1+j_2-1}}-\frac{e^{{\rm
i}(c_{\rm p}\omega_1|x-y|-c_{\rm s}\omega_2|x-z|)}}{c_{\rm
p}^{j_1-\frac{3}{2}}c_{\rm s}^{j_2+\frac{ 1 } { 2
}}}\bigg]\frac{(x_1-y_1)^2\mathbb
E(f_1(y)f_1(z))}{|x-y|^{j_1+\frac{5}{2}}|x-z|^{j_2+\frac{1}{2}}}{\rm d}y{\rm
d}z,
\end{align*}
\begin{align*}
&\mathbb E(\tilde{u}_{13}(x,\omega_1)\overline{\tilde{u}_{12}(x,\omega_2)})
=\frac{e^{-\frac{\pi}{2}{\rm
i}}}{16}\sum_{j_1=0}^4\sum_{j_2=0}^3\frac{a_{j_1}^{(2)}\overline{a_{j_2}^{(1)}}}
{\omega_1^{j_1+\frac{1}{2}}\omega_2^{j_2+\frac{3}{2}}}\\
&\quad\times\int_{\mathbb R^4}\bigg[-\frac{e^{{\rm
i}(c_{\rm s}\omega_1|x-y|-c_{\rm
s}\omega_2|x-z|)}}{c_{\rm s}^{j_1+j_2-2}}-\frac{e^{{\rm
i}(c_{\rm p}\omega_1|x-y|-c_{\rm p}\omega_2|x-z|)}}{c_{\rm p}^{j_1+j_2-2}}\\
&\qquad\qquad+\frac{e^{{\rm
i}(c_{\rm s}\omega_1|x-y|-c_{\rm
p}\omega_2|x-z|)}}{c_{\rm s}^{j_1-\frac{3}{2}}c_{\rm p}^{j_2-\frac{ 1 } { 2
}}}+\frac{e^{{\rm
i}(c_{\rm
p}\omega_1|x-y|-c_{\rm
s}\omega_2|x-z|)}}{c_{\rm p}^{j_1-\frac{3}{2}}c_{\rm s}^{j_2-\frac{ 1 } { 2
}}}\bigg]\frac{(x_1-y_1)^2\mathbb
E(f_1(y)f_1(z))}{|x-y|^{j_1+\frac{5}{2}}|x-z|^{ j_2+\frac{3}{2}}}{\rm d}y{\rm
d}z,
\end{align*}
\begin{align*}
&\mathbb E(\tilde{u}_{13}(x,\omega_1)\overline{\tilde{u}_{13}(x,\omega_2)})
=\frac{1}{16}\sum_{j_1,j_2=0}^4\frac{a_{j_1}^{(2)}\overline{a_{j_2}^{(2)}}}{
\omega_1^{j_1+\frac{1}{2}}\omega_2^{j_2+\frac{1}{2}}}\\
&\quad\times\int_{\mathbb R^4}\bigg[\frac{e^{{\rm
i}(c_{\rm s}\omega_1|x-y|-c_{\rm
s}\omega_2|x-z|)}}{c_{\rm s}^{j_1+j_2-3}}+\frac{e^{{\rm
i}(c_{\rm p}\omega_1|x-y|-c_{\rm p}\omega_2|x-z|)}}{c_{\rm p}^{j_1+j_2-3}}\\
&\qquad-\frac{e^{{\rm i}(c_{\rm s}\omega_1|x-y|-c_{\rm
p}\omega_2|x-z|)}}{c_{\rm s}^{j_1-\frac{3}{2}}c_{\rm p}^{j_2-\frac{ 3 } { 2
}}}-\frac{e^{{\rm
i}(c_{\rm p}\omega_1|x-y|-c_{\rm s}\omega_2|x-z|)}}{c_{\rm p}^{j_1-\frac{3}{2}}
c_{\rm s}^{j_2-\frac{3}{2
}}}\bigg]\frac{(x_1-y_1)^2(x_1-z_1)^2\mathbb
E(f_1(y)f_1(z))}{|x-y|^{j_1+\frac{5}{2}}|x-z|^{j_2+\frac{5}{2}}}{\rm d}y{\rm
d}z,
\end{align*}
\begin{align*}
&\mathbb E(\tilde{u}_{14}(x,\omega_1)\overline{\tilde{u}_{14}(x,\omega_2)})
=\frac{1}{16}\sum_{j_1,j_2=0}^4\frac{a_{j_1}^{(2)}\overline{a_{j_2}^{(2)}}}{
\omega_1^{j_1+\frac{1}{2}}\omega_2^{j_2+\frac{1}{2}}}\\
&\quad\times\int_{\mathbb R^4}\bigg[\frac{e^{{\rm
i}(c_{\rm s}\omega_1|x-y|-c_{\rm
s}\omega_2|x-z|)}}{c_{\rm s}^{j_1+j_2-3}}+\frac{e^{{\rm
i}(c_{\rm p}\omega_1|x-y|-c_{\rm
p}\omega_2|x-z|)}}{c_{\rm p}^{j_1+j_2-3}}-\frac{e^{{\rm
i}(c_{\rm s}\omega_1|x-y|-c_{\rm
p}\omega_2|x-z|)}}{c_{\rm s}^{j_1-\frac{3}{2}}c_{\rm p}^{j_2-\frac{ 3 } { 2
}}}\\
&\qquad\qquad-\frac{e^{{\rm
i}(c_{\rm
p}\omega_1|x-y|-c_{\rm
s}\omega_2|x-z|)}}{c_{\rm p}^{j_1-\frac{3}{2}}c_{\rm s}^{j_2-\frac{ 3 } { 2
}}}\bigg]\frac{(x_1-y_1)(x_2-y_2)(x_1-z_1)(x_2-z_2)}{|x-y|^{j_1+\frac{5}{2}}
|x-z|^{j_2+\frac{5}{2}}}\mathbb E(f_2(y)f_2(z)){\rm d}y{\rm d}z,
\end{align*}

For the second component $\tilde{u}_2(x,\omega)$,  we have from (\ref{b7}) that 
\begin{align*}
\tilde{u}_{21}(x,\omega)&= \frac{{\rm
i}}{4\mu}\int_{D}H_{0,2}^{(1)}(\kappa_{\rm s}|x-y|)f_2(y)dy\\
&= \frac{{\rm i}}{4\mu}\int_{D}\kappa_{\rm s}^{-\frac{1}{2}}
|x-y|^{-\frac{1}{2}}e^{{\rm i}(\kappa_{\rm s} |x-y|-\frac{1}{4}\pi)}\sum_{j=0}^2
a_j^{(0)}\bigg(\frac{1}{\kappa_{\rm s} |x-y|}\bigg)^jf_2(y){\rm d}y,
\end{align*}
\begin{align*}
\tilde{u}_{22}(x,\omega) &= \frac{{\rm
i}}{4\omega^2}\int_{D}\big[-\kappa_{\rm
s} H_{1,3}^{(1)}(\kappa_{\rm s}|x-y|)+\kappa_{\rm p} H_{1, 3}
^{(1)}(\kappa_{\rm p}|x-y|)\big]\frac{1}{|x-y|}f_2(y){\rm d}y\\
&=\frac{{\rm
i}}{4\omega^2}\int_{D}-\kappa_{\rm s}^{\frac{1}{2}}|x-y|^{-\frac{3}{2}}e^{{\rm
i}(\kappa_{\rm s}|x-y|-\frac{3}{4}\pi)}\sum_{j=0}^3 a_j^{(1)}\bigg(\frac{1}{
\kappa_{\rm s}|x-y|}\bigg)^jf_2(y){\rm d}y\\
&\quad+\frac{{\rm
i}}{4\omega^2}\int_{D}\kappa_{\rm p}^{\frac{1}{2}}|x-y|^{-\frac{3}{2} }e^{{\rm
i}(\kappa_{\rm p}|x-y|-\frac{3}{4}\pi)}\sum_{j=0}^3 a_j^{(1)}\bigg(\frac{1}{
\kappa_{\rm p}|x-y|}\bigg)^j f_2(y){\rm d}y,
\end{align*}
\begin{align*}
\tilde{u}_{23}(x,\omega) &= \frac{{\rm
i}}{4\omega^2}\int_{D}\big[\kappa_{\rm s}^2 H_{2,4}^{(1)}
(\kappa_{\rm s}|x-y|)-\kappa_{\rm p}^2 H_{2
, 4}^{(1)}(\kappa_{\rm p}|x-y|)\big]\frac{(x_2-y_2)^2}{|x-y|^2}f_2(y){\rm d}y\\
&=\frac{{\rm
i}}{4\omega^2}\int_{D}\kappa_{\rm s}^{\frac{3}{2}}|x-y|^{-\frac{5}{2}}
(x_2-y_2)^2 e^{{\rm i}(\kappa_{\rm s}|x-y|-\frac{5}{4}\pi)}\sum_{j=0}^4
a_j^{(2)} \bigg(\frac{1}{\kappa_{\rm s}|x-y|}\bigg)^j f_2(y){\rm d}y\\
&\quad-\frac{{\rm
i}}{4\omega^2}\int_{D}\kappa_{\rm p}^{\frac{3}{2}}|x-y|^{-\frac{5}{2}}
(x_2-y_2)^2e^{{ \rm
i}(\kappa_{\rm p}|x-y|-\frac{5}{4}\pi)}\sum_{j=0}^4 a_j^{(2)}\bigg(\frac{1}{
\kappa_{\rm p}|x-y|}\bigg)^j f_2(y){\rm d}y,
\end{align*}
\begin{align*}
\tilde{u}_{24}(x,\omega) &= \frac{{\rm
i}}{4\omega^2}\int_{D}\big[\kappa_{\rm s}^2 H_{2,4}^{(1)}
(\kappa_{\rm s}|x-y|)-\kappa_{\rm p}^2
H_{2, 4}^{(1)}(\kappa_{\rm
p}|x-y|)\big]\frac{(x_1-y_1)(x_2-y_2)}{|x-y|^2}f_1(y){\rm d}y\\
&=\frac{{\rm
i}}{4\omega^2}\int_{D}\kappa_{\rm s}^{\frac{3}{2}}|x-y|^{-\frac{5}{2}}
(x_1-y_1)(x_2-y_2)e^{{\rm
i}(\kappa_{\rm s}|x-y|-\frac{5}{4}\pi)}\sum_{j=0}^4 a_j^{(2)}\bigg(\frac{1}{
\kappa_{\rm s}|x-y|}\bigg)^j f_1(y){\rm d}y\\
&\quad-\frac{{\rm
i}}{4\omega^2}\int_{D}\kappa_{\rm p}^{\frac{3}{2}}|x-y|^{-\frac{5}{2}}
(x_1-y_1)(x_2-y_2)e^{{\rm
i}(\kappa_{\rm p}|x-y|-\frac{5}{4}\pi)}\sum_{j=0}^4 a_j^{(2)}\bigg(\frac{1}{
\kappa_{\rm p}|x-y|}\bigg)^jf_1(y){\rm d}y.
\end{align*}
Noting $\mathbb E(f_1f_2)=0$ from Assumption C, we deduce that
\begin{align*}
&\mathbb
E(\tilde{u}_{21}(x,\omega_1)\overline{\tilde{u}_{24}(x,\omega_2)})=0, \quad
\mathbb
E(\tilde{u}_{22}(x,\omega_1)\overline{\tilde{u}_{24}(x,\omega_2)})=0,\\
&\mathbb
E(\tilde{u}_{23}(x,\omega_1)\overline{\tilde{u}_{24}(x,\omega_2)})=0, \quad
\mathbb
E(\tilde{u}_{24}(x,\omega_1)\overline{\tilde{u}_{21}(x,\omega_2)})=0,\\
&\mathbb
E(\tilde{u}_{24}(x,\omega_1)\overline{\tilde{u}_{22}(x,\omega_2)})=0, \quad
\mathbb E(\tilde{u}_{24}(x,\omega_1)\overline{\tilde{u}_{23}(x,\omega_2)})=0,
\end{align*}
and
\begin{align*}
&\mathbb E(\tilde{u}_{21}(x,\omega_1)\overline{\tilde{u}_{21}(x,\omega_2)})\\
&=\frac{1}{16\mu^2}\sum_{j_1,j_2=0}^2\frac{a^{(0)}_{j_1}\overline{a^{(0)}_{j_2}
}}{c_{\rm s}^{j_1+j_2+1}\omega_1^{j_1+\frac{1}{2}}\omega_2^{j_1+\frac{1}{2}}}
\int_{\mathbb R^4}\frac{e^{{\rm i}(c_{\rm s}\omega_1|x-y|-c_{\rm
s}\omega_2|x-z|)}}{|x-y|^{j_1+\frac{1}{2}}|x-z|^{j_2+\frac{
1}{2}}}\mathbb E(f_2(y)f_2(z)){\rm d}y{\rm d}z,
\end{align*}
\begin{align*}
&\mathbb E(\tilde{u}_{21}(x,\omega_1)\overline{\tilde{u}_{22}(x,\omega_2)})
=\frac{e^{\frac{\pi}{2}{\rm
i}}}{16\mu}\sum_{j_1=0}^2\sum_{j_2=0}^3\frac{a_{j_1}^{(0)}\overline{a_{j_2}^{(1)
}}}{\omega_1^{j_1+\frac{1}{2}}\omega_2^{j_2+\frac{3}{2}}}\\
&\quad\times\int_{\mathbb R^4}\bigg[-\frac{e^{{\rm
i}(c_{\rm s}\omega_1|x-y|-c_{\rm s}\omega_2|x-z|)}}{c_{\rm
s}^{j_1+j_2}}+\frac{e^{{\rm i}(c_{\rm s}\omega_1|x-y|-c_{\rm
p}\omega_2|x-z|)}}{c_{\rm s}^{j_1+\frac{1}{2}}c_{\rm
p}^{j_2-\frac{1}{2}}}\bigg]\frac{\mathbb
E(f_2(y)f_2(z))}{|x-y|^{j_1+\frac{1}{2}}|x-z|^{j_2+\frac{3}{2}}}{\rm d}y{\rm
d}z,
\end{align*}
\begin{align*}
&\mathbb E(\tilde{u}_{21}(x,\omega_1)\overline{\tilde{u}_{23}(x,\omega_2)})
=\frac{e^{\pi{\rm i}}}{16\mu}\sum_{j_1=0}^2\sum_{j_2=0}^4\frac{a_{j_1}^{(0)}
\overline{a_{j_2}^{(2) }}}{\omega_1^{j_1+\frac{1}{2}}\omega_2^{j_2+\frac{1}{2}}}
\\
&\quad\times\int_{\mathbb R^4}\bigg[\frac{e^{{\rm i}(c_{\rm
s}\omega_1|x-y|-c_{\rm s}\omega_2|x-z|)}}{c_{\rm s}^{j_1+j_2-1}}-\frac{e^{{\rm
i}(c_{\rm s}\omega_1|x-y|-c_{\rm p}\omega_2|x-z|)}}{c_{\rm
s}^{j_1+\frac{1}{2}}c_{\rm p}^{j_2-\frac{3}{2
}}}\bigg]\frac{(x_2-z_2)^2\mathbb
E(f_2(y)f_2(z))}{|x-y|^{j_1+\frac{1}{2}}|x-z|^{j_2+\frac{5}{2}}}{\rm d}y{\rm
d}z,
\end{align*}
\begin{align*}
&\mathbb E(\tilde{u}_{22}(x,\omega_1)\overline{\tilde{u}_{21}(x,\omega_2)})
=\frac{e^{-\frac{\pi}{2}{\rm
i}}}{16\mu}\sum_{j_1=0}^3\sum_{j_2=0}^2\frac{a_{j_1}^{(1)}\overline{a_{j_2}^{(0)
}}}{\omega_1^{j_1+\frac{3}{2}}\omega_2^{j_2+\frac{1}{2}}}\\
&\quad\times\int_{\mathbb R^4}\bigg[-\frac{e^{{\rm
i}(c_{\rm s}\omega_1|x-y|-c_{\rm s}\omega_2|x-z|)}}{c_{\rm
s}^{j_1+j_2}}+\frac{e^{{\rm i}(c_{\rm p}\omega_1|x-y|-c_{\rm
s}\omega_2|x-z|)}}{c_{\rm p}^{j_1-\frac{1}{2}}c_{\rm s}^{j_2+\frac{1}{2
}}}\bigg]\frac{\mathbb
E(f_2(y)f_2(z))}{|x-y|^{j_1+\frac{3}{2}}|x-z|^{j_2+\frac{1 }{2}}}{\rm d}y{\rm
d}z,
\end{align*}
\begin{align*}
&\mathbb E(\tilde{u}_{22}(x,\omega_1)\overline{\tilde{u}_{22}(x,\omega_2)})
=\frac{1}{16}\sum_{j_1,j_2=0}^3\frac{a_{j_1}^{(1)}\overline{a_{j_2}^{(1)}}}{
\omega_1^{j_1+\frac{3}{2}}\omega_2^{j_2+\frac{3}{2}}}\\
&\quad\times\int_{\mathbb R^4}\bigg[\frac{e^{{\rm
i}(c_{\rm s}\omega_1|x-y|-c_{\rm
s}\omega_2|x-z|)}}{c_{\rm s}^{j_1+j_2-1}}+\frac{e^{{\rm
i}(c_{\rm p}\omega_1|x-y|-c_{\rm p}\omega_2|x-z|)}}{c_{\rm p}^{j_1+j_2-1}}\\
&\qquad\qquad-\frac{e^{{\rm
i}(c_{\rm s}\omega_1|x-y|-c_{\rm
p}\omega_2|x-z|)}}{c_{\rm s}^{j_1-\frac{1}{2}}c_{\rm p}^{j_2-\frac{1}{2
}}}-\frac{e^{{\rm i}(c_{\rm
p}\omega_1|x-y|-c_{\rm
s}\omega_2|x-z|)}}{c_{\rm p}^{j_1-\frac{1}{2}}c_{\rm s}^{j_2-\frac{1}{2
}}}\bigg]\frac{\mathbb
E(f_2(y)f_2(z))}{|x-y|^{j_1+\frac{3}{2}}|x-z|^{j_2+\frac{3}{2}}}{\rm d}y{\rm
d}z,
\end{align*}
\begin{align*}
&\mathbb E(\tilde{u}_{22}(x,\omega_1)\overline{\tilde{u}_{23}(x,\omega_2)})
=\frac{e^{\frac{\pi}{2}{\rm
i}}}{16}\sum_{j_1=0}^3\sum_{j_2=0}^4\frac{a_{j_1}^{(1)}\overline{a_{j_2}^{(2)}}}
{\omega_1^{j_1+\frac{3}{2}}\omega_2^{j_2+\frac{1}{2}}}\\
&\quad\times\int_{\mathbb R^4}\bigg[-\frac{e^{{\rm
i}(c_{\rm s}\omega_1|x-y|-c_{\rm
s}\omega_2|x-z|)}}{c_{\rm s}^{j_1+j_2-2}}-\frac{e^{{\rm
i}(c_{\rm p}\omega_1|x-y|-c_{\rm p}\omega_2|x-z|)}}{c_{\rm p}^{j_1+j_2-2}}\\
&\qquad\qquad+\frac{e^{{\rm
i}(c_{\rm
s}\omega_1|x-y|-c_{\rm
p}\omega_2|x-z|)}}{c_{\rm s}^{j_1-\frac{1}{2}}c_{\rm p}^{j_2-\frac{ 3 } { 2
}}}+\frac{e^{{\rm
i}(c_{\rm
p}\omega_1|x-y|-c_{\rm s}\omega_2|x-z|)}}{c_{\rm
p}^{j_1-\frac{1}{2}}c_{\rm s}^{j_2-\frac{3}{2
}}}\bigg]\frac{(x_2-z_2)^2\mathbb
E(f_2(y)f_2(z))}{|x-y|^{j_1+\frac{3}{2}}|x-z|^{j_2+\frac{5}{2}}}{\rm d}y{\rm
d}z,
\end{align*}
\begin{align*}
&\mathbb E(\tilde{u}_{23}(x,\omega_1)\overline{\tilde{u}_{21}(x,\omega_2)})
=\frac{e^{-\pi{\rm
i}}}{16\mu}\sum_{j_1=0}^4\sum_{j_2=0}^2\frac{a_{j_1}^{(2)}\overline{a_{j_2}^{(0)
}}}{\omega_1^{j_1+\frac{1}{2}}\omega_2^{j_2+\frac{1}{2}}}\\
&\quad\times\int_{\mathbb R^4}\bigg[\frac{e^{{\rm
i}(c_{\rm s}\omega_1|x-y|-c_{\rm s}\omega_2|x-z|)}}{c_{\rm
s}^{j_1+j_2-1}}-\frac{e^{{\rm
i}(c_{\rm p}\omega_1|x-y|-c_{\rm s}\omega_2|x-z|)}}{c_{\rm
p}^{j_1-\frac{3}{2}}c_{\rm s}^{j_2+\frac{1}{2
}}}\bigg]\frac{(x_2-y_2)^2\mathbb
E(f_2(y)f_2(z))}{|x-y|^{j_1+\frac{5}{2}}|x-z|^{j_2+\frac{1}{2}}}{\rm d}y{\rm
d}z,
\end{align*}
\begin{align*}
&\mathbb E(\tilde{u}_{23}(x,\omega_1)\overline{\tilde{u}_{22}(x,\omega_2)})
=\frac{e^{-\frac{\pi}{2}{\rm
i}}}{16}\sum_{j_1=0}^4\sum_{j_2=0}^3\frac{a_{j_1}^{(2)}\overline{a_{j_2}^{(1)}}}
{\omega_1^{j_1+\frac{1}{2}}\omega_2^{j_2+\frac{3}{2}}}\\
&\times\int_{\mathbb R^4}\bigg[-\frac{e^{{\rm
i}(c_{\rm s}\omega_1|x-y|-c_{\rm s}\omega_2|x-z|)}}{c_{\rm
s}^{j_1+j_2-2}}-\frac{e^{{\rm
i}(c_{\rm p}\omega_1|x-y|-c_{\rm p}\omega_2|x-z|)}}{c_{\rm p}^{j_1+j_2-2}}\\
&\qquad\qquad+\frac{e^{{\rm
i}(c_{\rm s}\omega_1|x-y|-c_{\rm p}\omega_2|x-z|)}}{c_{\rm
s}^{j_1-\frac{3}{2}}c_{\rm p}^{j_2-\frac{1}{2
}}}+\frac{e^{{\rm
i}(c_{\rm p}\omega_1|x-y|-c_{\rm s}\omega_2|x-z|)}}{c_{\rm
p}^{j_1-\frac{3}{2}}c_{\rm s}^{j_2-\frac{1}{2
}}}\bigg]\frac{(x_2-y_2)^2\mathbb
E(f_2(y)f_2(z))}{|x-y|^{j_1+\frac{5}{2}}|x-z|^{j_2+\frac{3}{2}}}{\rm d}y{\rm
d}z,
\end{align*}
\begin{align*}
&\mathbb E(\tilde{u}_{23}(x,\omega_1)\overline{\tilde{u}_{23}(x,\omega_2)})
=\frac{1}{16}\sum_{j_1,j_2=0}^4\frac{a_{j_1}^{(2)}\overline{a_{j_2}^{(2)}}}{
\omega_1^{j_1+\frac{1}{2}}\omega_2^{j_2+\frac{1}{2}}}\\
&\quad\times\int_{\mathbb R^4}\bigg[\frac{e^{{\rm
i}(c_{\rm s}\omega_1|x-y|-c_{\rm s}\omega_2|x-z|)}}{c_{\rm
s}^{j_1+j_2-3}}+\frac{e^{{\rm
i}(c_{\rm p}\omega_1|x-y|-c_{\rm p}\omega_2|x-z|)}}{c_{\rm p}^{j_1+j_2-3}}\\
&\qquad-\frac{e^{{\rm
i}(c_{\rm s}\omega_1|x-y|-c_{\rm p}\omega_2|x-z|)}}{c_{\rm
s}^{j_1-\frac{3}{2}}c_{\rm p}^{j_2-\frac{3}{2
}}}-\frac{e^{{\rm
i}(c_{\rm p}\omega_1|x-y|-c_{\rm s}\omega_2|x-z|)}}{c_{\rm
p}^{j_1-\frac{3}{2}}c_{\rm s}^{j_2-\frac{3}{2
}}}\bigg]\frac{(x_2-y_2)^2(x_2-z_2)^2\mathbb
E(f_2(y)f_2(z))}{|x-y|^{j_1+\frac{5}{2}}|x-z|^{j_2+\frac{5}{2}}}{\rm d}y{\rm
d}z,
\end{align*}
\begin{align*}
&\mathbb E(\tilde{u}_{24}(x,\omega_1)\overline{\tilde{u}_{24}(x,\omega_2)})
=\frac{1}{16}\sum_{j_1,j_2=0}^4\frac{a_{j_1}^{(2)}\overline{a_{j_2}^{(2)}}}{
\omega_1^{j_1+\frac{1}{2}}\omega_2^{j_2+\frac{1}{2}}}\\
&\quad\times\int_{\mathbb R^4}\bigg[\frac{e^{{\rm
i}(c_{\rm s}\omega_1|x-y|-c_{\rm s}\omega_2|x-z|)}}{c_{\rm
s}^{j_1+j_2-3}}+\frac{e^{{\rm
i}(c_{\rm p}\omega_1|x-y|-c_{\rm p}\omega_2|x-z|)}}{c_{\rm
p}^{j_1+j_2-3}}-\frac{e^{{\rm
i}(c_{\rm s}\omega_1|x-y|-c_{\rm p}\omega_2|x-z|)}}{c_{\rm
s}^{j_1-\frac{3}{2}}c_{\rm p}^{j_2-\frac{3}{2
}}}\\
&\qquad\qquad-\frac{e^{{\rm
i}(c_{\rm p}\omega_1|x-y|-c_{\rm s}\omega_2|x-z|)}}{c_{\rm
p}^{j_1-\frac{3}{2}}c_{\rm s}^{j_2-\frac{3}{2
}}}\bigg]\frac{(x_1-y_1)(x_2-y_2)(x_1-z_1)(x_2-z_2)}{|x-y|^{j_1+\frac{5}{2}}
|x-z|^{j_2+\frac{5}{2}}}\mathbb E(f_1(y)f_1(z)){\rm d}y{\rm d}z.
\end{align*}

A direct application of Lemma \ref{lemma2b} to each item on the right hand side
of (\ref{d16}) gives the following lemma.

\begin{lemma}\label{lemma2d}
For $\omega_1\geq 1, \omega_2\geq 1$, the following estimates
\begin{align*}
|\mathbb E(\tilde{\bm{u}}(x,\omega_1)\cdot
\tilde{\bm{u}}(x,\omega_2))|&\leq c_n(1+|\omega_1-\omega_2|)^{-n}
(\omega_1+\omega_2)^{-m-1},\\
|\mathbb
E(\tilde{\bm{u}}(x,\omega_1)\cdot\tilde{\bm{u}}(x,
\omega_2))|&\leq c_n(\omega_1+\omega_2)^{-n}(1+|\omega_1-\omega_2|)^{-m}
\end{align*}
holds uniformly for $x\in U$, where $n\in\mathbb N$ is arbitrary and $c_n>0$ is
a constant depending only on $n$. 
\end{lemma}

To obtain the relation between the scattering data and the function in the
principle symbol, it is required to estimate the order of $\mathbb
E(\tilde{u}(x,\omega)\cdot\overline{\tilde{u}(x,\omega)})$. According to
(\ref{d16}) where we set $\omega_1=\omega_2=\omega$, it reduces to estimate the
order of
$\mathbb E(\tilde{u}_{ij_1}(x,\omega)\overline{\tilde{u}_{ij_2}(x,\omega)})$ for
$i=1,2, j_1, j_2=1,2,3,4$. Applying Lemma \ref{lemma4b} gives that
\begin{align*}
\mathbb
E(\tilde{u}_{11}(x,\omega)\overline{\tilde{u}_{12}(x,\omega)})&=O(\omega^{
-(m+2)}), \quad
\mathbb
E(\tilde{u}_{12}(x,\omega)\overline{\tilde{u}_{11}(x,\omega)})=O(\omega^{
-(m+2)}),\\
\mathbb
E(\tilde{u}_{12}(x,\omega)\overline{\tilde{u}_{12}(x,\omega)})&=O(\omega^{
-(m+2)}), \quad
\mathbb
E(\tilde{u}_{12}(x,\omega)\overline{\tilde{u}_{13}(x,\omega)})=O(\omega^{
-(m+2)}),\\
\mathbb
E(\tilde{u}_{13}(x,\omega)\overline{\tilde{u}_{12}(x,\omega)})&=O(\omega^{
-(m+2)}), \quad
\mathbb
E(\tilde{u}_{21}(x,\omega)\overline{\tilde{u}_{22}(x,\omega)})=O(\omega^{
-(m+2)}),\\
\mathbb
E(\tilde{u}_{22}(x,\omega)\overline{\tilde{u}_{21}(x,\omega)})&=O(\omega^{
-(m+2)}), \quad
\mathbb
E(\tilde{u}_{22}(x,\omega)\overline{\tilde{u}_{22}(x,\omega)})=O(\omega^{
-(m+2)}),\\
\mathbb
E(\tilde{u}_{22}(x,\omega)\overline{\tilde{u}_{23}(x,\omega)})&=O(\omega^{
-(m+2)}), \quad
\mathbb
E(\tilde{u}_{23}(x,\omega)\overline{\tilde{u}_{22}(x,\omega)})=O(\omega^{
-(m+2)}),
\end{align*}
and
\begin{align*}
\mathbb
E(\tilde{u}_{11}(x,\omega)\overline{\tilde{u}_{11}(x,\omega)})&=N^{(2)}
_1(x)\omega^{-(m+1)}+O(\omega^{-(m+2)})\\
\mathbb
E(\tilde{u}_{11}(x,\omega)\overline{\tilde{u}_{13}(x,\omega)})&=N^{(2)}_2(x,
\omega)\omega^{-(m+1)}+O(\omega^{-(m+2)})\\
\mathbb
E(\tilde{u}_{13}(x,\omega)\overline{\tilde{u}_{11}(x,\omega)})&=N^{(2)}_3(x,
\omega)\omega^{-(m+1)}+O(\omega^{-(m+2)})\\
\mathbb
E(\tilde{u}_{13}(x,\omega)\overline{\tilde{u}_{13}(x,\omega)})&=N^{(2)}_4(x,
\omega)\omega^{-(m+1)}+O(\omega^{-(m+2)})\\
\mathbb
E(\tilde{u}_{14}(x,\omega)\overline{\tilde{u}_{14}(x,\omega)})&=N^{(2)}_5(x,
\omega)\omega^{-(m+1)}+O(\omega^{-(m+2)})\\
\mathbb
E(\tilde{u}_{21}(x,\omega)\overline{\tilde{u}_{21}(x,\omega)})&=N^{(2)}
_6(x)\omega^{-(m+1)}+O(\omega^{-(m+2)})\\
\mathbb
E(\tilde{u}_{21}(x,\omega)\overline{\tilde{u}_{23}(x,\omega)})&=N^{(2)}_7(x,
\omega)\omega^{-(m+1)}+O(\omega^{-(m+2)})\\
\mathbb
E(\tilde{u}_{23}(x,\omega)\overline{\tilde{u}_{21}(x,\omega)})&=N^{(2)}_8(x,
\omega)\omega^{-(m+1)}+O(\omega^{-(m+2)})\\
\mathbb
E(\tilde{u}_{23}(x,\omega)\overline{\tilde{u}_{23}(x,\omega)})&=N^{(2)}_9(x,
\omega)\omega^{-(m+1)}+O(\omega^{-(m+2)})\\
\mathbb
E(\tilde{u}_{24}(x,\omega)\overline{\tilde{u}_{24}(x,\omega)})&=N^{(2)}_{10}(x,
\omega)\omega^{-(m+1)}+O(\omega^{-(m+2)}),
\end{align*}
where
\begin{align*}
N^{(2)}_1(x)&=N^{(2)}_6(x)=a_{1}\int_{\mathbb
R^2}\frac{1}{|x-y|}\phi(y){\rm d}y,\\
N^{(2)}_2(x,\omega)&=\int_{\mathbb R^2}(a_{2}e^{{\rm
i}(c_{\rm s}-c_{\rm
p})|x-y|\omega}-a_{1})\frac{(x_1-y_1)^2}{|x-y|^3}\phi(y){\rm d}y,\\
N^{(2)}_3(x,\omega)&=\int_{\mathbb R^2}(a_{2}e^{{\rm
i}(c_{\rm
p}-c_{\rm
s})|x-y|\omega}-a_{1})\frac{(x_1-y_1)^2}{|x-y|^3}\phi(y){\rm d}y,\\
N^{(2)}_4(x,\omega)&=\int_{\mathbb
R^2}(a_{3}-2a_{2}\cos((c_{\rm s}-c_{\rm
p})|x-y|\omega))\frac{(x_1-y_1)^4}{|x-y|^5}
\phi(y){\rm d}y,\\
N^{(2)}_5(x,\omega)&=N^{(2)}_{10}(x,\omega)=\int_{\mathbb
R^2}(a_{3}-2a_{2}\cos((c_{\rm s}-c_{\rm
p})|x-y|\omega))\frac{(x_1-y_1)^2(x_2-y_2)^2}{
|x-y|^5}\phi(y){\rm d}y,\\
N^{(2)}_7(x,\omega)&=\int_{\mathbb R^2}(a_{2}e^{{\rm
i}(c_{\rm s}-c_{\rm
p})|x-y|\omega}-a_{1})\frac{(x_2-y_2)^2}{|x-y|^3}\phi(y){\rm d}y,\\
N^{(2)}_8(x,\omega)&=\int_{\mathbb R^2}(a_{2}e^{{\rm
i}(c_{\rm p}-c_{\rm
s})|x-y|\omega}-a_{1})\frac{(x_2-y_2)^2}{|x-y|^3}\phi(y){\rm d}y,\\
N^{(2)}_9(x,\omega)&=\int_{\mathbb
R^2}(a_{3}-2a_{2}\cos((c_{\rm s}-c_{\rm
p})|x-y|\omega))\frac{(x_2-y_2)^4}{|x-y|^5}\phi(y){\rm d}y.
\end{align*}
Here, $a_{1}$, $a_{2}$, and $a_{3}$ are positive constants given by
\[
a_{1}=\frac{1}{32\pi c_{\rm s}^{m-3}},\quad a_{2}=
\frac{(c_{\rm s}c_{\rm p})^{\frac{3}{2}}}{32\pi}\left(\frac{2}{c_{\rm s}+c_{\rm
p}}\right)^{m},\quad
a_{3}=\frac{1}{32\pi}\left(\frac{1}{c_{\rm s}^{m-3}}+\frac{1}{c_{\rm
p}^{m-3}}\right).
\]
By (\ref{d16}) and a simple calculation, we obtain
\begin{eqnarray}\label{d50}
\mathbb
E(\tilde{\bm{u}}(x,\omega)\cdot\overline{\tilde{\bm{u}}(x,\omega)})=T^{(2)}_{\rm
 E}(x)\omega^{-(m+1)}+O(\omega^{-(m+2)}),
\end{eqnarray}
where 
\begin{equation}\label{t1}
T^{(2)}_{\rm E}(x)=\sum_{j=1}^{10}N^{(2)}_j(x,\omega)=a_{3}\int_{\mathbb
R^2}\frac{1}{|x-y|}\phi(y)dy.
\end{equation}

Now we are ready to present the main result for elastic waves in the 
two dimensions.

\begin{theorem}
Let the external source $\bm{f}$ be a microlocally isotropic Gaussian random
vector field which satisfies Assumption C. Then for all $x\in U$, it
holds almost surely that 
\begin{eqnarray}\label{d53}
\lim_{Q\rightarrow\infty}\frac{1}{Q-1}\int_1^Q\omega^{m+1}|\bm{u}(x,
\omega)|^2{\rm d}\omega=T^{(2)}_{\rm E}(x),
\end{eqnarray}
where $T^{(2)}_{\rm E}(x)$ is given in (\ref{t1}). Moreover, the
scattering data $T^{(2)}_{\rm E}(x)$, for $x\in U$ uniquely determine the
micro-correlation strength $\phi$ through the linear integral
equation (\ref{t1}).
\end{theorem}

\begin{proof}
Since 
\begin{align*}
&\frac{1}{Q-1}\int_1^Q\omega^{m+1}|\bm{u}(x,\omega)|^2{\rm d}\omega\\
&=\frac{1}{Q-1}\int_1^Q\omega^{m+1}|\tilde{\bm{u}}(x,\omega)+\bm{u}(x,
\omega)-\tilde{\bm{u}}(x,\omega)|^2{\rm d}\omega\\
&=\frac{1}{Q-1}\int_1^Q\omega^{m+1}|\tilde{\bm{u}}(x,\omega)|^2d\omega+
\frac{1}{Q-1}\int_1^Q\omega^{m+1}|\bm{u}(x,\omega)-\tilde{\bm{u}}(x,
\omega)|^2{\rm d}\omega\\
&\quad+\frac{2}{Q-1}\int_1^Q\omega^{m+1}\mathbb
\Re\left[\overline{\tilde{\bm{u}}(x,\omega)}(\bm{u}(x,\omega)-\tilde{\bm{u}}(x,
\omega))\right]{\rm d}\omega
\end{align*}
thus, (\ref{d53}) holds as long as we show that
\begin{align}\label{d56}
&\lim_{Q\rightarrow\infty}\frac{1}{Q-1}\int_1^Q\omega^{m+1}|\tilde{\bm{u}}(x,
\omega)|^2{\rm d}\omega=T^{(2)}_{\rm E}(x),\\\label{d57}
&\lim_{Q\rightarrow\infty}\frac{1}{Q-1}\int_1^Q\omega^{m+1}|\bm{u}(x,
\omega)-\tilde{\bm{u}}(x,\omega)|^2{\rm d}\omega=0,\\\label{d58}
&\lim_{Q\rightarrow\infty}\frac{2}{Q-1}\int_1^Q\omega^{m+1}\mathbb
\Re\left[\overline{\tilde{\bm{u}}(x,\omega)}(\bm{u}(x,\omega)-\tilde{\bm{u}}(x,
\omega))\right]{\rm d}\omega=0.
\end{align}
To prove (\ref{d56}), we denote
$Y(x,\omega)=\omega^{m+1}(|\tilde{\bm{u}}(x,\omega)|^2-\mathbb E
(|\tilde{\bm{u}}(x,\omega)|^2) )$, which yields
\[
\int_1^Q\omega^{m+1}|\tilde{\bm{u}}(x,\omega)|^2{\rm d}\omega
=\int_1^Q\omega^{m+1}\mathbb E(|\tilde{\bm{u}}(x,\omega)|^2){\rm d}\omega+
\int_1^QY(x,\omega){\rm d}\omega.
\]
Hence, (\ref{d56}) holds as long as we prove 
\begin{equation}\label{d59}
\lim_{Q\rightarrow\infty}\frac{1}{Q-1}\int_1^Q\omega^{m+1}\mathbb
E(|\tilde{\bm{u}}(x,\omega)|^2){\rm d}\omega= T^{(2)}_{\rm E}(x), \quad
\lim_{Q\rightarrow\infty}\frac{1}{Q-1}\int_1^Q Y(x,\omega){\rm d}\omega=0.
\end{equation}
Multiplying (\ref{d50}) by $\omega^{m+1}$ and integrating with respect to the
frequency $\omega$ in the internal $(1, Q)$, we arrive at 
\begin{eqnarray*}
\frac{1}{Q-1}\int_1^Q\omega^{m+1}\mathbb
E(|\tilde{\bm{u}}(x,\omega)|^2){\rm d}\omega
=\frac{1}{Q-1}\int_1^Q\left(T^{(2)}_{\rm E}(x)+O(\omega^{m-3})\right){\rm
d}\omega.
\end{eqnarray*}
It is clear to note that 
\begin{eqnarray*}
\bigg|\frac{1}{Q-1}\int_1^QO(\omega^{m-3}){\rm d}\omega\bigg|
\lesssim \frac{1}{Q-1}\int_1^Q\omega^{m-3}{\rm d}\omega
\rightarrow 0\quad {\rm as}\;\;Q\rightarrow \infty,
\end{eqnarray*}
where we use the fact that $m\in [d,d+\frac{1}{2})$. Thus, the first equation in
(\ref{d59}) holds. Now we focus on the second equation in (\ref{d59}) and want
to show that 
\begin{eqnarray*}
\lim_{Q\rightarrow\infty}\frac{1}{Q-1}\int_1^QY(x,\omega){\rm d}\omega=0.
\end{eqnarray*}
By the definition of $Y(x,\omega)$, 
\begin{align*}
Y(x,\omega)&= \omega^{m+1}(|\tilde{\bm{u}}(x,\omega)|^2-\mathbb E
(|\tilde{\bm{u}}(x,\omega)|^2) )\\
&=\omega^{m+1}\Big((\Re\tilde{\bm{u}}(x,\omega))^2-\mathbb
E(\Re\tilde{\bm{u}}(x,\omega))^2+(\Im\tilde{\bm{u}}(x,\omega))^2-\mathbb
E(\Im\tilde{\bm{u}}(x,\omega))^2\Big).
\end{align*}
Therefore, 
\begin{eqnarray*}
\mathbb E(Y(x,\omega_1)Y(x,\omega_2))=I_{E,1}+I_{E,2}+I_{E,3}+I_{E,4},
\end{eqnarray*}
where
\begin{align*}
I_{E,1}&=\omega_1^{m+1} \omega_2^{m+1}\mathbb
E\bigg[((\Re\tilde{\bm{u}}(x,\omega_1))^2-\mathbb
E(\Re\tilde{\bm{u}}(x,\omega_1))^2)((\Re\tilde{\bm{u}}(x,\omega_2))^2-\mathbb
E(\Re\tilde{\bm{u}}(x,\omega_2))^2)\bigg],\\
I_{E,2}&= \omega_1^{m+1} \omega_2^{m+1}\mathbb
E\bigg[((\Re\tilde{\bm{u}}(x,\omega_1))^2-\mathbb
E(\Re\tilde{\bm{u}}(x,\omega_1))^2)((\Im\tilde{\bm{u}}(x,
\omega_2))^2-\mathbb E(\Im\tilde{\bm{u}}(x,\omega_2))^2)\bigg],\\
I_{E,3} &=\omega_1^{m+1} \omega_2^{m+1}\mathbb
E\bigg[((\Im\tilde{\bm{u}}(x,\omega_1))^2-\mathbb
E(\Im\tilde{\bm{u}}(x,\omega_1))^2)((\Re\tilde{\bm{u}}(x,\omega_2))^2-\mathbb
E(\Re\tilde{\bm{u}}(x,\omega_2))^2)\bigg],\\
I_{E,4} &= \omega_1^{m+1} \omega_2^{m+1}\mathbb
E\bigg[((\Im\tilde{\bm{u}}(x,\omega_1))^2-\mathbb
E(\Im\tilde{\bm{u}}(x,\omega_1))^2)((\Im\tilde{\bm{u}}(x,\omega_2))^2-\mathbb
E(\Im\tilde{\bm{u}}(x,\omega_2))^2)\bigg].
\end{align*}
Combing the expression of $\tilde{\bm{u}}(x,\omega)$ and the assumption $\mathbb
E(f_1)=0$, $\mathbb E(f_2)=0$ gives that $\Re \tilde{\bm{u}}(x,\omega)$ and
$\Im\tilde{\bm{u}}(x,\omega)$ are zero-mean Gaussian random variables. Applying
Lemmas \ref{lemma2a} and \ref{lemma2d} leads to
\begin{align*}
I_{E,1}&=2\omega_1^{m+1} \omega_2^{m+1}[\mathbb E
(\Re\tilde{\bm{u}}(x,\omega_1)\Re\tilde{\bm{u}}(x,\omega_2))]^2\\
&=\frac{1}{2}\omega_1^{m+1} \omega_2^{m+1}\bigg[\mathbb
E\Big(\Re(\tilde{\bm{u}}(x,\omega_1)\tilde{\bm{u}}(x,\omega_2))+\Re(\tilde{\bm{
u}}(x,\omega_1)\overline{\tilde{\bm{u}}(x,\omega_2)})\Big)\bigg]^2\\
&\lesssim\bigg[\frac{\omega_1^{\frac{m+1}{2}}\omega_2^{\frac{m+1}{2}}}{
(\omega_1+\omega_2)^n(1+|\omega_1-\omega_2|)^m}+\frac{\omega_1^{\frac{m+1}{2}}
\omega_2^{\frac{m+1}{2}}}{(\omega_1+\omega_2)^{m+1}(1+|\omega_1-\omega_2|)^n}
\bigg]^2\\
&\lesssim \bigg[\frac{1}{(1+|\omega_1-\omega_2|)^{m}}+\frac{1}{
(1+|\omega_1-\omega_2|)^n}\bigg]^2.
\end{align*}
We can obtain the same estimates for $I_{E,2}$, $I_{E,3}$, and $I_{E,4}$ by the
similar arguments. Thus, an application of Lemma \ref{lemma3a} gives 
\begin{eqnarray*}
\lim_{Q\rightarrow\infty}\frac{1}{Q-1}\int_1^QY(x,\omega){\rm d}\omega=0.
\end{eqnarray*}
 
To prove (\ref{d57}), from lemma \ref{lemma1d}, we obtain
\begin{align*}
&\bigg|\frac{1}{Q-1}\int_1^Q\omega^{m+1}|\bm{u}(x,\omega)-\tilde{\bm{u}}(x,
\omega)|^2{\rm d}\omega\bigg| \lesssim
\frac{1}{Q-1}\int_1^Q\omega^{m+1}\omega^{-7}{\rm d}\omega\\
&\lesssim\frac{1}{Q-1}\int_1^Q\omega^{m-6}{\rm d}\omega
\lesssim\frac{1}{m-5}\frac{Q^{ m-5}-1}{Q-1}\rightarrow 0 \quad{\rm
as}\;\;Q\rightarrow\infty.
\end{align*}
To prove (\ref{d58}), by the H\"{o}lder inequality, we have 
\begin{align*}
&\bigg|\frac{2}{Q-1}\int_1^Q\omega^{m+1}
\Re\left[\overline{\tilde{\bm{u}}(x,\omega)}(\bm{u}(x,\omega)-\tilde{\bm{u}}(x,
\omega))\right]{\rm d}\omega\bigg|\\
&\lesssim \frac{2}{Q-1}\int_1^Q\omega^{m+1}|\tilde{\bm{u}}(x,\omega)||\bm{u}(x,
\omega)-\tilde{\bm{u}}(x,\omega)|{\rm d}\omega\\
&\lesssim 2\bigg[\frac{1}{Q-1}\int_1^Q\omega^{m+1}|\tilde{\bm{u}}(x,
\omega)|^2{\rm
d}\omega\bigg]^{\frac{1}{2}}\bigg[\frac{1}{Q-1}\int_1^Q\omega^{m+1}|\bm{u}(x,
\omega)-\tilde{\bm{u}}(x,\omega)|^2{\rm d}\omega\bigg]^{\frac{1}{2}}\\
&\rightarrow 2T^{(2)}_{\rm E}(x)^{\frac{1}{2}}\cdot 0 = 0\quad{\rm
as}\;\;Q\rightarrow\infty.
\end{align*}
The unique determination of $\phi$ by $T_{\rm E}^{(2)}(x)$ for $x\in U$ is a
direct consequence of Lemma \ref{lemma_o}.
\end{proof}

\subsection{The three-dimensional case}

To derive the linear relation between the scattering data and the function in
the principle symbol, it is required to express the wave field more explicitly
than (\ref{s20}). Substituting (\ref{s15}) into (\ref{s5}) gives the wave field 
$\bm{u}(x,\omega)=(u_1(x,\omega), u_2(x,\omega), u_3(x,\omega))^\top$ where each
component $u_i(x,\omega)$ is given by
\begin{eqnarray*}
u_i(x,\omega)=u_{i1}(x,\omega)+u_{i2}(x,\omega)+u_{i3}(x,\omega).
\end{eqnarray*}
Here
\begin{align*}
u_{i1}(x,\omega) &= \frac{1}{4\pi\mu}\int_{\mathbb R^3}\frac{e^{{\rm
i}\kappa_s|x-y|}}{|x-y|}f_i(y){\rm d}y,\\
u_{i2}(x,\omega) &= \frac{1}{4\pi\omega^2}\int_{\mathbb
R^3}\frac{1}{|x-y|^3}\bigg[e^{{\rm i}\kappa_s|x-y|}({\rm
i}\kappa_s|x-y|-1)-e^{{\rm i}\kappa_p|x-y|}({\rm
i}\kappa_p|x-y|-1)\bigg]f_i(y){\rm d}y,\\
u_{i3}(x,\omega) &= -\frac{1}{4\pi\omega^2}\int_{\mathbb
R^3}\bigg[\left(\frac{3({\rm
i}\kappa_s|x-y|-1)}{|x-y|^2}+\kappa_s^2\right)e^{{\rm i}\kappa_s|x-y|}\\
&\qquad\qquad\quad-\left(\frac{3({\rm
i}\kappa_p|x-y|-1)}{|x-y|^2}+\kappa_p^2\right)e^{{\rm
i}\kappa_p|x-y|}\bigg]\frac{(x_i-y_i)}{|x-y|^3}(x-y)\cdot \bm{f}(y){\rm d}y.
\end{align*}

As mentioned above, we need to derive the relationship between the scattering
data and the function $\phi$ in the principle symbol. For this end, it is
required to calculate the exception $\mathbb
E(\bm{u}(x,\omega_1)\cdot\overline{\bm{u}(x,\omega_2)})$. Noting that
$\bm{u}(x,\omega)\in\mathbb C^3$ and each component has been decomposed into
three parts, we obtain 
\begin{align}\label{e4}
&\mathbb E(\bm{u}(x,\omega)\cdot \overline{\bm{u}(x,\omega)})\notag\\
&=\mathbb
E(u_1(x,\omega)\overline{u_1(x,\omega)}+u_2(x,\omega)\overline{u_2(x,\omega)}
+u_3(x,\omega)\overline{u_3(x,\omega)})\notag\\
&=\sum_{i,j=1}^3\mathbb
E\left(u_{1i}(x,\omega)\overline{u_{1j}(x,\omega)}+u_{2i}(x,\omega)\overline{u_{
2j}(x,\omega)}+u_{3i}(x,\omega)\overline{u_{3j}(x,\omega)}\right).
\end{align}
To calculate the expectation $\mathbb
E(\bm{u}(x,\omega_1)\cdot\overline{\bm{u}(x,\omega_2)})$, it is required to
calculate the items on the right hand side of (\ref{e4}). Using the expression
of $u_{ij}(x,\omega)(i,j=1,2,3)$, we have from direct calculations that 
\begin{align*}
\mathbb
E(u_{i1}(x,\omega_1)\overline{u_{i1}(x,\omega_2)})=\frac{1}{16\pi^2\mu^2}\int_{
\mathbb R^6}\frac{e^{{\rm
i}(c_{\rm s}\omega_1|x-y|-c_{\rm s}\omega_2|x-z|)}}{|x-y||x-z|}\mathbb
E(f_i(y)f_i(z)){\rm d}y{\rm d}z,
\end{align*}
\begin{align*}
&\mathbb
E(u_{i1}(x,\omega_1)\overline{u_{i2}(x,\omega_2)})=-\frac{1}{
16\pi^2\mu\omega_2^2}\int_{\mathbb R^6}\bigg[e^{{\rm
i}(c_{\rm s}\omega_1|x-y|-c_{\rm s}\omega_2|x-z|)}({\rm i}c_{\rm
s}\omega_2|x-z|+1)\\
&\qquad\qquad\qquad\qquad-e^{{\rm i}(c_{\rm s}\omega_1|x-y|-c_{\rm
p}\omega_2|x-z|)}({\rm
i}c_{\rm p}\omega_2|x-z|+1)\bigg]\frac{\mathbb
E(f_i(y)f_i(z))}{|x-y||x-z|^3}{\rm d}y{\rm d}z,
\end{align*}
\begin{align*}
&\mathbb E(u_{i1}(x,\omega_1)\overline{u_{i3}(x,\omega_2)})=\frac{1}{
16\pi^2\mu\omega_2^2}
\int_{\mathbb R^6}\bigg[\bigg(\frac{3}{|x-z|^2}({\rm
i}c_{\rm s}\omega_2|x-z|+1)-c_{\rm s}^2\omega_2^2\bigg)\times\\
&e^{{\rm i}(c_{\rm s}\omega_1|x-y|-c_{\rm s}\omega_2|x-z|)}
-\bigg(\frac{3}{|x-z|^2}({\rm i}c_{\rm p}\omega_2|x-z|+1)-c_{\rm
p}^2\omega_2^2\bigg)
e^{{\rm i}(c_{\rm s}\omega_1|x-y|-c_{\rm p}\omega_2|x-z|)}\bigg]\\
&\qquad\qquad\qquad\qquad\qquad\qquad\times\frac{(x_i-z_i)^2}{|x-y||x-z|^3}
\mathbb E(f_i(y)f_i(z)){\rm d}y{\rm d}z,
\end{align*}
\begin{align*}
&\mathbb E(u_{i2}(x,\omega_1)\overline{u_{i1}(x,\omega_2)})=\frac{1}{
16\pi^2\mu\omega_1^2}
\int_{\mathbb R^6}\bigg[({\rm i}c_{\rm s}\omega_1|x-y|-1)e^{{\rm
i}(c_{\rm s}\omega_1|x-y|-c_{\rm s}\omega_2|x-z|)}\\
&\qquad\qquad\qquad\qquad-({\rm i}c_{\rm p}\omega_1|x-y|-1)e^{{\rm
i}(c_{\rm p}\omega_1|x-y|-c_{\rm s}\omega_2|x-z|)}\bigg]\frac{\mathbb
E(f_i(y)f_i(z))}{|x-y|^3|x-z|}{\rm d}y{\rm d}z,
\end{align*}
\begin{align*}
&\mathbb E(u_{i2}(x,\omega_1)\overline{u_{i2}(x,\omega_2)})
=\frac{1}{16\pi^2\omega_1^2\omega_2^2}\times\\
&\int_{\mathbb R^6}\bigg[({\rm i}c_{\rm s}\omega_1|x-y|-1)(-{\rm
i}c_{\rm s}\omega_2|x-z|-1)e^{{\rm i}(c_{\rm s}\omega_1|x-y|-c_{\rm
s}\omega_2|x-z|)}\\
&+({\rm i}c_{\rm p}\omega_1|x-y|-1)(-{\rm i}c_{\rm p}\omega_2|x-z|-1)e^{{\rm
i}(c_{\rm p}\omega_1|x-y|-c_{\rm p}\omega_2|x-z|)}\\
&-({\rm i}c_{\rm s}\omega_1|x-y|-1)(-{\rm i}c_{\rm p}\omega_2|x-z|-1)e^{{\rm
i}(c_{\rm s}\omega_1|x-y|-c_{\rm p}\omega_2|x-z|)}\\
&-({\rm i}c_{\rm p}\omega_1|x-y|-1)(-{\rm i}c_{\rm s}\omega_2|x-z|-1)e^{{\rm
i}(c_{\rm p}\omega_1|x-y|-c_{\rm s}\omega_2|x-z|)}\bigg]\frac{\mathbb
E(f_i(y)f_i(z))}{|x-y|^3|x-z|^3}{\rm d}y{\rm d}z,
\end{align*}
\begin{align*}
&\mathbb E(u_{i2}(x,\omega_1)\overline{u_{i3}(x,\omega_2)})=\frac{1}{
16\pi^2\omega_1^2\omega_2^2}\times\\
&\int_{\mathbb R^6}\bigg[
({\rm i}c_{\rm s}\omega_1|x-y|-1)\left(\frac{3}{|x-z|^2}({\rm
i}c_{\rm s}\omega_2|x-z|+1)+c_{\rm s}^2\omega_2^2\right)e^{{\rm
i}(c_{\rm s}\omega_1|x-y|-c_{\rm s}\omega_2|x-z|)}\\
&+({\rm i}c_{\rm p}\omega_1|x-y|-1)\left(\frac{3}{|x-z|^2}({\rm i}c_{\rm
p}\omega_2|x-z|+1)+c_{\rm p}^2\omega_2^2\right)e^{{\rm i}(c_{\rm
p}\omega_1|x-y|-c_{\rm p}\omega_2|x-z|)}\\
&-({\rm i}c_{\rm s}\omega_1|x-y|-1)\left(\frac{3}{|x-z|^2}({\rm i}c_{\rm
p}\omega_2|x-z|+1)+c_{\rm p}^2\omega_2^2\right)e^{{\rm
i}(c_{\rm s}\omega_1|x-y|-c_{\rm p}\omega_2|x-z|)}\\
&-({\rm i}c_{\rm p}\omega_1|x-y|-1)\left(\frac{3}{|x-z|^2}({\rm
i}c_{\rm s}\omega_2|x-z|+1)+c_{\rm s}^2\omega_2^2\right)e^{{\rm i}(c_{\rm
p}\omega_1|x-y|-c_{\rm s}\omega_2|x-z|)}
\bigg]\\
&\times\frac{(x_i-z_i)^2}{|x-y|^3|x-z|^3}\mathbb E(f_i(y)f_i(z)){\rm d}y{\rm
d}z,
\end{align*}
\begin{align*}
&\mathbb E(u_{i3}(x,\omega_1)\overline{u_{i1}(x,\omega_2)})
=-\frac{1}{16\pi^2\mu\omega_1^2}\times\\
&\int_{\mathbb R^6}\bigg[\bigg(\frac{3}{|x-y|^2}({\rm
i}c_{\rm s}\omega_1|x-y|-1)+c_{\rm s}^2\omega_1^2\bigg)e^{{\rm
i}(c_{\rm s}\omega_1|x-y|-c_{\rm s}\omega_2|x-z|)}\\
&\qquad\qquad-\bigg(\frac{3}{|x-y|^2}({\rm i}c_{\rm p}\omega_1|x-y|-1)+c_{\rm
p}^2\omega_1^2\bigg)e^{{\rm i}(c_{\rm p}\omega_1|x-y|-c_{\rm
s}\omega_2|x-z|)}\bigg]\\
&\qquad\qquad\times\frac{(x_i-y_i)^2}{|x-y|^3|x-z|}\mathbb
E(f_i(y)f_i(z)){\rm d}y{\rm d}z,
\end{align*}
\begin{align*}
&\mathbb E(u_{i3}(x,\omega_1)\overline{u_{i2}(x,\omega_2)})=\frac{1}{
16\pi^2\omega_1^2\omega_2^2}\times\\
&\int_{\mathbb R^6}\bigg[({\rm
i}c_{\rm s}\omega_2|x-z|+1)\left(\frac{3}{|x-y|^2}({\rm
i}c_{\rm s}\omega_1|x-y|-1)+c_{\rm s}^2\omega_1^2\right)e^{{\rm
i}(c_{\rm s}\omega_1|x-y|-c_{\rm s}\omega_2|x-z|)}\\
&+({\rm i}c_{\rm p}\omega_2|x-z|+1)\left(\frac{3}{|x-y|^2}({\rm i}c_{\rm
p}\omega_1|x-y|-1)+c_{\rm p}^2\omega_1^2\right)e^{{\rm i}(c_{\rm
p}\omega_1|x-y|-c_{\rm p}\omega_2|x-z|)}\\
&-({\rm i}c_{\rm p}\omega_2|x-z|+1)\left(\frac{3}{|x-y|^2}({\rm
i}c_{\rm s}\omega_1|x-y|-1)+c_{\rm s}^2\omega_1^2\right)e^{{\rm i}(c_{\rm
s}\omega_1|x-y|-c_{\rm p}\omega_2|x-z|)}\\
&-({\rm i}c_{\rm s}\omega_2|x-z|+1)\left(\frac{3}{|x-y|^2}({\rm i}c_{\rm
p}\omega_1|x-y|-1)+c_{\rm p}^2\omega_1^2\right)e^{{\rm i}(c_{\rm
p}\omega_1|x-y|-c_{\rm s}\omega_2|x-z|)}
\bigg]\\
&\times\frac{(x_i-y_i)^2}{|x-y|^3|x-z|^3}\mathbb E(f_i(y)f_i(z)){\rm d}y{\rm
d}z,
\end{align*}
\begin{align*}
&\mathbb
E(u_{i3}(x,\omega_1)\overline{u_{i3}(x,\omega_2)})=\frac{1}{
16\pi^2\omega_1^2\omega_2^2}
\int_{\mathbb R^6}\bigg[\bigg(\frac{3}{|x-y|^2}({\rm
i}c_{\rm s}\omega_1|x-y|-1)+c_{\rm s}^2\omega_1^2\bigg)\times\\
&\qquad\left(\frac{3}{|x-z|^2}(-{\rm
i}c_{\rm s}\omega_2|x-z|-1)+c_{\rm s}^2\omega_2^2\right)e^{{\rm
i}(c_{\rm s}\omega_1|x-y|-c_{\rm s}\omega_2|x-z|)}\\
&+\left(\frac{3}{|x-y|^2}({\rm i}c_{\rm p}\omega_1|x-y|-1)+c_{\rm
p}^2\omega_1^2\right)\times\\
&\qquad\left(\frac{3}{|x-z|^2}(-{\rm i}c_{\rm p}\omega_2|x-z|-1)+c_{\rm
p}^2\omega_2^2\right)e^{{\rm i}(c_{\rm p}\omega_1|x-y|-c_{\rm
p}\omega_2|x-z|)}\\
&-\bigg(\frac{3}{|x-y|^2}({\rm
i}c_{\rm s}\omega_1|x-y|-1)+c_{\rm s}^2\omega_1^2\bigg)\times\\
&\qquad\left(\frac{3}{|x-z|^2}(-{\rm i}c_{\rm p}\omega_2|x-z|-1)+c_{\rm
p}^2\omega_2^2\right)e^{{\rm i}(c_{\rm s}\omega_1|x-y|-c_{\rm
p}\omega_2|x-z|)}\\
&-\bigg(\frac{3}{|x-y|^2}({\rm i}c_{\rm p}\omega_1|x-y|-1)+c_{\rm
p}^2\omega_1^2\bigg)\times\\
&\qquad\left(\frac{3}{|x-z|^2}(-{\rm
i}c_{\rm s}\omega_2|x-z|-1)+c_{\rm s}^2\omega_2^2\right)e^{{\rm i}(c_{\rm
p}\omega_1|x-y|-c_{\rm s}\omega_2|x-z|)}\bigg]\\
&\times\frac{(x_i-y_i)(x_i-z_i)}{|x-y|^3|x-z|^3}\sum_{j=1}
^3(x_j-y_j)(x_j-z_j)\mathbb E(f_j(y)f_j(z)){\rm d}y{\rm d}z.
\end{align*}

Observe the above expressions, it is easy to see that $\mathbb
E(\bm{u}(x,\omega_1)\cdot\overline{\bm{u}(x,\omega_2)})$ is a linear combination
of $I(x,\omega_1,\omega_2)$ which is defined by (\ref{b11}). A direct
application of Lemma \ref{lemma2b} leads to the following lemma which plays
an important role in the proof of the main results.

\begin{lemma}\label{lemma1e}
For $\omega_1\geq 1, \omega_2\geq 1,$ the estimates
\begin{align*}
|\mathbb E(\bm{u}(x,\omega_1)\cdot\overline{\bm{u}(x,\omega_2)})|&\leq
c_n(1+|\omega_1-\omega_2|)^{-n}(\omega_1+\omega_2)^{-m},\\
|\mathbb E(\bm{u}(x,\omega_1)\cdot \bm{u}(x,\omega_2))|&\leq
c_n(\omega_1+\omega_2)^{-n}(1+|\omega_1-\omega_2|)^{-m}
\end{align*}
holds uniformly for $x\in U$, where $n\in\mathbb N$ is arbitrary and $c_n\geq 0$
is a constant depending only on $n$.
\end{lemma}

Now we are ready to compute the order of $\mathbb E(|u(x,\omega)|^2)$. Let
$\omega_1=\omega_2=\omega$ in 
$\mathbb E(u_{ij}(x,\omega_1)\overline{u_{ik}(x,\omega_1)})$ for $i,j,k=1,2,3$,
a direct application of Lemma \ref{lemma4b} gives that 
\begin{align*}
&\mathbb E(u_{i1}(x,\omega)\overline{u_{i2}(x,\omega)})=O(\omega^{-(m+1)}),\quad
\mathbb E(u_{i2}(x,\omega)\overline{u_{i1}(x,\omega)})=O(\omega^{-(m+1)}),\\
&\mathbb E(u_{i2}(x,\omega)\overline{u_{i2}(x,\omega)})=O(\omega^{-(m+1)}),\quad
\mathbb E(u_{i2}(x,\omega)\overline{u_{i3}(x,\omega)})=O(\omega^{-(m+1)}),\\
&\mathbb E(u_{i3}(x,\omega)\overline{u_{i2}(x,\omega)})=O(\omega^{-(m+1)}),
\end{align*} 
and
\begin{align*}
\mathbb
E(u_{i1}(x,\omega)\overline{u_{i1}(x,\omega)})&=N_{1i}^{(3)}(x)\omega^{-m }
+O(\omega^{-(m+1)}),\\
\mathbb E(u_{i1}(x,\omega)\overline{u_{i3}(x,\omega)})&=N_{2i}^{(3)}(x)(x,
\omega)\omega^{-m}+O(\omega^{-(m+1)}),\\
\mathbb E(u_{i3}(x,\omega)\overline{u_{i1}(x,\omega)})&=N_{3i}^{(3)}(x,
\omega)\omega^{-m}+O(\omega^{-(m+1)}),\\
\mathbb
E(u_{i3}(x,\omega)\overline{u_{i3}(x,\omega)})&=\sum_{j=1}^3N^{(3)}_{4i,j
}(x, \omega)\omega^{-m}+O(\omega^{-(m+1)}),
\end{align*}
where
\begin{align*}
N_{1i}^{(3)}(x)&= b_{1}\int_{\mathbb
R^3}\frac{1}{|x-y|^2}\phi(y){\rm d}y,\\
N_{2i}^{(3)}(x,\omega)&=\int_{\mathbb R^3}(b_{2}e^{{\rm
i}(c_{\rm s}-c_{\rm
p})|x-y|\omega}-b_{1})\frac{(x_i-y_i)^2}{|x-y|^4}\phi(y){\rm d}y,\\
N_{3i}^{(3)}(x,\omega)&=\int_{\mathbb R^3}(b_{2}e^{{\rm
i}(c_{\rm p}-c_{\rm
s})|x-y|\omega}-b_{1})\frac{(x_i-y_i)^2}{|x-y|^4}\phi(y){\rm d}y,\\
N^{(3)}_{4i,j}(x,\omega)&=\int_{\mathbb
R^3}(b_{3}-2b_{2}\cos((c_{\rm s}-c_{\rm
p})|x-y|\omega))\frac{(x_i-y_i)^2(x_j-y_j)^2}{
|x-y|^6}\phi(y){\rm d}y.
\end{align*}
Here $b_{1}, b_{2}, b_{3}$ are positive constants given by
\begin{eqnarray*}
b_{1}=\frac{1}{128\pi^2}c_{\rm s}^{4-m},\quad
b_{2} = \frac{(c_{\rm s}c_{\rm p})^2}{128\pi^2}\left(\frac{2}{c_{\rm s}+c_{\rm
p}}\right)^{m},\quad b_{3}=\frac{1}{128\pi^2}(c_{\rm s}^{4-m}+c_{\rm p}^{4-m}).
\end{eqnarray*}
Therefore
\begin{eqnarray}\label{e18}
\mathbb E(|\bm{u}(x,\omega)|^2)=T^{(3)}_{\rm
E}(x)\omega^{-m}+O(\omega^{-(m+1)}),
\end{eqnarray}
where 
\begin{eqnarray}\label{e19}
T^{(3)}_{\rm E}(x)=(b_3-b_1)\int_{\mathbb R^3}\frac{1}{|x-y|^2}\phi(y){\rm d}y.
\end{eqnarray}
Now we are ready to present the main result of elastic waves for the
three-dimensional case. 

\begin{theorem}
Let the external source $\bm{f}$ be a microlocally isotropic Gaussian random
vector field which satisfies Assumption C. Then for all $x\in U$, it holds
almost surely that
\begin{eqnarray*}
 \lim_{Q\rightarrow\infty}\frac{1}{Q-1}\int_1^Q\omega^{m}|\bm{u}(x,\omega)|^2{
\rm d}\omega=T^{(3)}_{\rm E}(x).\\
\end{eqnarray*}
Moreover, the scattering data $T^{(3)}_{\rm E}(x)$, for $x\in U$, uniquely
determine the micro-correlation strength $\phi$ through the linear relation
(\ref{e19}).
\end{theorem}

\begin{proof}
Using Lemma \ref{lemma_o} and (\ref{e18}), we may follow the same proof as that
for the two-dimensional case. The details are omitted here for brevity.
\end{proof}

\section{Conclusion}

We have studied an inverse source scattering problem for the two- and
three-dimensional Helmholtz equation and Navier equation. The source 
is assumed to be a generalized Gaussian random function whose covariance
operator is a classical pseudo-differential operator. By an exact expression of
the random wave field and microlocal analysis, we derive a linear integral
equation which connects the principle symbol of the covariance operator and
the amplitude of the scattering data generated from a single realization of the
random source. Based on this relationship, we obtain the uniqueness for the
recovery of the principle symbol of the random source for the Helmholtz
and Navier equations. A possible continuation of this work is to investigate the
uniqueness for Maxwell's equations with a distributional source. Since the
Green tensor has a higher singularity for the Maxwell equations, a new
technique must be developed. Another interesting direction is to study the
uniqueness for the inverse random source problems in inhomogeneous media,
where the analytical Green function or tensor is not available any more and
the present method may not be directly applicable. We hope to be able to report
the progress on these problems in the future.

\end{document}